\pdfoutput=1

\documentclass[reqno,a4paper]{amsart}
\usepackage{microtype}
\usepackage{color}
\usepackage{amssymb}
\usepackage{stmaryrd}
\usepackage[utf8]{inputenc}
\usepackage{a4wide}
\usepackage{paralist}
\usepackage{enumitem}
\usepackage[all]{xy}
\usepackage{ifthen}
\usepackage{graphicx}
\usepackage{slashed}
\usepackage{comment}

\usepackage[pdftex,breaklinks,colorlinks,filecolor=blue,linkcolor=blue,citecolor=blue,urlcolor=blue,hypertexnames=true,plainpages=false]{hyperref}

\newcommand\C{\mathbb{C}}
\newcommand\R{\mathbb{R}}
\renewcommand\SS{\mathbb{S}}
\newcommand\Q{\mathbb{Q}}
\newcommand\Z{\mathbb{Z}}

\newcommand\FF{\mathcal{F}^+}
\newcommand\NN{\mathcal{N}^+}

\newcommand{\Aut}{\textup{Aut}}
\newcommand{\Ad}{\textup{ad}}
\newcommand{\Br}{\textup{Br}}
\newcommand{\coker}{\textup{cok}}
\newcommand{\Id}{\textup{Id}}
\newcommand{\im}{\textup{im}}
\newcommand{\ls}{\textup{ls}}
\newcommand{\Sym}{\mathop{\mathrm{Sym}}\nolimits}

\newcommand{\Image}{\mathop{\mathrm{Im}}\nolimits}
\newcommand{\Tors}{\mathop{\mathrm{Tors}}\nolimits}

\newcommand{\interior}{\mathop{\mathrm{int}}\nolimits}
\newcommand{\fr}{\mathrm{fr}}
\newcommand{\pr}{\mathrm{pr}}
\newcommand{\dd}{\ul d}

\newcommand{\an}[1]{\langle{#1}\rangle}

\newcommand{\ol}{\overline}
\newcommand{\ul}{\underline}
\newcommand{\xra}{\xrightarrow}
\newcommand{\Sigmaex}{\Sigma_{\mathrm ex}^8}
\newcommand{\Th}{\textup{Th}}
\newcommand{\PT}{\textup{PT}}

\newcommand{\Ker}{\textup{Ker}}
\newcommand{\Map}{\textup{Map}}
\renewcommand{\Im}{\textup{Im}}

\newcommand{\qsnmo}{QS^0\!/SO}

\theoremstyle{plain}
\newtheorem{theorem}{Theorem}[section]
\newtheorem{thm}[theorem]{Theorem}
\newtheorem{lemma}[theorem]{Lemma}
\newtheorem{lem}[theorem]{Lemma}
\newtheorem{corollary}[theorem]{Corollary}
\newtheorem{proposition}[theorem]{Proposition}
\newtheorem{prop}[theorem]{Proposition}
\newtheorem{conjecture}[theorem]{Conjecture}
\newtheorem{conj}[theorem]{Conjecture}

\theoremstyle{definition}
\newtheorem{definition}[theorem]{Definition}
\newtheorem{defin}[theorem]{Definition}
\newtheorem{example}[theorem]{Example}
\newtheorem{notation}[theorem]{Notation}
\newtheorem{const}[theorem]{Construction}

\theoremstyle{remark}
\newtheorem{remark}[theorem]{Remark}
\newtheorem{rem}[theorem]{Remark}

\advance \topmargin by -\headheight
\advance \topmargin by -\headsep
     
\evensidemargin \oddsidemargin
\title{The smooth classification of $4$-dimensional complete intersections}
\author[D. Crowley]{Diarmuid Crowley}
\address{School of Mathematics and Statistics,
University of Melbourne,
Parkville, VIC, 3010, Australia}
\email{dcrowley@unimelb.edu.au}
\author[Cs. Nagy]{Csaba Nagy}
\address{School of Mathematics and Statistics, University of Glasgow, United Kingdom}
\email{csaba.nagy@glasgow.ac.uk}
\date{\today}

\begin{document}

\begin{abstract}
We prove the ``Sullivan Conjecture'' on the classification of $4$-dimensional complete intersections up to diffeomorphism. Here an $n$-dimensional complete intersection is a smooth complex variety formed by the transverse intersection of $k$ hypersurfaces in $\C P^{n+k}$. 

Previously Kreck and Traving proved the $4$-dimensional Sullivan Conjecture when $64$ divides the total degree (the product of the degrees of the defining hypersurfaces) and Fang and Klaus proved that the conjecture holds up to the action of the group of homotopy $8$-spheres $\Theta_8 \cong \Z/2$.

Our proof involves several new ideas, including the use of the Hambleton-Madsen theory of degree-$d$ normal maps, which provide a fresh perspective on the Sullivan Conjecture in all dimensions. This leads to an unexpected connection between the Segal Conjecture for $S^1$ and the Sullivan Conjecture.
\end{abstract}

\maketitle

\section{Introduction} 
\label{s:intro}

\subsection{Complete intersections and the Sullivan Conjecture} 
\label{ss:SC}

A {\em complete intersection} $X_n(\dd) \subset \C P^{n+k}$ is the transverse intersection of $k$ complex hypersurfaces of degrees 
$\dd = \{d_1, \ldots, d_k\}$.  
We regard $X_n(\dd)$ as an oriented smooth manifold of real dimension $2n$ and consider the problem of classifying complete intersections up to orientation preserving diffeomorphism. Hence throughout this paper, all manifolds are oriented and all diffeomorphisms and homeomorphisms are assumed to preserve orientations. 
By an observation of Thom, the diffeomorphism type of $X_n(\dd)$ depends only on the {\em multidegree} $\dd$. 

The main conjecture organising the classification of complete intersections for $n \geq 3$ is the ``Sullivan Conjecture''. The statement of the conjecture relies on the following fact (see Remark \ref{rem:char}): There are integers $p_i(n,\dd)$ such that the Pontryagin classes of $X_n(\dd)$ satisfy $p_i(X_n(\dd)) = p_i(n,\dd) x^{2i}$, where $x \in H^2(X_n(\dd))$ is the pullback of a generator of $H^2(\C P^{n+k})$. 
Let $d := d_1 \cdots d_k$ denote the 
{\em total degree of $X_n(\dd)$,} which is the product of the individual degrees. 

\begin{defin}
The {\em Sullivan data} associated to the complete intersection $X_n(\dd)$ is the tuple
\[
SD_n(\dd) := \bigl( d, (p_i(n,\dd))_{i=1}^{\lfloor n/2 \rfloor}, \chi(X_n(\dd)) \bigr) \in \Z^+ \times \Z^{\lfloor n/2 \rfloor} \times \Z \text{\,,}
\] 
which consists of the total degree $d$, the Pontryagin classes of $X_n(\dd)$ regarded as integers and the Euler characteristic of $X_n(\dd)$. 
For a fixed $n$, each of these integers is a polynomial function of the individual degrees; see Section \ref{ss:SD}.
\end{defin}

\begin{conj}[The Sullivan Conjecture]
Suppose that $n \geq 3$ and $X_n(\dd)$ and $X_n(\dd')$ are complete intersections. 
If $SD_n(\dd) = SD_n(\dd')$, then $X_n(\dd)$ is diffeomorphic to $X_n(\dd')$.
\end{conj}

The main result of this paper is that the Sullivan Conjecture holds in complex dimension $4$.

\begin{theorem} \label{thm:main}
Suppose that $X_4(\dd)$ and $X_4(\dd')$ are complete intersections with $SD_4(\dd) = SD_4(\dd')$. Then $X_4(\dd)$ is diffeomorphic to $X_4(\dd')$.
\end{theorem}

\subsection{Background and an application} \label{ss:background}

We first list some existing results about the Sullivan Conjecture, its analogue in dimensions $n < 3$
and its converse.

When $n = 1$, $X_1(\dd)$ is an oriented surface and the classification is classical (in particular the Sullivan Conjecture holds but its converse does not).

When $n = 2$, $X_2(\dd)$ is a simply-connected smooth manifold and smooth classification results are currently out of reach.  However, the topological classification can be deduced from results of Freedman \cite{Fr}: Two complete intersections are homeomorphic if and only if they have the same Pontryagin class $p_1$ and the same Euler characteristic. The converse fails, because the total degree is not even a diffeomorphism invariant (e.g.\ $X_2(4)$, $X_2(3,2)$ and $X_2(2,2,2)$ are all K3-surfaces.)

When $n \geq 3$, the converse of the Sullivan Conjecture holds; see Proposition \ref{prop:SC-conv}.

If $n = 3$ the Sullivan Conjecture follows from classification theorems of Wall or Jupp \cite{Wa1, J}.

If $n = 4$, Fang and Klaus \cite[Remark 2]{F-K} proved that the Sullivan Conjecture holds up to connected sum with homotopy $8$-spheres: 

\begin{thm}[Fang and Klaus \cite{F-K}] \label{thm:FK}
Suppose that $X_4(\dd)$ and $X_4(\dd')$ are complete intersections with $SD_4(\dd) = SD_4(\dd')$. Then there is a homotopy $8$-sphere $\Sigma$ such that $X_4(\dd')$ and $X_4(\dd) \sharp \Sigma$ are diffeomorphic.
\end{thm}

If $5 \leq n \leq 7$, then Fang and Wang \cite{F-W} proved that the Sullivan Conjecture holds up to homeomorphism. 

For $n \geq 3$, Kreck and Traving proved the following general statement.
Let $\nu_p(d)$ be the largest integer such that $p^{\nu_p(d)} \mid d$.  
If $SD_n(\dd) = SD_n(\dd')$ and $\nu_p(d) \geq \frac{2n+1}{2(p{-}1)}+1$ for every prime $p$ with $p(p{-}1) \leq n{+}1$, then $X_n(\dd)$ and $X_n(\dd')$ are diffeomorphic \cite[Theorem A]{Kr}.
If $n = 4$, then the condition says that $64 \mid d$.

A motivation for the diffeomorphism classification of complete intersections is the result of Libgober and Wood \cite[Corollary 8.3]{L-W}, which says that if $n \geq 3$ and 
diffeomorphic complete intersections have different multidegrees, 
then their complex structures lie in different connected components of the moduli space of complex structures on the underlying smooth manifold.  
Here and in general, multidegrees are regarded as equal 
if one can be obtained from the other by adding or removing $1$s,
because then the corresponding complete intersections have a common representative.  
Libgober and Wood used this result to show that for all odd $n \geq 3$ there are complete intersections
having a complex moduli space with arbitrarily many connected components.
Their proof relied on a counting argument, valid in all dimensions, which shows that the sets $\{ \dd' \mid SD_n(\dd') = SD_n(\dd) \}$ of multidegrees with the same Sullivan data can be arbitrarily large.

In future work we give an effective algorithm for finding pairs of multidegrees with the same Sullivan data. The Sullivan Conjecture then allows us to construct explicit examples of complete intersections in different components of the complex moduli space and we obtain the following application of Theorem \ref{thm:main}.

\begin{example} \label{ex:odd}
The complete intersections $X_4(3^{(150)},7^{(89)},9^{(65)},15,25^{(130)})$ and $X_4(5^{(261)},21^{(89)},27^{(64)})$ (where $3^{(150)}$ stands for $150$ copies of $3$, etc.) 
are diffeomorphic by Theorem \ref{thm:main} and the formulae in Section \ref{ss:SD}.
Hence the corresponding complex structures lie in different 
components of the complex moduli space.
\end{example}

\subsection{The outline of the proof of Theorem \ref{thm:main}} \label{ss:Proof} 
If $SD_4(\dd) = SD_4(\dd')$, then by Theorem \ref{thm:FK} of Fang and Klaus
there is a diffeomorphism $X_4(\dd) \to X_4(\dd') \sharp \Sigma$ 
for some homotopy sphere $\Sigma$.
The group of homotopy $8$-spheres, $\Theta_8 \cong \Z/2$, is known by Kervaire and Milnor \cite{K-M} 
and so we let $\Sigmaex$ denote the unique diffeomorphism class of the exotic $8$-sphere and introduce
the following terminology.

\begin{definition} \label{def:strongly_flexible_etc}
\quad 
\begin{compactitem} 
\item An $8$-manifold $M$ is {\em $\Theta$-rigid} if $M \sharp \Sigmaex$ is diffeomorphic to $M$.
\item An $8$-manifold $M$ is {\em $\Theta$-flexible} if $M \sharp \Sigmaex$ is not diffeomorphic to $M$.
\item A complete intersection $X_4(\dd)$ is {\em strongly $\Theta$-flexible} if $X_4(\dd) \sharp \Sigmaex$ is not diffeomorphic to a 
complete intersection.
\end{compactitem}
\end{definition}

As our proof of Theorem \ref{thm:main} involves treating several cases separately, we shall say that {\em the Sullivan Conjecture holds for a fixed complete intersection $X_n(\dd)$} if, for every $\dd'$, $SD_n(\dd) = SD_n(\dd')$ implies that $X_n(\dd')$ is diffeomorphic to $X_n(\dd)$. 
By Theorem~\ref{thm:FK} and Remark~\ref{rem:SD-equal}, the Sullivan Conjecture holds for $X_4(\dd)$ if and only if $X_4(\dd)$ is either $\Theta$-rigid or strongly $\Theta$-flexible.
To prove the $4$-dimensional Sullivan Conjecture we consider four cases, which are indexed by 
the Wu classes of $X_4(\dd)$ and the parity of the total degree:
\[
\begin{array}{c|c|c|c|c}
\vphantom{\gtreqqless} v_2(X_4(\dd)) & v_4(X_4(\dd)) & d~\mathrm{mod}~2 & \text{$\Theta$-rigidity} & \text{Treated in} \\
\hline \hline
\vphantom{\bar{\gtreqqless}} 0 & - & - & \text{Strongly $\Theta$-flexible} & \text{Theorem \ref{thm:main-p1}} \\
\hline
\vphantom{\bar{\gtreqqless}} 1 & 0 & - & \text{$\Theta$-rigid} & \text{Theorem \ref{thm:main-p2}} \\
\hline
\vphantom{\bar{\gtreqqless}} 1 & 1 & 0 & \text{Unknown in general} & \text{Theorem \ref{thm:main-p3}\,(a)} \\
\hline
\vphantom{\bar{\gtreqqless}} 1 & 1 & 1 & \text{Unknown in general} & \text{Theorem \ref{thm:main-p3}\,(b)}
\end{array}
\]
Here
$v_i(X_4(\dd)) \in H^i(X_4(\dd); \Z/2)$ is the $i^{\text{th}}$ Wu class of $X_n(\dd)$, which can be regarded as an element of $\Z/2$ by Remark \ref{rem:char}, a $``-"$ indicates the value of the invariant is not relevant in that case and in the cases when the $\Theta$-rigidity of $X_4(\dd)$ is unknown, we conjecture that it depends on $p_1(4,\dd)$ mod $8$; see Conjecture \ref{conj:inertia}. 

Now we discuss the proof in each of the four cases.

For a spin complete intersection $X_4(\dd)$ (equivalently, by Proposition \ref{prop:SW-classes}, when $v_2(X_4(\dd))=0$) we find a diffeomorphism invariant property of complete intersections not shared by $X_4(\dd) \sharp \Sigmaex$; see Section \ref{s:spin}. 
Namely, if $S(X, \alpha)$ denotes the total space of the circle bundle over a space $X$ with first Chern class $\alpha$, then $S(X_n(\dd), \pm x)$ admits a framing making it a null-cobordant framed $(2n{+}1)$-manifold (for any $X_n(\dd)$), whereas $S(X_4(\dd) \sharp \Sigmaex; \pm x)$ does not (for a spin $X_4(\dd)$). 
Hence (see Theorem \ref{thm:spin}) we have 

\begin{theorem} \label{thm:main-p1}
If $X_4(\dd)$ is spin, then $X_4(\dd)$ is strongly $\Theta$-flexible. In particular, the Sullivan Conjecture holds for $X_4(\dd)$.
\end{theorem}

\begin{remark} \label{rem:SC5}
For the $5$-dimensional Sullivan Conjecture, the group of homotopy $10$-spheres 
$\Theta_{10} \cong \Z/2 \times \Z/3$ will play a central role.
We believe that ``transfer'' arguments similar to those we use in the $4$-dimensional spin case
will control the $(\Z/3)$-factor of $\Theta_{10}$.  The $(\Z/2)$-factor of
$\Theta_{10}$ is detected by the $\alpha$-invariant and Baraglia \cite{Ba} has recently computed
the $\alpha$-invariant of spin complete intersections, 
verifying its values are consistent with the Sullivan Conjecture.
We anticipate that these ideas will lead to a proof of the $5$-dimensional Sullivan Conjecture in
future work. 
\end{remark}

In the non-spin cases we apply Kreck's modified surgery theory \cite{Kr}. 
Consider $B_n := {\C}P^\infty \times BO\an{n{+}1}$ with the stable bundle 
$\xi_n(\dd) \times \gamma_{BO\an{n{+}1}}$ over it; for the notation see Definition \ref{def:xi-nd} and Section \ref{ss:surgery}.  
Recall from \cite[Section 8]{Kr} that a {\em normal $(n{-}1)$-smoothing in $(B_n, \xi_n(\dd) \times \gamma_{BO\an{n{+}1}})$} is a pair $(f, \bar f)$ where $f \colon M \to B_n$ is an $n$-connected map from a closed smooth manifold $M$ and $\bar f \colon \nu_M \to \xi_n(\dd) \times \gamma_{BO\an{n{+}1}}$ is a map of stable bundles from the normal bundle of $M$, which covers $f$.
Recall also that the normal $(n{-}1)$-type of $X_n(\dd)$ is $(B_n, \xi_n(\dd) \times \gamma_{BO\an{n{+}1}})$, in particular $X_n(\dd)$ admits a normal $(n{-}1)$-smoothing in $(B_n, \xi_n(\dd) \times \gamma_{BO\an{n{+}1}})$.
In this setting \cite[Proposition 10]{Kr} reduces the Sullivan Conjecture to a statement about bordism classes over $(B_n, \xi_n(\dd) \times \gamma_{BO\an{n{+}1}})$. For our purposes, it is useful to state an altered version of \cite[Proposition 10]{Kr}, which compares a complete intersection $X_n(\dd)$ to a somewhat more general closed $2n$-manifold $X'$. The proof of Proposition \ref{prop:classification_general} is identical to the proof of the sufficient condition of \cite[Proposition 10]{Kr}.

\begin{proposition} \label{prop:classification_general}
Let $n \geq 3$, $X_n(\dd)$ be a complete intersection and $X'$ a closed $2n$-manifold such that 
$\chi(X_n(\dd)) = \chi(X')$ and $X_n(\dd)$ and $X'$ admit bordant normal $(n{-}1)$-smoothings over 
$(B_n, \xi_n(\dd) \times \gamma_{BO\an{n{+}1}})$.
If $\dd \neq \{1\}, \{2\}$ or $\{2, 2\}$, then $X_n(\dd)$ and $X'$ are diffeomorphic.
\hfill $\qed$
\end{proposition}

\begin{rem} \label{rem:cl-gen}
In fact, the assumption that $\dd \neq \{1\}$ can be removed by applying \cite[Proposition 8\,(i)]{Kr}. We do not know the situation for $\dd = \{2\}, \{2, 2\}$. However, for all three of these exceptional multidegrees $\dd$, it is elementary that $SD_n(\dd) = SD_n(\dd')$ implies $\dd = \dd'$ and so the Sullivan Conjecture holds for these complete intersections.
\end{rem}

The main challenge when applying Proposition \ref{prop:classification_general} is showing that the bordism condition holds; see the discussion in Section \ref{ss:surgery}.  
Note that the bordism group of $8$-manifolds over
$(B_4, \xi_4(\dd) \times \gamma_{BO\an{5}})$ is canonically 
isomorphic to the twisted string bordism group $\Omega^{O\an{7}}_8({\C}P^\infty; \xi_4(\dd))$,
since  $BO\an{5} = BO\an{8} = B(O\an{7})$.

In the case of a non-spin complete intersection $X_4(\dd)$ with $v_4(X_4(\dd)) = 0$, we will use Proposition \ref{prop:classification_general} to compare $X_4(\dd)$ with $X' = X_4(\dd) \sharp \Sigmaex$. 
They admit normal $3$-smoothings over $(B_4, \xi_4(\dd) \times \gamma_{BO\an{8}})$, whose bordism classes differ by the image of $\Sigmaex$ under the canonical homomorphism $i_0 \colon \Theta_8 \to \Omega^{O\an{7}}_8({\C}P^\infty; \xi_4(\dd))$.
The map $i_0$ factors through $\Tors \Omega^{O\an{7}}_8({\C}P^1; \xi_4(\dd)\big|_{\C P^1})$ 
and (see Lemma \ref{lem:Omega_8(CP1)}) we prove

\begin{prop} \label{prop:Z/4}
If $X_4(\dd)$ is non-spin, then $\Tors \Omega^{O\an{7}}_8({\C}P^1; \xi_4(\dd)\big|_{\C P^1}) \cong \Z/4$.
\end{prop}

When $v_4(X_4(\dd)) = 0$, we combine Proposition \ref{prop:Z/4} with the computations of \cite[Section 2.2]{F-K} to show that the map $\Theta_8 \to \Omega^{O\an{7}}_8(\C P^\infty; \xi_4(\dd))$ vanishes (Proposition \ref{prop:i_CP1}), which gives (see Theorem \ref{thm:rigid}) 

\begin{theorem} \label{thm:main-p2}
Suppose that $X_4(\dd)$ is a non-spin complete intersection with $v_4(X_4(\dd)) = 0$. If $\dd \neq \{2, 2\}$, then $X_4(\dd)$ is $\Theta$-rigid and so the Sullivan Conjecture holds for $X_4(\dd)$.
\end{theorem}

\begin{rem} \label{rem:22}
In fact $X_4(2,2)$ is $\Theta$-rigid too. This follows from from results in the second author's PhD thesis \cite[Theorem 4.6.1]{N} but will not be proven here.
\end{rem}

If $X_4(\dd)$ is non-spin, $v_4(X_4(\dd)) = 1$ and the total degree $d$ is even, then $16 \!  \mid \! d$ (see Remark \ref{rem:p-mod-4}). 
We add Proposition \ref{prop:Z/4} to the Adams filtration argument of Kreck and Traving \cite[Section 8]{Kr} and the calculations of 
\cite[Section 2.4]{F-K} to prove (see Proposition \ref{prop:K-T_improved}) Part (a) of the following theorem.

\begin{theorem} \label{thm:main-p3}
Let $X_4(\dd)$ and $X_4(\dd')$ be non-spin complete intersections with $SD_4(\dd) = SD_4(\dd')$ and 
suppose that either 
\begin{compactenum}[(a)]
\item $v_4(X_4(\dd)) \neq 0$ and the total degree $d$ is even, or
\item the total degree $d$ is odd.
\end{compactenum}
Then $X_4(\dd)$ and $X_4(\dd')$ admit bordant normal $3$-smoothings over 
$(B_4, \xi_4(\dd) \times \gamma_{BO\an{8}})$. Consequently, $X_4(\dd)$ and $X_4(\dd')$ 
are diffeomorphic and the Sullivan Conjecture holds for $X_4(\dd)$.
\end{theorem}

Note that the cases discussed so far 
(i.e.\ those prior to Theorem \ref{thm:main-p3}\,(b)), have a significant overlap 
with, but are not implied by, the theorem of Kreck and Traving \cite[Theorem A]{Kr}. 
However, the case of odd total degree covered in Theorem \ref{thm:main-p3}\,(b) is completely new.
Note also that the total degree can be odd only if $v_2(X_4(\dd)) \neq 0$ and $v_4(X_4(\dd)) \neq 0$; see Proposition \ref{prop:SW-classes}.

To prove Theorem \ref{thm:main-p3}\,(b) we use the Hambleton-Madsen theory of degree-$d$ normal maps \cite{H-M}.  A complete intersection $X_n(\dd)$ (with a canonical choice of normal data) represents an element in the set $\NN_d({\C}P^n)$ of normal bordism classes of degree-$d$ normal maps over ${\C}P^n$.
As explained in Section~\ref{ss:drnm}, an oriented version of the Hambleton-Madsen theory gives a bijective normal invariant map
\[ \eta \colon \NN_d(\C P^n) \equiv [\C P^n, (\qsnmo)_d],\]
which is the usual normal invariant in the familiar case when $d = 1$ and where $(\qsnmo)_d$ is 
the oriented version of the classifying space for isomorphism classes of stable fibrewise degree-$d$ maps between sphere bundles of vector bundles, which was identified by Brumfiel and Madsen \cite[\S 4]{B-M}.
We establish a relationship between certain 
``relative divisors" of a vector bundle and degree-$d$ normal maps over the vector bundle
(Lemma \ref{lem:div}) and then use this to give a formula for the 
canonical degree-$d$ normal invariant of $X_n(\dd)$ (Theorem \ref{thm:eta_of_X}). 

The surgery argument of Proposition \ref{prop:classification_general} also works if we have bordant representatives in $\NN_d({\C}P^n)$ (Lemma \ref{lem:SC_via_cni0}). This and the formula of Theorem \ref{thm:eta_of_X} leads to a new perspective on the stable homotopy-theoretic input needed to prove the Sullivan Conjecture (see Theorem \ref{thm:SC_via_cni}). This new perspective allows us to prove the $4$-dimensional Sullivan Conjecture when the total degree is odd and we anticipate that it will lead to other new results in higher dimensions; e.g.\ see Remark \ref{rem:SC_prime_to_d}.

Notice that Fang and Klaus (Theorem \ref{thm:FK}) reduced the $4$-dimensional Sullivan Conjecture to a $2$-local problem. When $d$ is odd, the work of Brumfiel and Madsen \cite{B-M} shows that there is 
an equivalence of $2$-localisations
$((\qsnmo)_d)_{(2)} \simeq (G/O)_{(2)}$, where $G/O$ is the familiar classifying space
from classical surgery theory \cite{Br1, Wa2}.  We can then exploit Sullivan's $2$-local splitting (see \cite[Theorem 5.18]{M-M}),
$$ (G/O)_{(2)} \simeq (BSO)_{(2)} \times \coker J_{(2)},$$
where $\coker J_{(2)}$ is a 
$2$-local space whose homotopy groups are certain large summands of the
$2$-primary component of the cokernel of the $J$-homomorphism (see \cite[Definition 5.16]{M-M}).
It follows that we have a sequence of maps
\[
[\C P^n, (\qsnmo)_d] \to [\C P^n, ((\qsnmo)_d)_{(2)}]  \xra{~\equiv~} [\C P^n, (G/O)_{(2)}] \xra{~\equiv~} [\C P^n, (BSO)_{(2)}] \times [\C P^n, \coker J_{(2)}].
\]
The formula for the degree-$d$ normal invariant of $X_n(\dd)$ shows that it is the restriction of a map ${\C}P^\infty \to (\qsnmo)_d$.  Now the proof of a theorem of Feshbach \cite[Theorem 6]{Fe1}, which is based on 
the Segal Conjecture for the Lie group $S^1$,
implies that any map $\C P^\infty \to \coker J_{(2)}$ is null-homotopic and this is enough to prove that the 
$[\C P^n, \coker J_{(2)}]$ factor of the $2$-localised normal invariant is trivial (Corollary \ref{cor:mod_p_normal_inv}). The $[\C P^n, (BSO)_{(2)}]$ factor is controlled by the Sullivan data, hence in dimension $4$ the degree-$d$ normal invariant is completely determined by the Sullivan data (Theorem \ref{thm:bordism_over_CP4}). The $4$-dimensional Sullivan Conjecture for complete intersections with odd total degree follows (Theorem \ref{thm:SC_d=odd}).

\subsection{Inertia groups of $4$-dimensional complete intersections} 
\label{ss:inertia}
Recall that the inertia group of a closed connected 
$m$-manifold $M$ is the subgroup
\[
I(M) := \bigl\{ \Sigma \in \Theta_m \mid \text{$M$ and $M \sharp \Sigma$ are diffeomorphic} \bigr\} \subseteq \Theta_m 
\]
of the group of homotopy $m$-spheres $\Theta_m$ \cite{K-M}.
For example, an $8$-manifold $M$ is $\Theta$-rigid if and only if $I(M) = \Theta_8$.
The results in Section \ref{ss:Proof} determine the inertia groups of a
$4$-dimensional complete intersection when $X_4(\dd)$ is spin, or when $X_4(\dd)$ is
non-spin and $v_4(X_4(\dd)) = 0$.
When $X_4(\dd)$ is non-spin and $v_4(X_4(\dd)) \neq 0$, 
we have $p_1(4,\dd) \equiv 3$ mod $4$ 
(see Proposition \ref{prop:SW-classes} and the calculations in Section \ref{ss:SD})
and we offer the third and fourth rows of the table in the following conjecture.

\begin{conjecture} \label{conj:inertia}
The inertia groups $I(X_4(\dd)) \subseteq \Theta_8 \cong \Z/2$ of $4$-dimensional complete intersections are given by the table below:
\[
\begin{array}{c|c|c|c}
\vphantom{\gtreqqless} v_2(X_4(\dd)) & v_4(X_4(\dd)) & p_1(4,\dd)\!\!\mod~8 & I(X_4(\dd)) \\
\hline \hline
\vphantom{\bar{\gtreqqless}} 0 & - & - & 0 \\
\hline
\vphantom{\bar{\gtreqqless}} 1 & 0 & - & \Theta_8 \\
\hline
\vphantom{\bar{\gtreqqless}} 1 & 1 & 3 & 0 \\
\hline
\vphantom{\bar{\gtreqqless}} 1 & 1 & 7 & \Theta_8
\end{array}
\]
Here a $``-"$ indicates the value of the invariant is not relevant for $I(X_4(\dd))$ in that case.
\end{conjecture}

\begin{rem}
The first the line of the table follows from Theorem \ref{thm:main-p1} and the second line
follows from Theorem \ref{thm:main-p2} and Remark \ref{rem:22}.
By \cite[Remark 2.6\,(1)]{Ka}, $I(X_4(1)) = I(\C P^4) = 0$, which is consistent with the third line of the table. The conjecture is based on analysing the homotopy type of the Thom spectrum of $\xi_4(\dd)\big| _{\C P^4}$ and using this to determine the map $\Theta_8 \cong \Tors \Omega^{O\an{7}}_8 \to \Omega^{O\an{7}}_8(\C P^\infty; \xi_4(\dd))$.

In the spin case, we identified a diffeomorphism invariant property which distinguishes the manifolds $X_4(\dd)$ and $X_4(\dd) \sharp \Sigmaex$. In the $\Theta$-flexible non-spin cases, besides the bordism class in $\Omega^{O\an{7}}_8(\C P^\infty; \xi_4(\dd))$, we do not know of such a property.
\end{rem}

The rest of this paper is organised as follows.  Section \ref{s:prelims} covers necessary preliminaries.
Section \ref{s:spin} treats the spin case. 
Section \ref{s:non-spin1} treats the two non-spin cases whose solution relies on Proposition \ref{prop:Z/4}, which are the case with $v_4 = 0$ and the case with $v_4 \neq 0$ and even total degree (together comprising all non-spin complete intersections with even total degree). Section \ref{s:non-spin2} treats the case of odd total degree. Finally, Section \ref{s:appendix} is an appendix about Toda brackets and extensions, which is needed in Section \ref{s:non-spin1} and specifically for the proof of Proposition \ref{prop:Z/4}.

\subsection*{Acknowledgements} 

It is a pleasure to thank Dennis Sullivan for stimulating conversations about the origins of the Sullivan Conjecture and his work on the classification of manifolds.  We thank the anonymous referee for many useful comments. The second author was supported by the Melbourne Research Scholarship. 

\section{Preliminaries} \label{s:prelims}
In this section we recall and establish some basic facts about complete intersections and Sullivan data.
We then recall Kreck's modified surgery setting for the classification of complete intersections.

\subsection{Complete intersections}
Given a finite multiset $\dd = \left\{ d_1, d_2, \ldots , d_k \right\}$ of positive integers, consider homogeneous polynomials $f_1, f_2, \ldots , f_k \in \C[x_0, x_1, \ldots, x_{n+k}]$ with these degrees. If the zero set $\{ [\ul x] \in {\C}P^{n+k} \mid f_i(\ul x) = 0 \}$ of $f_i$ is a smooth submanifold of ${\C}P^{n+k}$ for every $i$ and these submanifolds are transverse, then their intersection is a representative of the complete intersection $X_n(\dd)$. Any two representatives are diffeomorphic, due to an argument generally attributed to Thom (see e.g.\ \cite{Br2}), which we outline below. 

Let $P_n(\dd)$ denote the space of tuples $(f_1, f_2, \ldots, f_k)$ of homogeneous polynomials in $n{+}k{+}1$ variables of degrees $d_1, d_2, \ldots , d_k$, and let $P_n(\dd)^{\text{ns}} \subseteq P_n(\dd)$ be the subspace of tuples that define complete intersections. The restriction of the tautological map $\{ ([\ul x],(f_1, f_2, \ldots, f_k)) \in {\C}P^{n+k} \times P_n(\dd) \mid  \forall i \colon f_i(\ul x) = 0 \} \rightarrow P_n(\dd)$ to $P_n(\dd)^{\text{ns}}$ is a locally trivial bundle, and its fibres are the representatives of $X_n(\dd)$. Since $P_n(\dd) \setminus P_n(\dd)^{\text{ns}} \subset P_n(\dd)$ is a subvariety of positive complex codimension, $P_n(\dd)^{\text{ns}}$ is a generic (i.e.\ open and everywhere dense) subset in $P_n(\dd)$ and it is path-connected.

This implies that every tuple in $P_n(\dd)$ can be approximated by one in $P_n(\dd)^{\text{ns}}$. 
We also get that any two tuples in $P_n(\dd)^{\text{ns}}$ can be joined by a path in $P_n(\dd)^{\text{ns}}$, which determines a diffeomorphism (up to isotopy) between the fibres over them. So if we take $X_n(\dd)$ to mean any of its representatives, then it is well-defined up to diffeomorphism. Moreover, if two representatives are identified via a path as above, then their natural embeddings in ${\C}P^{n+k}$ are isotopic; hence $X_n(\dd)$ comes equipped with an embedding $i \colon X_n(\dd) \rightarrow {\C}P^{n+k}$, well-defined up to isotopy. The embedding $i$ is $n$-connected (this follows from the Lefschetz Hyperplane Theorem, or see \cite[Chapter 5 (2.6)]{D}).

\subsection{Computation of Sullivan data and the converse of the Sullivan Conjecture} \label{ss:SD} 

\begin{defin}
Let $x \in H^2({\C}P^{\infty})$ denote the standard generator (satisfying $\left< x, [{\C}P^1] \right> = 1$). The pullbacks of $x$ (by the standard embeddings) in $H^2({\C}P^m)$ and $H^2(X_n(\dd))$ will also be denoted by $x$.
\end{defin}

\begin{defin}
For a (complex) bundle $\xi$ and a positive integer $r$ let $r\xi = \xi \oplus \ldots \oplus \xi$ denote the $r$-fold Whitney sum of $\xi$ with itself and let $-r\xi$ denote the stable bundle which is the inverse of $r\xi$. Let $\xi^r = \xi \otimes \ldots \otimes \xi$ be the $r$-fold tensor product (over $\C$) of $\xi$ with itself. For a tuple $\ul r = (r_1, r_2, \ldots , r_k)$ let $\xi^{\ul r} = \xi^{r_1} \oplus \xi^{r_2} \oplus \ldots \oplus \xi^{r_k}$.
\end{defin}

\begin{defin} \label{def:gamma}
Let $\gamma$ be the conjugate of the tautological complex line bundle over ${\C}P^{\infty}$.
\end{defin}

With this notation, the tautological bundle is $\bar{\gamma}$, and since $c_1(\bar{\gamma})=-x$, we have $c_1(\gamma)=x$. It is well-known that the normal bundle of ${\C}P^m$ in ${\C}P^{m+1}$ is $\nu({\C}P^m \rightarrow {\C}P^{m+1}) \cong \gamma \big| _{{\C}P^m}$ and that the stable normal bundle of ${\C}P^m$ is $\nu_{{\C}P^m} \cong -(m{+}1)\gamma \big| _{{\C}P^m}$ (see e.g.\ \cite[\S14]{M-S}). 

\begin{defin} \label{def:xi-nd}
The stable vector bundle $\xi_n(\dd)$ over ${\C}P^{\infty}$ is defined to be 
\[ 
\xi_n(\dd) := -(n{+}k{+}1)\gamma \oplus \gamma^{d_1} \oplus \ldots \oplus \gamma^{d_k}.
\]
\end{defin}

Since the normal bundle of a degree-$r$ hypersurface in ${\C}P^m$ is the restriction of $\gamma^r$ (cf.\ Construction \ref{const:poly1} and Remark \ref{rem:divisor}), we have

\begin{prop} \label{prop:normal-b}
The stable normal bundle $\nu_{X_n(\dd)}$ of $X_n(\dd)$ is isomorphic to $i^*(\xi_n(\dd) \big| _{{\C}P^{n+k}})$. 
\hfill \qed
\end{prop}

\begin{rem} \label{rem:char}
Since $\nu_{X_n(\dd)}$ is the pullback of a bundle over ${\C}P^{\infty}$, all of the stable characteristic classes of 
$X_n(\dd)$ 
lie in the subring $i^*(H^*(\C P^{\infty})) \subseteq H^*(X_n(\dd))$, which is generated by $x \in H^2(X_n(\dd))$. 
In particular, 
$p_j(X_n(\dd)) \in \left< x^{2j} \right> \cong \Z$, $c_j(X_n(\dd)) \in \left< x^j \right> \cong \Z$ and if $2j \leq n$, then $w_{2j}(X_n(\dd)), v_{2j}(X_n(\dd)) \in \left< \varrho_2(x^j) \right> \cong \Z/2$, where $\varrho_2 \colon H^*(X_n(\dd)) \rightarrow H^*(X_n(\dd); \Z/2)$ is reduction mod $2$. (If $2j > n$ and $d$ is even, then $\varrho_2(x^j)=0$.)
\end{rem}

Proposition \ref{prop:normal-b} allows us to compute the characteristic classes of $X_n(\dd)$ in terms of the degrees
$d_1, \ldots, d_k$.  Since $c(\gamma^r) = 1+rx$, the total Chern class of $\xi_n(\dd)$ is $c(\xi_n(\dd)) = (1+x)^{-(n+k+1)} \prod_{i=1}^k (1+d_ix)$. The same formula holds for the normal bundle $\nu_{X_n(\dd)}$, because it is the pullback of $\xi_n(\dd)$. This implies that $c(X_n(\dd)) = (1+x)^{n+k+1} \prod_{i=1}^k (1+d_ix)^{-1}$. For the Pontryagin classes we have $p(\gamma^r) = 1-r^2x^2$, hence $p(X_n(\dd)) = (1-x^2)^{n+k+1} \prod_{i=1}^k (1-d_i^2x^2)^{-1}$.

The Euler characteristic of $X_n(\dd)$ can also be determined, namely 
\begin{multline*}
\chi(X_n(\dd)) = \left< c_n(X_n(\dd)), [X_n(\dd)] \right> = \left< c_n(-\nu_{X_n(\dd)}), [X_n(\dd)] \right> = \left< c_n(-i^*(\xi_n(\dd))), [X_n(\dd)] \right> = \\
\left< c_n(-\xi_n(\dd)), i_*([X_n(\dd)]) \right>,
\end{multline*}
where $i_*([X_n(\dd)]) \in H_{2n}({\C}P^{n+k})$ is $d$ times the generator.

It will be useful to explicitly compute the Stiefel-Whitney classes $w_2$ and $w_4$ and Wu classes $v_2$ and $v_4$ of a $4$-dimensional complete intersection $X_4(\dd)$.

\begin{defin} \label{def:p(d)}
For a multidegree $\dd$ let $p(\dd)$ denote the number of even degrees in $\dd$. 
\end{defin}

\begin{prop} \label{prop:SW-classes}
The Stiefel-Whitney classes $w_2$ and $w_4$ of $\nu_{X_4(\dd)}$ and $X_4(\dd)$ and Wu classes $v_2$ and $v_4$ of $X_4(\dd)$ are determined by $p(\dd)$ mod $4$ as follows (by Remark \ref{rem:char} these Stiefel-Whitney classes and Wu classes can be regarded as elements of $\Z/2$): 
\[
\begin{array}{ c | c | c | c | c}
\vphantom{\gtreqqless} p(\dd) \text{ mod } 4 & 0 & 1 & 2 & 3 \\ \hline \hline
\vphantom{\bar{\gtreqqless}} w_2(\nu_{X_4(\dd)}) = w_2(X_4(\dd)) = v_2(X_4(\dd)) & 1 & 0 & 1 & 0 \\ \hline
\vphantom{\bar{\gtreqqless}} w_4(\nu_{X_4(\dd)}) = v_4(X_4(\dd)) & 1 & 1 & 0 & 0 \\ \hline
\vphantom{\bar{\gtreqqless}} w_4(X_4(\dd)) & 0 & 1 & 1 & 0 \\
\end{array}
\]
\end{prop}

\begin{proof}
The total Chern class of $\xi_n(\dd)$ is given by the following formula: 
\begin{multline*}
c(\xi_n(\dd)) = (1+x)^{-(n+k+1)} \prod_{i=1}^k (1+d_ix) = \\
1 + \Biggl( - (n + 1) + \sum_{i=1}^k (d_i-1) \Biggr) x + \Biggl( \binom{n+2}{2} - (n+2)\sum_{i=1}^k (d_i-1) + \sum_{1 \leq i < j \leq k} (d_i-1)(d_j-1) \Biggr) x^2  + \ldots
\end{multline*}
We have $w_{2i} = \varrho_2(c_i)$.  Therefore 
\[
\begin{aligned}
w_2(\xi_4(\dd)) &= \varrho_2 \Biggl( \Biggl( -5 + \sum_{i=1}^k (d_i-1) \Biggr) x \Biggr) = \varrho_2((1+p(\dd)) x)  \quad\quad \text{and} \\
w_4(\xi_4(\dd)) &= \varrho_2 \Biggl( \Biggl( 15 - 6 \sum_{i=1}^k (d_i-1) + \sum_{1 \leq i < j \leq k} (d_i-1)(d_j-1) \Biggr) x^2 \Biggr) = \varrho_2 \Biggl( \left( 1+\binom{p(\dd)}{2} \right) x^2 \Biggr) \text{\,.}
\end{aligned} 
\]
We have the same formulas for the Stiefel-Whitney classes of $\nu_{X_4(\dd)}$, because 
$\nu_{X_4(\dd)}$ is the pullback of $\xi_4(\dd)$. 
Since $H^1(X_4(\dd); \Z/2) \cong H^3(X_4(\dd); \Z/2) \cong 0$, the Stiefel-Whitney classes $w_2(X_4(\dd))$ and $w_4(X_4(\dd))$ are determined by $w_2(\nu_{X_4(\dd)})$ and $w_4(\nu_{X_4(\dd)})$ via the Cartan formula. We get that $w_2(X_4(\dd)) = w_2(\nu_{X_4(\dd)})$ and $w_4(X_4(\dd)) = w_2(\nu_{X_4(\dd)})^2 + w_4(\nu_{X_4(\dd)})$. By applying the Wu formula we get that $v_2(X_4(\dd)) = w_2(X_4(\dd))$ and $v_4(X_4(\dd)) = w_2(X_4(\dd))^2 + w_4(X_4(\dd)) = w_4(\nu_{X_4(\dd)})$.
\end{proof}

\begin{rem} \label{rem:p-mod-4}
Notice that if $v_2(X_4(\dd)) \neq 0$ and $v_4(X_4(\dd)) \neq 0$, then $p(\dd)$ is divisible by $4$. This means that either $p(\dd)=0$, hence all degrees are odd and so the total degree is odd; or $p(\dd) \geq 4$, so there are at least $4$ even degrees and then the total degree is divisible by $16$. 
\end{rem}

The following proposition implies that the converse of the Sullivan Conjecture holds.

\begin{prop} \label{prop:SC-conv}
Let $n \geq 3$ and let $\dd$ and $\dd'$ be two multidegrees. If there is a homotopy equivalence $f \colon X_n(\dd) \rightarrow X_n(\dd')$ such that $f^*(\nu_{X_n(\dd')}) \cong \nu_{X_n(\dd)}$ (e.g.\ if $f$ is a diffeomorphism), then $SD_n(\dd) = SD_n(\dd')$.
\end{prop}

\begin{proof}
If $n \geq 3$, then $H^2(X_n(\dd)) \cong H^2(X_n(\dd')) \cong \Z$, so any homotopy equivalence $X_n(\dd) \rightarrow X_n(\dd')$ preserves $x$ up to sign. If $f$ sends $x$ to $-x$, then we can replace it with another homotopy equivalence that preserves $x$, by composing it with a self-diffeomorphism of $X_n(\dd)$ (or $X_n(\dd')$) that changes the sign of $x$. (Consider the conjugation map of the ambient $\C P^{n+k}$, it sends $x$ to $-x$. If a representative of $X_n(\dd)$ is given by polynomials $f_1, f_2, \ldots, f_k$, then its image is another representative of the same complete intersection, given by the conjugate polynomials $\bar{f}_1, \bar{f}_2, \ldots, \bar{f}_k$. By Thom's argument there is a diffeomorphism between the two representatives such that after identifying them their embeddings into $\C P^{n+k}$ are isotopic. By composing this diffeomorphism with the restriction of the conjugation map, we get a self-diffeomorphism of either representative that changes the sign of $x$.) Since $\left< x^n, [X_n(\dd)] \right> = d$ and $\left< x^n, [X_n(\dd')] \right> = d'$, this means that $d=d'$. The Euler characteristic is a homotopy invariant. The Pontryagin classes are preserved by $f$ because of the assumption on the normal bundles, and since the elements $x^{2i}$ are preserved, the Pontryagin classes are also invariant when regarded as integers.
\end{proof}

\begin{rem} \label{rem:SD-equal}
If $\Sigma \in \Theta_{2n}$ is a homotopy sphere, then there is a homeomorphism between $X_n(\dd)$ and $X_n(\dd) \sharp \Sigma$ which preserves normal bundles. 
Thus if $n \geq 3$ and $X_n(\dd) \sharp \Sigma$ is diffeomorphic to a complete intersection $X_n(\dd')$, then $SD_n(\dd) = SD_n(\dd')$.
\end{rem}

\subsection{The setting for modified surgery} \label{ss:surgery}
We recall the setup for the modified surgery arguments of Kreck \cite[Section 8]{Kr} and Fang and Klaus \cite{F-K}, which will be used in Sections \ref{s:non-spin1} and \ref{s:non-spin2}.

Recall that 
the inclusion $i \colon X_n(\dd) \to {\C}P^\infty$ is $n$-connected. It is covered by a bundle map $\bar{i} \colon \nu_{X_n(\dd)} \rightarrow \xi_n(\dd)$ (Proposition \ref{prop:normal-b})
and therefore $(i,\bar{i})$ is a normal $(n{-}1)$-smoothing over $({\C}P^\infty, \xi_n(\dd))$.

Let $\gamma_{BO}$ denote the universal stable vector bundle over $BO$ and $\gamma_{BO\an{j}}$ its pullback to $BO\an{j}$, the $(j{-}1)$-connected cover of $BO$. 
Let $B_n := {\C}P^\infty \times BO\an{n{+}1}$, 
then $(i,\bar{i})$ can be regarded as a normal map over $(B_n, \xi_n(\dd) \times \gamma_{BO\an{n+1}})$ (and it is still $n$-connected). Moreover, the map $B_n \rightarrow BO$ inducing $\xi_n(\dd) \times \gamma_{BO\an{n+1}}$ from $\gamma_{BO}$ is $n$-co-connected, therefore $(B_n, \xi_n(\dd) \times \gamma_{BO\an{n+1}})$ is the normal $(n{-}1)$-type of $X_n(\dd)$. 
When $n = 4$, we have that  $BO\an{5} = BO\an{8} = BString$ by Bott periodicity, and thus 
$(i,\bar{i})$ represents an element in the bordism group of closed $8$-manifolds with normal maps to
$(B_4, \xi_4(\dd) \times \gamma_{BO\an{8}})$.  We denote this bordism group by
$\Omega_8^{\fr}(B_4; \xi_4(\dd) \times \gamma_{BO\an{8}})$; it is canonically
isomorphic to the twisted string bordism group $\Omega_8^{O\an{7}}({\C}P^\infty; \xi_4(\dd))$.

First we will want to apply Proposition \ref{prop:classification_general} when $X' = X_4(\dd) \sharp \Sigmaex$. 
There is a canonical homeomorphism $h \colon X_4(\dd) \sharp \Sigmaex \rightarrow X_4(\dd)$, 
and since homotopy spheres are stably parallelisable \cite[Theorem 3.1]{K-M}, $h$ is covered by a bundle map $\bar{h}$ of stable normal bundles. Then $(i \circ h,\bar{i} \circ \bar{h})$ is also a normal $3$-smoothing over 
$(B_4, \xi_4(\dd) \times \gamma_{BO\an{8}})$, and in the bordism group $\Omega_8^{O\an{7}}({\C}P^\infty; \xi_4(\dd))$ it represents $[i,\bar{i}] + [\Sigmaex]$, where $[\Sigmaex]$ is the image of $\Sigmaex$ under the canonical homomorphism 
$i_0 \colon \Theta_8 \rightarrow \Omega_8^{O\an{7}}({\C}P^\infty; \xi_4(\dd))$. 
So to apply Proposition \ref{prop:classification_general} in this setting we need to show that this homomorphism is trivial and we do this in the non-spin case with $v_4(X_4(\dd)) = 0$; see Proposition \ref{prop:i_CP1}.

Now suppose that $X_4(\dd')$ is another complete intersection with an analogous normal $3$-smoothing
$(i',\bar{i}')$ over $({\C}P^\infty, \xi_4(\dd'))$. If the Pontryagin classes of $X_4(\dd)$ and $X_4(\dd')$ agree, in particular if $SD_4(\dd) = SD_4(\dd')$, then the Pontryagin classes $p_1$ and $p_2$ of $\xi_4(\dd)$ and $\xi_4(\dd')$ also agree. This implies that 
$\xi_4(\dd) \big| _{{\C}P^4} \cong \xi_4(\dd') \big| _{{\C}P^4}$ (by Sanderson \cite[Theorem (3.9)]{Sa} every stable bundle over ${\C}P^4$ is isomorphic to $\xi_{a,b} := a \gamma \oplus b (\gamma \otimes_{\R} \gamma)$ for some $a,b \in \Z$, and the function $(a,b) \mapsto (p_1(\xi_{a,b}),p_2(\xi_{a,b}))$ is injective). 
Thus $\xi_4(\dd') \ominus \xi_4(\dd)$ is trivial over ${\C}P^4$, so it has an $O\an{7}$-structure. 
Therefore $\Id_{{\C}P^\infty}$ has a lift $g \colon {\C}P^\infty \rightarrow B_4$ which induces $\xi_4(\dd')$ from $\xi_4(\dd) \times \gamma_{BO\an{8}}$.
Hence if $\bar{g} \colon \xi_4(\dd') \rightarrow \xi_4(\dd) \times \gamma_{BO\an{8}}$ is a bundle map over $g$, 
then $(g \circ i', \bar{g} \circ \bar{i}')$ is a normal $3$-smoothing of $X_4(\dd')$ over $(B_4, \xi_4(\dd) \times \gamma_{BO\an{8}})$.

If $SD_4(\dd) = SD_4(\dd')$, then the discussion in the paragraph above shows that $X_4(\dd)$ and $X_4(\dd')$ admit normal $3$-smoothings over $(B_4, \xi_4(\dd) \times \gamma_{BO\an{8}})$ and $\chi(X_4(\dd)) = \chi(X_4(\dd'))$, therefore to apply Proposition \ref{prop:classification_general} it is enough to prove that these
normal $3$-smoothings represent the same bordism class in $\Omega_8^{O\an{7}}({\C}P^\infty; \xi_4(\dd))$. Fang and Klaus obtained Theorem \ref{thm:FK} by showing that the difference of these bordism classes 
is in the image of the canonical homomorphism 
$i_0 \colon \Theta_8 \rightarrow \Omega_8^{O\an{7}}({\C}P^\infty; \xi_4(\dd))$.  
In the non-spin cases with $v_4(X_4(\dd)) \neq 0$, we are able to show in Sections \ref{ss:even2} 
and \ref{s:non-spin2} that the bordism classes agree.


\section{The spin case}
\label{s:spin}

In this section we prove that $4$-dimensional spin complete intersections are strongly $\Theta$-flexible, hence the Sullivan Conjecture holds for them.

\begin{defin}
For a smooth manifold $X$ and a cohomology class $\alpha \in H^2(X)$, let $E(X,\alpha)$ denote the total space of the complex line bundle over $X$ with first Chern class $\alpha$. Let $D(X,\alpha)$ denote its disc bundle and $S(X,\alpha)$ denote its sphere bundle.
\end{defin}

Recall that $x \in H^2(X_n(\dd))$ is the pullback of the standard generator of $H^2({\C}P^{\infty})$. First we will prove that for every complete intersection $X_n(\dd)$ the total space $S(X_n(\dd),x)$ admits a framing such that it is framed null-cobordant (where by a framing of a manifold we mean a trivialisation of its stable normal bundle, equivalently, of its stable tangent bundle); see Theorem \ref{thm:framed-0}.

Recall that (a representative of) the complete intersection $X_{n+1}(\dd) \subset \C P^{n+k+1}$ is the set of common zeros of some homogeneous polynomials $f_1, f_2, \ldots , f_k \in \C[x_0, x_1, \ldots , x_{n+k+1}]$. If $f_{k+1} \in \C[x_0, x_1, \ldots , x_{n+k+1}]$ is linear and its zero set $L$ is transverse to $X_{n+1}(\dd)$, then $X_n(\dd) = X_{n+1}(\dd) \cap L$. 

\begin{prop} \label{prop:par}
The complement $X_{n+1}(\dd) \setminus X_n(\dd)$ is stably parallelisable. 
\end{prop}

\begin{proof}
We have the following commutative diagram of embeddings: 
\[
\xymatrix{
\C P^{n+k+1} \setminus L \ar[r] & \C P^{n+k+1} \\
X_{n+1}(\dd) \setminus X_n(\dd) \ar[r] \ar[u]^-{i} & X_{n+1}(\dd) \ar[u]_{i}
}
\]
So $\nu_{X_{n+1}(\dd) \setminus X_n(\dd)} \cong \nu_{X_{n+1}(\dd)} \big| _{X_{n+1}(\dd) \setminus X_n(\dd)} \cong i^*(\xi_{n+1}(\dd)) \big| _{X_{n+1}(\dd) \setminus X_n(\dd)} \cong i^* \bigl( \xi_{n+1}(\dd) \big| _{\C P^{n+k+1} \setminus L} \bigr)$ (using Proposition \ref{prop:normal-b}), and this is trivial, because $\C P^{n+k+1} \setminus L$ is contractible (recall that $L$ is a hyperplane). 
\end{proof}

\begin{prop} \label{prop:nu}
We have $\nu(X_n(\dd) \rightarrow X_{n+1}(\dd)) \cong i^*(\gamma)$ (see Definition \ref{def:gamma}).
\end{prop}

\begin{proof}
Since $L$ is transverse to $X_{n+1}(\dd)$ and $X_n(\dd) = X_{n+1}(\dd) \cap L$, the normal bundle $\nu(X_n(\dd) \rightarrow X_{n+1}(\dd))$ is the restriction of $\nu(L \rightarrow \C P^{n+k+1})$, hence $\nu(X_n(\dd) \rightarrow X_{n+1}(\dd)) \cong \nu(L \rightarrow \C P^{n+k+1}) \big| _{X_n(\dd)} \cong \gamma \big| _{X_n(\dd)}$.
\end{proof}

\begin{thm} \label{thm:framed-0}
For any complete intersection $X_n(\dd)$, (the total space of) the $S^1$-bundle $S(X_n(\dd),x)$ 
admits a framing $F_0$ such that $[S(X_n(\dd),x), F_0]=0 \in \Omega^{\fr}_{2n+1}$. 
\end{thm}

\begin{proof}
Let $U$ be a tubular neighbourhood of $X_n(\dd)$ in $X_{n+1}(\dd)$. By Proposition \ref{prop:nu} it is diffeomorphic to the disc bundle of $i^*(\gamma)$, whose first Chern class is $x$, therefore $\partial U \approx S(X_n(\dd),x)$. Its complement, $X_{n+1}(\dd) \setminus \interior U$ is a codimension-$0$ submanifold in $X_{n+1}(\dd) \setminus X_n(\dd)$. The latter is stably parallelisable by Proposition \ref{prop:par}, so $X_{n+1}(\dd) \setminus \interior U$ is stably parallelisable too. If we choose $F_0$ to be the restriction of a framing of $X_{n+1}(\dd) \setminus \interior U$ to the boundary $\partial(X_{n+1}(\dd) \setminus \interior U) \approx \partial U \approx S(X_n(\dd),x)$, then $(S(X_n(\dd),x),F_0)$ is framed null-cobordant. 
\end{proof}

The goal of the rest of this section is to prove that $S(X_4(\dd) \sharp \Sigmaex,x)$ is not framed nullcobordant (with any framing) if $X_4(\dd)$ is spin; see Theorem \ref{thm:fr-all}. First we show that, when an $m$-manifold $X$ is replaced by $X \sharp \Sigma$ for a homotopy $m$-sphere $\Sigma$, the framed cobordism class of $S(X,\alpha)$ changes by $\Sigma \times S^1$ (with a certain choice of framings); see Lemma \ref{lem:sum}. In Lemma \ref{lem:s1-fr} we give a formula to compute the framing of the $S^1$ component. By applying this formula we prove that if $X_4(\dd)$ is spin, then $S(X_4(\dd) \sharp \Sigmaex,x)$ has a framing such that it is not framed nullcobordant (Theorem \ref{thm:fr-exist}). Finally we show that we cannot make the framed cobordism class vanish by changing the framing.

\begin{lem} \label{lem:sum}
Suppose that $m \geq 3$, $X$ is an $m$-manifold, $\alpha \in H^2(X)$ and $F_0$ is a framing of $S(X,\alpha)$. Then there exists a framing $F_2$ of $S^1$ such that for every $\Sigma \in \Theta_m$ and framing $F_1$ of $\Sigma$ there is a framing $F$ of $S(X \sharp \Sigma,\alpha)$ such that 
\[
[S(X,\alpha), F_0] + [\Sigma \times S^1, F_1 \times F_2] = [S(X \sharp \Sigma,\alpha), F] \in \Omega_{m+1}^{\fr} \text{\,.}
\]
\end{lem}

\begin{proof}
Fix an embedding $D^m \rightarrow X$ where the connected sum is done. There is a homotopically unique homeomorphism between $X$ and $X \sharp \Sigma$ that is the identity on $X \setminus \interior D^m$, so there is a canonical isomorphism $H^2(X) \cong H^2(X \sharp \Sigma)$, so $\alpha$ can be regarded as an element of $H^2(X \sharp \Sigma)$, so $S(X \sharp \Sigma,\alpha)$ makes sense. The homomorphisms $H^2(X) \rightarrow H^2(X \setminus \interior D^m) \leftarrow H^2(X \sharp \Sigma)$ are injective (isomorphisms if $m \geq 4$), therefore $S(X \sharp \Sigma,\alpha)$ is (the total space of) the unique $S^1$-bundle over $X \sharp \Sigma$ whose restriction to $X \setminus \interior D^m$ is isomorphic to that of $S(X,\alpha)$. 

Let $W = (S(X,\alpha) \sqcup \Sigma \times S^1) \times I \cup_f (D^m \times S^1 \times I)$, where the gluing map 
\[
f \colon D^m \times S^1 \times \partial I 
\rightarrow (S(X,\alpha) \sqcup \Sigma \times S^1) \times \{ 1 \}
\]
is the disjoint union of the (homotopically unique) local trivialisation $f_0 \colon D^m \times S^1 \times \{ 0 \} \rightarrow S(X,\alpha) \times \{ 1 \}$ of $S(X,\alpha)$ over the fixed $D^m$ and the product map $f_1 \colon D^m \times S^1 \times \{ 1 \} \rightarrow \Sigma \times S^1 \times \{ 1 \}$, where $D^m \rightarrow \Sigma$ is the embedding used to construct the connected sum $X \sharp \Sigma$. Then $\partial W = \partial_- W \sqcup \partial_+ W$, where $\partial_- W = (S(X,\alpha) \sqcup \Sigma \times S^1) \times \{ 0 \}$ and $\partial_+ W = (S(X,\alpha) \setminus (\interior D^m \times S^1) \sqcup (\Sigma \setminus \interior D^m) \times S^1) \times \{ 1 \} \cup_f S^{m-1} \times S^1 \times I$. Thus $\partial_+ W$ is an $S^1$-bundle over $(X \setminus \interior D^m) \cup S^{m-1} \times I \cup (\Sigma \setminus \interior D^m) \approx X \sharp \Sigma$ and it coincides with $S(X,\alpha)$ over $X \setminus \interior D^m$, therefore $\partial_+ W \approx S(X \sharp \Sigma,\alpha)$. 

The inclusion $S(X,\alpha) \times \{ 0 \} \rightarrow S(X,\alpha) \times I \cup_{f_0} D^m \times S^1 \times I$ is a homotopy equivalence, covered by a bundle map between the stable normal bundles, therefore the framing $F_0$ can be extended to a framing of $S(X,\alpha) \times I \cup_{f_0} D^m \times S^1 \times I$. The restriction of this framing to $D^m \times S^1 \times \{ 1 \}$ is $E_m \times F_2$, where $E_m$ is the homotopically unique framing of $D^m$ and $F_2$ is some framing of $S^1$ (because every framing of $D^m \times S^1$ is of this form). Similarly, we can take the framing $F_1 \times F_2$ of $\Sigma \times S^1$ and extend it to $\Sigma \times S^1 \times I$. The restriction of this framing to $D^m \times S^1 \times \{ 1 \}$ is again $E_m \times F_2$ (up to homotopy), therefore the framings of $S(X,\alpha) \times I \cup_{f_0} D^m \times S^1 \times I$ and $\Sigma \times S^1 \times I$ together determine a framing of $W$. Let $F$ denote its restriction to $\partial_+ W \approx S(X \sharp \Sigma,\alpha)$. Then $W$ is a framed cobordism between the framed manifolds $(S(X,\alpha),F_0) \sqcup (\Sigma \times S^1,F_1 \times F_2)$ and $(S(X \sharp \Sigma,\alpha),F)$.
\end{proof}

\begin{lem} \label{lem:s1-fr}
Suppose that, in addition to the assumptions of Lemma \ref{lem:sum}, there is an $[a] \in \pi_2(X)$ such that $\left< \alpha, \rho([a]) \right> = 1$, where $\rho \colon \pi_2(X) \rightarrow H_2(X)$ is the Hurewicz homomorphism. Then for any such $[a] \in \pi_2(X)$ and the framing $F_2$ constructed in the proof of Lemma \ref{lem:sum} we have 
\[
[S^1,F_2] = \left< w_2(X), \rho([a]) \right> + 1 \text{\,,}
\]
where both sides are regarded as elements of $\Z/2$ (using that $\Omega_1^{\fr} \cong \Z/2$). 
\end{lem}

\begin{proof}
Fix a local trivialisation $f_0 \colon D^m \times S^1 \rightarrow S(X,\alpha)$, as in the proof of Lemma \ref{lem:sum}. The framing $F_2$ is defined by the property that the restriction of $F_0$ to $f_0(D^m \times S^1)$ is $E_m \times F_2$ (throughout this proof we will identify the framings of $f_0(D^m \times S^1)$ with the framings of $D^m \times S^1$ via (the derivative of) $f_0$). First we will give another characterisation of $E_m \times F_2$.

If $\partial \colon \pi_2(X) \rightarrow \pi_1(S^1) \cong \Z$ denotes the boundary map 
in the homotopy long exact sequence of the fibration $S^1 \rightarrow S(X,\alpha) \rightarrow X$, then $\partial([a]) = \left< \alpha, \rho([a]) \right>$ (this holds if $X=S^2$ and $a = \Id_{S^2}$, because $\alpha$ is the Euler class of $E(S^2,\alpha)$, and in general $a \colon S^2 \rightarrow X$ induces a commutative diagram between the exact sequences). Moreover, $\partial$ is the composition of the isomorphism $\pi_2(X) \cong \pi_2(S(X,\alpha), S^1)$ and the boundary map $\pi_2(S(X,\alpha), S^1) \rightarrow \pi_1(S^1)$. Therefore for any $[a]$ with $\left< \alpha, \rho([a]) \right> = 1$ there is a map $\tilde{a} \colon D^2 \rightarrow S(X,\alpha)$ (well-defined up to homotopy) such that $\tilde{a} \big| _{S^1}$ is the inclusion of a fibre and $(\tilde{a}, \tilde{a} \big| _{S^1})$ represents the element in $\pi_2(S(X,\alpha), S^1)$ corresponding to $[a] \in \pi_2(X)$. We can lift $\tilde{a}$ to a map $\bar{a} \colon D^2 \rightarrow S(X,\alpha) \times \R^+_0$ (where $\R^+_0$ denotes $[0,\infty)$) such that $\bar{a}$ is an embedding, it is transverse to $S(X,\alpha) \times \{ 0 \}$, $\bar{a}^{-1}(S(X,\alpha) \times \{ 0 \}) = S^1$, $\bar{a} \big| _{S^1} \colon S^1 \rightarrow S(X,\alpha) \times \{ 0 \}$ is the inclusion of a fibre and $[\bar{a}, \bar{a} \big| _{S^1}] = [\tilde{a}, \tilde{a} \big| _{S^1}] \in \pi_2(S(X,\alpha) \times \R^+_0, S^1 \times \R^+_0) \cong \pi_2(S(X,\alpha), S^1)$. 

Let $U$ be a tubular neighbourhood of $\bar{a}(D^2)$ in $S(X,\alpha) \times \R^+_0$, we can assume that $U \cap S(X,\alpha) \times \{ 0 \} = f_0(D^m \times S^1)$. (Note that $U$ is the total space of a $D^m$-bundle over $D^2$, so it has a homotopically unique trivialisation $D^2 \times D^m \rightarrow U$, but the restriction of this trivialisation to $S^1$ may differ from $f_0$, the difference is given by an element of $\pi_1(SO_m) \cong \Z/2$.) The framing $F_0$ can be extended to a framing of $S(X,\alpha) \times \R^+_0$ and then restricted to a framing of $U$. As mentioned above, if we further restrict this framing to $f_0(D^m \times S^1)$, we get $E_m \times F_2$. Since $U$ is contractible, it has a homotopically unique framing, so this means that $E_m \times F_2$ is the restriction of the homotopically unique framing of $U$. 

The local trivialisation $f_0$ is the restriction of a local trivialisation $\bar{f}_0 \colon D^m \times D^2 \rightarrow D(X,\alpha)$. The homotopically unique framing of $\bar{f}_0(D^m \times D^2)$ is $E_m \times E_2$; its restriction to $f_0(D^m \times S^1)$ is $E_m \times (E_2 \big| _{S^1})$. Since $(S^1, E_2 \big| _{S^1})$ is the framed boundary of $(D^2, E_2)$, $[S^1, E_2 \big| _{S^1}] = 0 \in \Omega_1^{\fr} \cong \Z/2$. So if $g \in \pi_1(SO) \cong \Z/2$ denotes the difference of the framings $F_2$ and $E_2 \big| _{S^1}$ of $S^1$, then $[S^1, F_2] = g \in \Z/2$. 

We have $D(X,\alpha) \cup_{S(X,\alpha)} S(X,\alpha) \times \R^+_0 \approx E(X,\alpha)$ (in each fibre $D^2 \cup_{S^1} S^1 \times \R^+_0 \approx \R^2$) and $\bar{f}_0(D^m \times D^2) \cup U$ is a tubular neighbourhood of $\bar{f}_0(\{ 0 \} \times D^2) \cup \bar{a}(D^2) \approx S^2$ in $E(X,\alpha)$. As a $D^m$-bundle over $S^2$ it is classified by an element of $\pi_2(BSO_m) \cong \Z/2$. Under the isomorphism $\pi_2(BSO_m) \cong \pi_2(BSO) \cong \pi_1(SO)$ this element corresponds to $g$ (because it is equal to the difference of the restrictions of the unique framings of $\bar{f}_0(D^m \times D^2)$ and $U$, which are $E_m \times (E_2 \big| _{S^1})$ and $E_m \times F_2$ respectively). 

So we need to determine the normal bundle of the embedding $S^2 \rightarrow E(X,\alpha)$ as an element of $\pi_2(BSO)$. Since $S^2$ is stably parallelisable, it is the same as the restriction of the stable tangent bundle $\tau_{E(X,\alpha)}$ to $S^2$. The embedding $S^2 \rightarrow E(X,\alpha)$ is homotopic to its projection to the zero section ($X$). Since $\bar{f}_0(\{ 0 \} \times D^2)$ is a fibre of $D(X,\alpha)$, its projection to $X$ is one point. The map $\bar{a} \colon D^2 \rightarrow S(X,\alpha) \times \R^+_0$ is a lift of $\tilde{a} \colon D^2 \rightarrow S(X,\alpha)$, which is a lift of a map $a \colon S^2 \rightarrow X$ representing $[a] \in \pi_2(X)$. Therefore the composition of the embedding $S^2 \rightarrow E(X,\alpha)$ and the projection to $X$ is $a$. The restriction of $\tau_{E(X,\alpha)}$ to $X$ is $E(X,\alpha) \oplus \tau_X$. So the bundle we are interested in is the pullback of $E(X,\alpha) \oplus \tau_X$ by $a$.

The second Stiefel-Whitney class detects $\pi_2(BSO)$, so 
\begin{multline*}
g = \left< w_2(a^*(E(X,\alpha) \oplus \tau_X)), [S^2] \right> = \left< w_2(E(X,\alpha) \oplus \tau_X), a_*([S^2]) \right> = {} \\
\left< w_2(E(X,\alpha)) + w_2(\tau_X), \rho([a]) \right> = \left< \varrho_2(\alpha), \rho([a]) \right> + \left< w_2(X), \rho([a]) \right> = 1 + \left< w_2(X), \rho([a]) \right>,
\end{multline*}
where $\varrho_2 \colon H^2(X) \rightarrow H^2(X; \Z/2)$ denotes reduction mod $2$ and we used that $E(X,\alpha)$ is a complex line bundle, so $w_1(E(X,\alpha)) = 0$ and $w_2(E(X,\alpha)) = \varrho_2(c_1(E(X,\alpha))) = \varrho_2(\alpha)$ and that $\left< \alpha, \rho([a]) \right> = 1$. 

We already saw that $g$ corresponds to $[S^1,F_2]$, so the statement follows. 
\end{proof}

\begin{prop} \label{prop:product}
If $F_2$ is the (homotopically unique) framing of $S^1$ such that $[S^1,F_2]$ is the non-trivial element in 
$\Omega^{\fr}_1 \cong \Z/2$, and $F_1$ is any framing of $\Sigmaex$, then $[\Sigmaex \times S^1, F_1 \times F_2] \neq 0 \in \Omega_9^{\fr}$. Moreover, $[\Sigmaex \times S^1, F_1 \times F_2]$ is not contained in the image of the $J$-homomorphism $J_9 \colon \pi_9(SO) \rightarrow \Omega_9^{\fr}$.
\end{prop}

\begin{proof}
It follows from Kervaire-Milnor \cite[Section 4 and Theorem 5.1]{K-M} that $[\Sigmaex,F_1] \not\in \Image J_8$. Under the Pontryagin-Thom isomorphism the map ${} \times [S^1, F_2] \colon \Omega_8^{\fr} \rightarrow \Omega_9^{\fr}$ corresponds to ${} \cdot \eta \colon \pi^s_8 \rightarrow \pi^s_9$, which is injective by Toda \cite[p.\ 189 and Theorem 14.1 i)]{T}. By Adams \cite[Proof of Example 12.15]{A}, we have $\Image J_9 = (\Image J_8)\eta$. Therefore 
\begin{multline*}
[\Sigmaex, F_1] \times [S^1, F_2] \in (\Omega_8^{\fr} \setminus \Image J_8) \times [S^1, F_2] = (\Omega_8^{\fr} \times [S^1, F_2]) \setminus (\Image J_8 \times [S^1, F_2]) = \\
(\Omega_8^{\fr} \times [S^1, F_2]) \setminus \Image J_9 \subseteq \Omega_9^{\fr} \setminus \Image J_9 .
\end{multline*}
In particular $[\Sigmaex, F_1] \times [S^1, F_2] \neq 0$.
\end{proof}

\begin{thm} \label{thm:fr-exist}
If $X_4(\dd)$ is spin, then there is a framing $F$ such that $[S(X_4(\dd) \sharp \Sigmaex,x),F] \neq 0 \in \Omega_9^{\fr}$. 
\end{thm}

\begin{proof}
Let $F_0$ be a framing of $S(X_4(\dd),x)$ such that $[S(X_4(\dd),x),F_0]=0$ (see Theorem \ref{thm:framed-0}). Let $F_1$ be any framing of $\Sigmaex$. By Lemma \ref{lem:sum} there are framings $F_2$ and $F$ such that $[S(X_4(\dd) \sharp \Sigmaex,x),F] = [S(X_4(\dd),x),F_0] + [\Sigmaex \times S^1, F_1 \times F_2] = [\Sigmaex \times S^1, F_1 \times F_2]$. Since $x$ is a generator of $H^2(X_4(\dd))$, there is a generator $[a]$ of $\pi_2(X_4(\dd))$ such that $\left< x, \rho([a]) \right> = 1$, so we can apply Lemma \ref{lem:s1-fr}, and since $X_4(\dd)$ is spin, we get that $[S^1, F_2] = 1$. By Proposition \ref{prop:product} $[\Sigmaex \times S^1, F_1 \times F_2] \neq 0$ and this implies that $[S(X_4(\dd) \sharp \Sigmaex,x),F] \neq 0$.
\end{proof}

\begin{thm} \label{thm:fr-all}
If $X_4(\dd)$ is spin, then for every framing $F$, $[S(X_4(\dd) \sharp \Sigmaex,x), F] \neq 0 \in \Omega_9^{\fr}$. 
\end{thm}

\begin{proof}
First we show that $S(X_4(\dd) \sharp \Sigmaex,x)$ is $3$-connected. Recall that the embedding $X_4(\dd) \rightarrow \C P^{4+k}$ is $4$-connected. 
Therefore we have $\pi_1(X_4(\dd)) \cong \pi_3(X_4(\dd)) \cong 0$ and $\pi_2(X_4(\dd)) \cong \Z$. From the homotopy long exact sequence of the fibration $S^1 \rightarrow S(X_4(\dd),x) \rightarrow X_4(\dd)$, we obtain that $S(X_4(\dd),x)$ is $3$-connected. Since $S(X_4(\dd) \sharp \Sigmaex,x)$ is homeomorphic to $S(X_4(\dd),x)$, it is also a $3$-connected $9$-manifold. This implies that $S(X_4(\dd) \sharp \Sigmaex,x)$ is homotopy equivalent to a CW-complex with cells only in dimensions $0$, $4$, $5$ and $9$ (see \cite[Theorem 6.1]{Sm}). 

Any two framings of $S(X_4(\dd) \sharp \Sigmaex,x)$ differ by a map $S(X_4(\dd) \sharp \Sigmaex,x) \rightarrow SO$. Since $\pi_4(SO) \cong \pi_5(SO) \cong 0$, this difference is in fact an element of $\pi_9(SO)$. Changing the framing of the $9$-cell by an element of $\pi_9(SO)$ has the same effect of on the framed cobordism class as taking connected sum with $S^9$ with the corresponding framing, which is given by the $J$-homomorphism $J_9 \colon \pi_9(SO) \rightarrow \Omega_9^{\fr}$. 
Therefore the set of cobordism classes in $\Omega_9^{\fr}$ represented by $S(X_4(\dd) \sharp \Sigmaex,x)$ (with any framing) is a coset of $\Image J_9$. By Proposition \ref{prop:product} and the proof of Theorem \ref{thm:fr-exist} this coset has an element which is not in $\Image J_9$, therefore it is not the trivial coset. So it does not contain $0$, therefore $0 \in \Omega_9^{\fr}$ is not represented by $S(X_4(\dd) \sharp \Sigmaex,x)$ with any framing. 
\end{proof}

Now we can conclude that $4$-dimensional spin complete intersections are strongly $\Theta$-flexible.

\begin{thm} \label{thm:spin}
If $X_4(\dd)$ is spin, then $X_4(\dd) \sharp \Sigmaex$ is not diffeomorphic to a complete intersection.
\end{thm}

\begin{proof}
Suppose that $X_4(\dd) \sharp \Sigmaex$ is diffeomorphic to some complete intersection $X_4(\dd')$. The diffeomorphism induces an isomorphism between $H^2(X_4(\dd'))$ and $H^2(X_4(\dd) \sharp \Sigmaex)$. We may assume that the generator $x \in H^2(X_4(\dd'))$ goes into the generator of $H^2(X_4(\dd) \sharp \Sigmaex)$ corresponding to $x$ under the isomorphism $H^2(X_4(\dd) \sharp \Sigmaex) \cong H^2(X_4(\dd) \setminus \interior D^8) \cong H^2(X_4(\dd))$ (see the proof of Lemma \ref{lem:sum}). This is because $X_4(\dd')$ has a self-diffeomorphism which sends $x$ to $-x$ (see the proof of Proposition \ref{prop:SC-conv}). This implies that $S(X_4(\dd) \sharp \Sigmaex,x)$ is diffeomorphic to $S(X_4(\dd'),x)$.

By Theorem \ref{thm:framed-0}, $S(X_4(\dd'),x)$ has a framing $F_0$ such that $(S(X_4(\dd'),x), F_0)$ is framed nullcobordant, but by Theorem \ref{thm:fr-all} $S(X_4(\dd) \sharp \Sigmaex,x)$ does not have such a framing, so they are not diffeomorphic. This contradiction shows that $X_4(\dd) \sharp \Sigmaex$ is not diffeomorphic to any complete intersection $X_4(\dd')$.
\end{proof}


\section{The non-spin cases with even total degree} 
\label{s:non-spin1}
In this section we prove Theorem \ref{thm:main-p2} and Theorem \ref{thm:main-p3}\,(a). 
Both of these results rely on the computation of $\Tors \Omega^{O\an{7}}_8(\C P^1; \xi^1) \cong \Z/4$ in Lemma \ref{lem:Omega_8(CP1)} below, where $\xi^1$ denotes the (unique up to isomorphism) non-trivial stable bundle over $\C P^1 = S^2$. Note that if $\xi$ is a stable bundle over $\C P^\infty$ with $w_2(\xi) \neq 0$, then its restriction to $\C P^1$ is isomorphic to $\xi^1$.

\subsection{The computation of $\Tors \Omega^{O\an{7}}_8(\C P^1; \xi^1)$} \label{ss:Z/4}
We first establish the necessary background to state and prove Lemma \ref{lem:Omega_8(CP1)}.
Let $\SS^0$ denote the sphere spectrum and write $\SS^k$ for
the $k$-fold suspension of $\SS^0$.
We let $\eta \colon \SS^1 \to \SS^0$ denote the generator of the $1$-stem $\pi^s_1 \cong \Z/2$,
and $C_\eta$ the cofibre of $\eta$. 
Since $\xi^1$ is the non-trivial stable bundle over $\C P^1$, the Thom spectrum of $\xi^1$ is given by 
$\Th(\xi^1) \simeq C_\eta$,
and the Pontryagin-Thom map for $\Omega^{O\an{7}}_*(\C P^1; \xi^1)$ is an isomorphism
\[ \PT \colon \Omega^{O\an{7}}_*(\C P^1; \xi^1) \to \pi_*(MO\an{8} \wedge \Th(\xi^1))
\cong \pi_*(MO\an{8} \wedge C_\eta), \]
where $\wedge$ denotes the smash product.
Smashing the cofibration $\SS^0 \to C_\eta \to \SS^2$ with $MO\an{8}$ and taking
homotopy groups, we obtain the long exact sequence
\begin{equation} \label{eq:MO(C_eta)}
 \ldots \to \pi_7(MO\an{8}) \xra{\eta_*} \pi_8(MO\an{8}) \to 
\pi_8(MO\an{8}) \wedge C_\eta) \to \pi_6(MO\an{8}) \xra{\eta_*} \pi_7(MO\an{8}) \to \ldots~.
\end{equation}

We shall need some basic facts about the low-dimensional string bordism groups 
$\Omega^{O\an{7}}_* \cong \pi_*(MO\an{8})$
and the natural forgetful map $F \colon \Omega^{\fr}_* \to \Omega^{O\an{7}}_*$.
These facts can be deduced from results of Giambalvo \cite{G}, and we also give a direct proof below.

\begin{lemma}[{Cf.\ Giambalvo \cite{G}}] \label{lem:pi_*(MO8)}
The natural map $F \colon \Omega^{\fr}_* \to \Omega^{O\an{7}}_*$ satisfies:
\begin{compactenum}[(a)]
\item $\Omega^{O\an{7}}_6 \cong \Z/2$ and $F \colon \Omega^{\fr}_6 \to \Omega^{O\an{7}}_6$ is an isomorphism;
\item $\Omega^{O\an{7}}_7 \cong 0$;
\item $\Omega^{O\an{7}}_8 \cong \Z/2 \oplus \Z$
and $F \colon \Omega^{\fr}_8 \to \Omega^{O\an{7}}_8$ has image $\Z/2$
and kernel the image of $J$-homomorphism 
$J_8 \colon \pi_8(SO) \to \pi^s_8 \cong \Omega^{\fr}_8 \cong (\Z/2)^2$. 
\end{compactenum}
\end{lemma}

\begin{proof}
Under the Pontryagin-Thom isomorphism, the map $F \colon \Omega^{\fr}_* \to \Omega^{O\an{7}}_*$ corresponds to the map on homotopy groups induced by the inclusion of the Thom cell $\SS^0 \to MO\an{8}$. To compute this map, we first replace $MO\an{8}$ with a simpler spectrum. Let $f_H \colon S^8 \to BO\an{8}$ represent a generator of $\pi_8(BO\an{8}) \cong \Z$ and let $\zeta_H$ be the stable vector bundle over $S^8$ classified by $f_H$. 

The Thom spectrum of any vector bundle over an $m$-sphere is the cofibre of a map $\SS^{m-1} \to \SS^0$, and it was Milnor \cite[Lemma 1]{M} who first observed that this map is obtained by applying the stable $J$-homomorphism to the clutching function of the bundle. Hence the Thom spectrum of $\zeta_H$ is $\Th(\zeta_H) \simeq C_{\bar \sigma}$, where $\bar \sigma \colon \SS^7 \to \SS^0$ is given by applying the $J$-homomorphism to the clutching function of $\zeta_H$, which generates $\pi_7(SO\an{7}) = \pi_7(SO)$. 

By construction, $f_H$ induces an isomorphism on $\pi_8$. We have $\pi_9(S^8) \cong \pi^s_1 \cong \Z_2$ and $\pi_{10}(S^8) \cong \pi^s_2 \cong \Z_2$, generated by $\eta$ and $\eta^2$ respectively, so by Bott periodicity and \cite[Proof of Example 12.5]{A}, $f_H$ also induces isomorphisms on $\pi_9$ and $\pi_{10}$. Hence $f_H$ is $10$-connected, and so the induced map of Thom spectra
\[ \Th(\zeta_H) \simeq C_{\bar \sigma} \to MO\an{8} \]
is also $10$-connected. 
Hence in dimensions $\ast \leq 9$ the map $F \colon \Omega^{\fr}_* \to \Omega^{O\an{7}}_*$ is isomorphic to the map on homotopy groups induced by the inclusion $\SS^0 \to C_{\bar \sigma}$. 

The cofibration $\SS^0 \to C_{\bar \sigma} \to \SS^8$ leads to a long exact sequence
\[ 
 \ldots \to \pi^s_1 \xra{\bar \sigma_*} \pi^s_8 \to \pi_8(C_{\bar \sigma}) \to \pi^s_0 \xra{\bar \sigma_*} \pi^s_7 \to
 \pi^s_7(C_{\bar \sigma}) \to 0 \to \pi^s_6 \to \pi_6(C_{\bar \sigma}) \to 0 \to
  \ldots~.
\]
We see immediately that $\pi^s_6 \to \pi_6(C_{\bar \sigma})$ is an isomorphism,
and so $F \colon \Omega^{\fr}_6 \to \Omega^{O\an{7}}_6$ is an isomorphism.  
Since $\Omega^{\fr}_6 \cong \pi^s_6 \cong \Z/2$ (where the last isomorphism is given in \cite[Ch.\ XIV]{T}), this proves Part (a).
For Part (b), we use that 
$J_7 \colon \pi_7(SO) \to \pi^s_7$ is onto
by Adams \cite[Example 7.17]{A}, and so $\bar \sigma$ generates $\pi^s_7$.  Hence $\bar \sigma_* \colon \pi^s_0 \rightarrow \pi^s_7$ is surjective, which proves Part (b).
For Part (c), 
since $\pi^s_0 \cong \Z$ and $\pi^s_7$ is finite, $\Ker\bigl(\bar \sigma_* \colon \pi^s_0 \rightarrow \pi^s_7 \bigr) \cong \Z$.  We also have 
$\Im \bigl(\bar \sigma_* \colon \pi^s_1 \rightarrow \pi^s_8 \bigr) = \left< \eta \bar \sigma \right> = \Im(J_8)$ (using \cite[Proof of Example 12.15]{A}). 
By Toda's calculations \cite[Ch.\ XIV]{T}, $\pi^s_8 \cong (\Z/2)^2$ with 
$\eta \ol \sigma \neq 0$, and this finishes the proof of Part (c).
\end{proof}

From the exact sequence \eqref{eq:MO(C_eta)} and Lemma \ref{lem:pi_*(MO8)}\,(b) we deduce that
there is a short exact sequence
\begin{equation} \label{eq:pi_8(MO(C_eta))}
0 \to \Omega^{O\an{7}}_8 \to \Omega^{O\an{7}}_8(\C P^1; \xi^1) \to \Omega^{O\an{7}}_6 \to 0.
\end{equation}
Noting that $\Omega^{\fr}_6 \cong \Omega^{O\an{7}}_6 \cong \Z/2$ is detected by the
Arf invariant, it is easy to see that the homomorphism $\Omega^{O\an{7}}_8(\C P^1; \xi^1) \to \Omega^{O\an{7}}_6$ can be identified with the codimension-$2$ Arf invariant
\[ A_{\C P^1} \colon \Omega^{O\an{7}}_8(\C P^1; \xi^1) \to \Z/2, \]
which is defined by making a normal map $(g, \bar g) \colon M \to S^2$ transverse
to a point $\ast \in S^2$ and taking the Arf invariant of the resulting $6$-manifold
$g^{-1}(\ast)$, which is canonically framed.

\begin{lemma} \label{lem:Omega_8(CP1)}
There is a non-split short exact sequence of abelian groups
\[ 0 \to \Theta_8 \to \Tors \Omega^{O\an{7}}_8(\C P^1; \xi^1) \xra{A_{\C P^1}} \Z/2 \to 0. \]
In particular $\Tors \Omega^{O\an{7}}_8(\C P^1; \xi^1) \cong \Z/4$. 
\end{lemma}

\begin{proof}
There is a natural forgetful map 
$F^1 \colon \Omega^{\fr}_8(\C P^1; \xi^1) \to \Omega^{O\an{7}}_8(\C P^1; \xi^1)$
and the exact sequence of \eqref{eq:pi_8(MO(C_eta))} forms part of the following 
commutative diagram:
\begin{equation} \label{eq:Omega_8}
\xymatrix{
\Omega^{\fr}_7 \ar[r]^{\eta_*} &
\Omega^{\fr}_8 \ar[d]^{F} \ar[r] &
\Omega^{\fr}_8(\C P^1; \xi^1) \ar[d]^{F^1} \ar[r]^-{A^{\fr}_{\C P^1}}   &
\Z/2 \ar[d]^{=} \ar[r]^{\eta_*} &
\Omega^{\fr}_7 \\
0 \ar[r] &
\Omega^{O\an{7}}_8 \ar[r] &
\Omega^{O\an{7}}_8(\C P^1; \xi^1) \ar[r]^-{A_{\C P^1}}   &
\Z/2 \ar[r] &
0
}
\end{equation}
Here $A^{\fr}_{\C P^1} \colon \Omega^{\fr}_8(\C P^1; \xi^1) \to \Z/2$ is a codimension-$2$
Arf invariant, which is defined analogously to the codimension-$2$ Arf invariant on
$\Omega^{O\an{7}}_8(\C P^1; \xi^1)$.
We shall first compute $\Omega^{\fr}_8(\C P^1; \xi^1)$ and we do this via 
the Pontryagin-Thom isomorphism
\[ \Omega^{\fr}_8(\C P^1; \xi^1) \cong \pi^s_8(C_\eta). \]
The cofibration $\SS^0 \to C_\eta \to \SS^2$ leads to the following long exact sequence (showing in particular that the top row of diagram \eqref{eq:Omega_8} is also exact):
\[ \ldots \to \pi^s_7 \xra{\eta_*} \pi^s_8 \to \pi_8(C_\eta) \to \pi^s_6 \xra{\eta_*} \pi^s_7 \to \ldots \]
From Toda's calculations \cite[Ch.\ XIV]{T}, we have $\pi^s_6 \cong \Z/2(\nu^2)$, $\pi^s_7 \cong \Z/240(\sigma)$,
$\pi^s_8 \cong \Z/2(\eta \sigma) \oplus \Z/2(\epsilon)$, where $\nu \in \pi^s_3$ 
is a generator and $\eta \nu \in \pi^s_4= \{0\}$.  It follows that $\eta_* \colon \pi^s_6 \to \pi^s_7$ is the zero map
and that there is a short exact sequence
\begin{equation} \label{eq:C_eta}
0 \to \Z/2([\epsilon]) \to \pi_8(C_\eta) \to \Z/2 \to 0,
\end{equation}
where $[\epsilon] \in \pi^s_8/\eta_*(\pi^s_7)$ denotes the equivalence class of $\epsilon$.
By Lemma \ref{lem:extension_and_Toda} from the Appendix, the extension 
\eqref{eq:C_eta} is determined by the Toda bracket
$$ \an{\eta, \nu^2, 2} \subset \pi^s_8.$$
By \cite[Proposition 3.4 ii]{T}, there is a Jacobi identity for Toda brackets,
\[ 0 \in \an{\eta, \nu^2, 2} + \an{2, \eta, \nu^2} + \an{\nu^2, 2, \eta},   \]
where we have ignored signs since all the Toda brackets consist of elements of order two or one.
Now by \cite[Proposition 1.2]{T}, $\an{2, \eta, \nu^2} \subseteq \an{2, \eta, \nu}\nu$.  Since
$\an{2, \eta, \nu} \subset \pi_5^s =\{0\}, \an{2, \eta, \nu} \nu = \{0\}$ and so $\an{2, \eta, \nu^2} = \{0\}$.
By \cite[p.~189]{T}, $\an{\nu^2, 2, \eta} = \{\epsilon, \epsilon + \eta \sigma\}$.
It follows that $\an{\eta, \nu^2, 2} = \{\epsilon, \epsilon + \eta \sigma\}$ is non-trivial and maps to 
the generator $[\epsilon] \in \pi^s_8/\eta_*(\pi^s_7)$.
Applying Lemma \ref{lem:extension_and_Toda}, we deduce that the extension
\eqref{eq:C_eta} is non-trivial and hence is isomorphic to the extension
\[ 0 \to \Z/2 \to \Z/4 \to \Z/2 \to 0. \]
The above shows that $\Omega^{\fr}_8(\C P^1; \xi^1) \cong \Z/4$. 

In diagram \eqref{eq:Omega_8} we can replace the top row with the short exact sequence \eqref{eq:C_eta} (noting that, as we saw in the proof of Lemma~\ref{lem:pi_*(MO8)}, $\Im(J_8) = \eta_*(\pi^s_7)$ and $\coker(J_8) = \Theta_8$) and restrict the bottom row to the torsion subgroups, to get the following commutative diagram: 
\[
\xymatrix{
0 \ar[r] & \Theta_8 \ar[d]^{F} \ar[r] & 
\Omega^{\fr}_8(\C P^1; \xi^1) \ar[d]^{F^1} \ar[r]^-{A^{\fr}_{\C P^1}} & 
\Z/2 \ar[d]^{=} \ar[r]^(0.55){\eta_*} & 0 \\
0 \ar[r] & \Tors \Omega^{O\an{7}}_8 \ar[r] & \Tors \Omega^{O\an{7}}_8(\C P^1; \xi^1) \ar[r]^-{A_{\C P^1}} & \Z/2 \ar[r] & 0
}
\]
We check that the bottom row is exact. The map $A_{\C P^1}$ is surjective by the commutativity of the diagram, while exactness at $\Tors \Omega^{O\an{7}}_8$ and $\Tors \Omega^{O\an{7}}_8(\C P^1; \xi^1)$ follows from the exactness of the original sequence. Now by Lemma~\ref{lem:pi_*(MO8)}\,(c) the map 
$F \colon \Theta_8 \rightarrow \Tors \Omega^{O\an{7}}_8$ is an isomorphism. 
So by the $5$-lemma, the homomorphism 
$F^1 \colon \Omega^{\fr}_8(\C P^1; \xi^1) \rightarrow \Tors \Omega^{O\an{7}}_8(\C P^1; \xi^1)$ is also an isomorphism, which completes the proof.
\end{proof}

\subsection{The non-spin case with $v_4(X_4(\dd)) = 0$} \label{ss:even1}

Let $X_4(\dd)$ be a non-spin complete intersection with $v_4(X_4(\dd)) = 0$. 
We will prove that $X_4(\dd)$ is $\Theta$-rigid.  As explained in Section \ref{ss:surgery}, 
it is enough to show that the canonical homomorphism 
$i_0 \colon \Theta_8 \cong \Tors \Omega^{O\an{7}}_8 \rightarrow \Omega^{O\an{7}}_8(\C P^{\infty}; \xi_4(\dd))$ 
is trivial.
We will exploit the fact that $i_0$ factors through the group 
$\Tors \Omega^{O\an{7}}_8({\C}P^1; \xi_4(\dd) \big| _{{\C}P^1})$.

\begin{prop} \label{prop:i_CP1}
Let $\xi$ be a stable bundle over ${\C}P^{\infty}$ such that $w_2(\xi) \neq 0$ and $w_4(\xi) = 0$. 
Then the natural map $i_0 \colon \Theta_8 \to \Omega^{O\an{7}}_8(\C P^{\infty}; \xi)$ is trivial.
\end{prop}

\begin{proof}
By \cite[Section 2.2]{F-K} we have $\Omega^{O\an{7}}_8(\C P^{\infty}, \ast;\xi) \cong \Z$. From the exactness of the sequence 
\[
\ldots \to \Omega^{O\an{7}}_8 \xra{j_0} 
\Omega^{O\an{7}}_8(\C P^{\infty};\xi) \xra{} 
\Omega^{O\an{7}}_8(\C P^{\infty}, \ast;\xi) \xra{} \ldots
\] 
we deduce that the image of $j_0$ contains the torsion subgroup of $\Omega^{O\an{7}}_8(\C P^{\infty};\xi)$. 
The signature defines non-trivial homomorphisms $\Omega^{O\an{7}}_8 \rightarrow \Z$ and 
$\Omega^{O\an{7}}_8(\C P^{\infty};\xi) \rightarrow \Z$ which commute with $j_0$. 
By Lemma \ref{lem:pi_*(MO8)}\,(c), $\Omega^{O\an{7}}_8 \cong \Z \oplus \Z/2$ and so $j_0$ is rationally injective.  Therefore its restriction to $\Tors \Omega^{O\an{7}}_8 \cong \Theta_8$ is surjective onto $\Tors \Omega^{O\an{7}}_8(\C P^{\infty};\xi)$. Therefore if $\Omega^{O\an{7}}_8(\C P^{\infty};\xi)$ has a non-trivial torsion element, then it has order $2$. 

Let $i_{\C P^1*} \colon \Tors \Omega^{O\an{7}}_8({\C}P^1; \xi^1) \rightarrow 
\Tors \Omega^{O\an{7}}_8(\C P^{\infty};\xi)$ denote the homomorphism induced by the inclusion ${\C}P^1 \rightarrow {\C}P^{\infty}$. 
By Lemma \ref{lem:Omega_8(CP1)}, $\Tors \Omega^{O\an{7}}_8(\C P^1; \xi^1) \cong \Z/4$ and if $a$ denotes a generator, then $\Sigmaex$ represents $2a$. Therefore $i_0(\Sigmaex) = i_{\C P^1*}(2a) = 2i_{\C P^1*}(a)=0$. 
\end{proof}

\begin{thm} \label{thm:rigid}
If $X_4(\dd)$ is a non-spin complete intersection with $v_4(X_4(\dd)) = 0$ and $\dd \neq \{ 2, 2 \}$, 
then $X_4(\dd)$ and $X_4(\dd) \sharp \Sigmaex$ are diffeomorphic. 
\end{thm}

\begin{proof}
In this case $w_2(\xi_4(\dd)) \neq 0$ and $w_4(\xi_4(\dd)) = 0$ (see Proposition \ref{prop:SW-classes}), so by Proposition \ref{prop:i_CP1} the canonical homomorphism $i_0 \colon \Theta_8 \to \Omega^{O\an{7}}_8(\C P^{\infty};\xi_4(\dd))$ is trivial. Therefore $X_4(\dd)$ and $X_4(\dd) \sharp \Sigmaex$ admit bordant normal $3$-smoothings over 
$(B_4, \xi_4(\dd) \times \gamma_{BO\an{8}})$.
By Proposition \ref{prop:classification_general}, $X_4(\dd)$ and $X_4(\dd) \sharp \Sigmaex$ are diffeomorphic. 
\end{proof}

\subsection{The non-spin case with $v_4(X_4(\dd)) \neq 0$ and even total degree} \label{ss:even2}
Now we suppose that $v_4(X_4(\dd)) \neq 0$ and the total degree $d$ is even. 
By Remark \ref{rem:p-mod-4} this implies that $\nu_2(d) \geq 4$ (where $\nu_2(d)$ denotes the exponent of $2$ in the prime factorisation of $d$). We will apply the Adams filtration argument of Kreck and Traving \cite[Section 8]{Kr}.
If $\xi$ is a stable vector bundle over $\C P^\infty$,
we use the Pontryagin-Thom isomorphism to identify the groups
$\Omega^{O\an{7}}_8(\C P^\infty; \xi) = \pi_8 \bigl( MO\an{8} \wedge \Th(\xi) \bigr)$
and hence their torsion subgroups.  In this way we obtain an Adams filtration on 
$\Tors\Omega^{O\an{7}}_8(\C P^\infty; \xi)$.
Recall that a map $\SS^i \to \mathbb{E}$ representing a torsion class in 
$\pi_i(\mathbb{E})$, the $i^{\text{th}}$ homotopy group of a spectrum $\mathbb{E}$,
has Adams filtration $\geq k$ if it can be factored as a composition of $k$ maps,
each of which is trivial on homology with $\Z/2$ coefficients.

We will need the following improvement of Kreck and Traving's vanishing result in dimension $8$.

\begin{lem} \label{lem:adams}
Let $\xi$ be a stable bundle over ${\C}P^{\infty}$ such that $w_2(\xi) \neq 0$, $w_4(\xi) \neq 0$ and the homomorphism $i_0 \colon \Theta_8 \to \Omega^{O\an{7}}_8({\C}P^{\infty}; \xi)$ is injective. Then the only element of $\Tors \Omega^{O\an{7}}_8({\C}P^{\infty};\xi)$ with Adams filtration $4$ or higher is the trivial element.  
\end{lem}

\begin{proof}
Consider the exact sequence 
\[
\ldots \to \Omega^{O\an{7}}_8 \xra{j_0} 
\Omega^{O\an{7}}_8(\C P^{\infty};\xi) \xra{} 
\Omega^{O\an{7}}_8(\C P^{\infty}, \ast;\xi) \xra{} \ldots~.
\] 
Fang and Klaus \cite[Section 2.4]{F-K} proved that 
$\Omega^{O\an{7}}_8(\C P^{\infty}, \ast;\xi) \cong \Z \oplus \Z/2$, 
where the $\Z/2$ summand is detected
by the codimension-$2$ Arf invariant.
Hence we have the following commutative diagram between exact sequences (the bottom sequence is exact, because $j_0$ is rationally injective, as in the proof of Proposition \ref{prop:i_CP1}):
\[
\xymatrix{
\Theta_8 \ar[d]^{\cong} \ar[r] &
\Tors \Omega^{O\an{7}}_8(\C P^1; \xi^1) \ar[d]^{i_{\C P^1*}} \ar[r]^(0.725){A_{\C P^1}} &
\Z/2 \ar[d]^= \\
\Tors \Omega^{O\an{7}}_8 \ar[r]^-{i_0} &
\Tors \Omega^{O\an{7}}_8(\C P^\infty; \xi) \ar[r]^(0.7){A} &
\Z/2
}
\]
Moreover, the top sequence is short exact by Lemma \ref{lem:Omega_8(CP1)}. The bottom sequence is also short exact (the surjectivity of $A$ follows from the commutativity of the diagram and the injectivity of $i_0$ was assumed). 
It follows that $i_{\C P^1*} \colon \Tors \Omega^{O\an{7}}_8(\C P^1; \xi^1) \to \Tors \Omega^{O\an{7}}_8(\C P^{\infty};\xi)$ is an isomorphism.

Now we choose a generator $a \in \Tors \Omega^{O\an{7}}_8(\C P^{\infty}; \xi) \cong \Z/4$ and let $[f_a] \in \pi_8(MO\an{8} \wedge \Th(\xi))$ represent the image of $a$ under the Pontryagin-Thom isomorphism. 
By \cite[p.\ 144]{F-K}, the image of $[f_a]$ in the group $\pi_8(MO\an{8} \wedge (\Th(\xi)/\SS^0))$ has Adams filtration $2$. Since the Adams filtration cannot decrease under composition, $[f_a]$ has Adams filtration $\leq 2$. 
Therefore $2a$, corresponding to $2[f_a]$ under the Pontryagin-Thom isomorphism, has Adams filtration $\leq 3$. Since $2a$ is a multiple of every non-zero element of $\Tors \Omega^{O\an{7}}_8(\C P^{\infty}; \xi) \cong \Z/4$, the only element with Adams filtration $4$ or higher is the trivial element.
\end{proof}

\begin{proposition} \label{prop:K-T_improved}
Let $X_4(\dd)$ and $X_4(\dd')$ be non-spin complete intersections with $SD_n(\dd) = SD_n(\dd')$ and 
$v_4(X_4(\dd)) \neq 0$. If the total degree $d$ is even, then $X_4(\dd)$ and $X_4(\dd')$ are diffeomorphic.
\end{proposition}

\begin{proof}
By Theorem \ref{thm:FK} there is a homotopy sphere $\Sigma \in \Theta_8$ such that $X_4(\dd) \approx X_4(\dd') \sharp \Sigma$. By Proposition \ref{prop:SW-classes} we have $w_2(\xi_4(\dd)) \neq 0$ and $w_4(\xi_4(\dd)) \neq 0$. Again we consider the natural map $i_0 \colon \Theta_8 \to \Omega^{O\an{7}}_8(\C P^{\infty}; \xi_4(\dd))$. 

If $i_0$ is zero, then $X_4(\dd')$ and $X_4(\dd') \sharp \Sigma$ are diffeomorphic by the same argument as in the proof of Theorem \ref{thm:rigid}. Hence $X_4(\dd)$ and $X_4(\dd')$ are diffeomorphic.

Now suppose that $i_0 \colon \Theta_8 \to \Omega^{O\an{7}}_8(\C P^{\infty}; \xi_4(\dd))$ is non-zero (hence injective). The arguments of Kreck and Traving \cite[Section 8]{Kr} show that $X_4(\dd)$ and $X_4(\dd')$ admit 
normal $3$-smoothings over $(B_4, \xi_4(\dd) \times \gamma_{BO\an{8}})$ whose bordism classes differ by a torsion element of Adams filtration $\nu_2(d)$ or higher. Since $d$ is even, we have $\nu_2(d) \geq 4$ (see Remark \ref{rem:p-mod-4}), so by Lemma \ref{lem:adams} any such torsion element is trivial. Hence $X_4(\dd)$ and $X_4(\dd')$ admit bordant normal $3$-smoothings over $(B_4, \xi_4(\dd) \times \gamma_{BO\an{8}})$ and so by Proposition \ref{prop:classification_general}, $X_4(\dd)$ and $X_4(\dd')$ are diffeomorphic. 
\end{proof}


\section{The case of odd total degree} 
\label{s:non-spin2}

It remains then to consider the case where the total degree $d$ is odd.
Note that in general this case is not $\Theta$-rigid as the following theorem of Kasilingam 
shows.

\begin{theorem}[{\cite[Remark 2.6\,(1)]{Ka}}] \label{thm:K}
$\C P^4$ is not diffeomorphic to $\C P^4 \sharp \Sigmaex$.
\end{theorem}

To prove the Sullivan Conjecture for $X_4(\dd)$ when $d$ is odd, we find a new
way to compare normal bordism classes for $X_4(\dd)$ and $X_4(\dd')$, which
is one of the main achievements of this paper.
In particular, we believe that introducing the Hambleton-Madsen \cite{H-M}
theory of degree-$d$ normal invariants will provide a new perspective on the
Sullivan Conjecture in all dimensions.

\subsection{Degree-$r$ normal maps and their normal invariants} \label{ss:drnm}
In this subsection we review the surgery classification of bordism classes of degree-$r$
normal maps for any integer $r$.  Our treatment follows Hambleton and 
Madsen \cite{H-M} but 
with minor modifications to suit our setting.
We will assume that all manifolds and all bundles are oriented and that all bundle maps are orientation
preserving.
We also choose to work with stable normal bundles in the source 
of normal maps, as opposed to stable tangent bundles,
and for simplicity, we only formulate the statements in the special case when the target
space of a degree-$r$ normal map is a closed smooth connected oriented $m$-manifold $P$.

\begin{defin}
Let $M$ and $P$ be closed smooth oriented $m$-manifolds and assume that $P$ is connected.
For $r \in \Z$, a degree-$r$ normal map $(f, \bar{f}) \colon M \to P$ is a map of stable vector bundles
\[
\xymatrix{
\nu_M \ar[r]^{\bar{f}} \ar[d]&
\xi \ar[d] \\
M \ar[r]^{f} &
P}
\]
from the stable normal bundle of $M$ to some stable vector bundle over $P$ such that 
$f \colon M \to P$ has degree $r$.
\end{defin}

When $r = \pm 1$, then $\xi$ is a vector bundle reduction of the Spivak normal fibration of $P$, but in general this only holds away from $r$.
Normal bordism of degree-$r$ normal maps is defined analogously to normal bordism of degree-$1$ normal maps \cite[Proposition 10.2]{Wa2}: the normal maps $(f, \bar{f}) \colon (M, \nu_M) \to (P,\xi)$ and $(f', \bar{f}') \colon (M', \nu_{M'}) \to (P,\xi')$ are normally bordant if there is an isomorphism $\alpha \colon \xi' \rightarrow \xi$ and a bordism between $(f, \bar{f})$ and $(f', \alpha \circ \bar{f}')$ over $(P,\xi)$. 

\begin{defin}
We denote the set of normal bordism classes of degree-$r$ normal maps to $P$ by $\NN_r(P)$,
where the superscript ``$+$'' indicates that we are working in the oriented setting.
\end{defin}

For a fixed $\xi$, let $\Omega^{\fr}_m(P; \xi)_r$ denote the subset of $\Omega^{\fr}_m(P; \xi)$ consisting of bordism classes whose representatives have degree $r$. The group of stable bundle automorphisms of $\xi$, $\Aut(\xi)$, acts on $\Omega^{\fr}_m(P; \xi)_r$ by post-composition. Let $\NN_r(P,\xi) \subseteq \NN_r(P)$ denote the subset of normal bordism classes that are representable by normal maps to $(P,\xi)$. Then we have a canonical bijection
\[
\NN_r(P,\xi) \equiv \Omega^{\fr}_m(P; \xi)_r / \Aut(\xi).
\]
Moreover, 
\[
\NN_r(P) = \bigsqcup_{[\xi]} \, \NN_r(P,\xi) \text{\,,}
\]
where we take the union over the isomorphism classes of stable bundles over $P$ which admit degree-$r$ normal maps.  To distinguish degree-$r$ normal bordism classes from usual bordism classes we use 

\begin{notation} \label{not:bordism}
We shall denote the bordism class of $(f, \bar f)$ in $\Omega^{\fr}_m(P; \xi)_r$ by
$[f, \bar f]_\xi$ and in $\NN_r(P)$ by $[f, \bar f]$.
\end{notation}

As in the degree-$1$ case, 
the computation of $\NN_r(P)$ proceeds via fibrewise degree-$r$ maps between vector bundles. 
Recall that for a bundle $\zeta$, the total space is denoted by $E\zeta$, the disc bundle is $D\zeta$, 
the sphere bundle is $S\zeta$ and the projection is $\pi_{\zeta}$. 
For a space $Y$ with oriented vector bundles $\zeta$ and $\theta$ of the same rank over $Y$, we consider fibrewise maps 
\[
\xymatrix{
S\zeta \ar[rr]^g \ar[dr] &&
S\theta \ar[ld] \\
& Y}
\]
between the associated sphere bundles, where the restriction of $g$ to each fibre has degree $r$.  
Given a fibrewise degree-$r_1$ map $g_1 \colon S\zeta_1 \to S\theta_1$ and a fibrewise
degree-$r_2$ map $g_2 \colon S\zeta_2 \to S\theta_2$, their fibrewise join is a fibrewise
degree-$r_1r_2$ map $g_1 \ast g_2 \colon S(\zeta_1 \oplus \zeta_2) \to S(\theta_1 \oplus \theta_2)$
between the spheres bundles of the Whitney sums of the original bundles.
An isomorphism between two fibrewise degree-$r$ maps, 
$g_i \colon S\zeta_i \to S\theta_i$, $i =0, 1$, is a pair of 
vector bundle isomorphisms $\alpha \colon \zeta_0 \to \zeta_1$ and 
$\beta \colon \theta_0 \to \theta_1$ such that the following diagram commutes
up to fibre homotopy over $Y$,
\[
\xymatrix{
S\zeta_0 \ar[r]^{g_0} \ar[d]_{S\alpha}&
S\theta_0 \ar[d]^{S\beta} \\
S\zeta_1 \ar[r]^{g_1} \ar[r] &
S\theta_1,}
\]
where $S\alpha$ and $S\beta$ are the induced maps of sphere bundles.

\begin{defin}
Two fibrewise degree-$r$ maps are equivalent if they become isomorphic after fibrewise join with the restriction of a vector bundle isomorphism (i.e.\ stabilisation), and we define
\[
\FF_r(Y) := \{ g \colon S\zeta \to S\theta \} / \! \sim
\]
to be the set of equivalence classes of fibrewise degree-$r$ maps of vector bundles over $Y$. The equivalence class of $g$ is denoted by $[g]$. 
\end{defin}

We now review how taking the transverse inverse image of the zero section
is used to define a map $T \colon \FF_r(P) \to \NN_r(P)$.
If $g \colon S\zeta \rightarrow S\theta$ represents an element 
$[g] \in \FF_r(P)$, then we can extend it to a fibre-preserving map $f \colon D\zeta \rightarrow D\theta$ that is transverse to the zero section $P \subset D\theta$. We set $M := f^{-1}(P)$ and $f_M := f \big| _M$.  Since $g$ has degree $r$, the map $f_M \colon M \rightarrow P$ has degree $r$ too. The map $f$ determines a bundle map $\bar{f}_0 \colon \nu(M \rightarrow D\zeta) \rightarrow \nu(P \rightarrow D\theta) \cong \theta$ over $f_M$. We have 
\[
\nu_M \cong \nu(M \rightarrow D\zeta) \oplus \nu_{D\zeta} \big| _M \cong \nu(M \rightarrow D\zeta) \oplus \bigl( \pi_{\zeta} \big| _M \bigr)^*(\nu_P \ominus \zeta) = \nu(M \rightarrow D\zeta) \oplus f_M^*(\nu_P \ominus \zeta).
\]
By adding the canonical map $f_M^*(\nu_P \ominus \zeta) \rightarrow \nu_P \ominus \zeta$ to $\bar{f}_0$, we get a bundle map
\[
\bar{f}_M \colon \nu_M \rightarrow \theta \oplus \nu_P \ominus \zeta
\]
over $f_M$. Then we define $T([g]) := \bigl[ f_M, \bar{f}_M \bigr]$.
For the case when $P$ is a smooth manifold (which we have assumed for simplicity),
the following theorem is 
the oriented version of a foundational result of 
Hambleton and Madsen on degree-$r$ normal maps (the proof of Hambleton and Madsen applies verbatim in the oriented setting).

\begin{theorem}[{cf.\ \cite[Theorem 2.2]{H-M}}]
\label{thm:T}
The map $T \colon \FF_r(P) \to \NN_r(P)$ is a well-defined bijection. \qed
\end{theorem}

\begin{remark} \label{rem:T}
In \cite{H-M} the source manifold $M$ is defined as the inverse image under $g$ of a section of $S\theta$. The construction above can be seen as a special case of this via stabilisation, as we now explain.
Let $\ul \R := (\R \times P \to P)$ denote the trivial rank-$1$ bundle over $P$ and let $\ul S^0 := S\,\ul \R$.
If $\zeta = \zeta' \oplus \ul \R$, $\theta = \theta' \oplus \ul \R$ and $g = g' \ast \Id_{\ul{S^0}} \colon S\zeta = S\zeta' \ast \ul{S^0} \to S\theta = S\theta' \ast \ul{S^0}$ for some $\zeta'$, $\theta'$ and $g' \colon S\zeta' \to S\theta'$, then $S\zeta = D_+\zeta' \cup_{S\zeta'} D_-\zeta'$ (where $D_{\pm}\zeta'$ are two copies of the disc bundle $D\zeta'$), $S\theta = D_+\theta' \cup_{S\theta'} D_-\theta'$, $f' := g \big| _{D_+\zeta'} \colon D_+\zeta' \rightarrow D_+\theta'$ is an extension of $g'$, and the zero section of $D_+\theta'$ coincides with the section $1 \times P \subset S^0 \times P = \ul{S^0} \subset S\theta' \ast \ul{S^0} = S\theta$ of $S\theta$.
\end{remark}

In order to apply Theorem \ref{thm:T} we need to be able to compute $\FF_r(P)$. The assignment $Y \mapsto \FF_r(Y)$ is a homotopy functor from the category of spaces to the category of sets. By Brown representability \cite{Br3}, this functor (restricted to CW-complexes) is represented by a classifying space. 

\begin{defin}
The classifying space of the functor $\FF_r$ is denoted by $(\qsnmo)_r$,
and the canonical bijection from $\FF_r(Y)$ to $[Y, (\qsnmo)_r]$ is denoted by  
\[
\Br \colon \FF_r(Y) \rightarrow [Y, (\qsnmo)_r].
\]
\end{defin}

\begin{remark}
Hambleton and Madsen \cite[\S1]{H-M} and Brumfiel and Madsen \cite[\S 4]{B-M} allow orientation-reversing bundle isomorphisms when they define the equivalence relation on fibrewise degree-$r$ maps. They denote the corresponding classifying space by $(QS^0/O)_r$ for $r \geq 0$. The forgetful map induces a map $(\qsnmo)_r \to (QS^0/O)_{|r|}$ of classifying spaces, which is a homotopy equivalence when $r \neq 0$ and a non-trivial $(\Z/2)$-covering when $r = 0$ (see \cite[Proposition 4.3]{B-M}).

When $r = 1$, we may identify $(\qsnmo)_1 = G/O$, where $G/O$ is the homotopy fibre of the canonical map $BSO \to BSG$, the forgetful map from the classifying space of stable vector bundles to the classifying space of stable spherical fibrations.
\end{remark}

The equivalence $\Br$ and Theorem \ref{thm:T} combine to give the following important

\begin{defin} \label{def:eta1}
Let $\eta \colon \NN_r(P) \rightarrow [P, (\qsnmo)_r]$ denote the composition $\Br \circ T^{-1}$. 
For a degree-$r$ normal map $(f, \bar{f}) \colon M \to P$, 
the homotopy class $\eta([f, \bar{f}]) \in [P, (\qsnmo)_r]$ is called the 
{\em normal invariant} of $(f, \bar{f})$. 
\end{defin}

We shall need two classes of examples of fibrewise degree-$r$ maps. The {\em trivial degree-$r$ map} of rank $k$ (well-defined up to fibre homotopy) is $h \times \Id_Y \colon S^{k-1} \times Y \rightarrow S^{k-1} \times Y$ for some degree-$r$ map $h \colon S^{k-1} \to S^{k-1}$.
For the second class of degree-$r$ maps, $\zeta$ has real rank $2$,
and we regard $\zeta$ as a complex line bundle over $Y$.
Setting $\theta := \zeta^r$ to be the $r$-fold complex tensor product of $\zeta$ with itself,
we have the {\em canonical degree-$r$ map}
\[ t_r(\zeta) \colon S\zeta \to S\zeta^r, 
\quad v \mapsto v^r = v \otimes v \otimes \ldots  \otimes v.\]
For the classification of complete intersections the universal examples of
such maps, where 
$Y = {\C}P^n$ or ${\C}P^\infty$ and $\zeta = \gamma \big| _{\C P^n}$ or $\gamma$,
will play a central role. 

\begin{defin} \label{def:eta2}
For a $k$-tuple of integers $\ul r = (r_1, \ldots, r_k)$ with $r = r_1r_2 \ldots r_k$ set
\[
\begin{aligned}
& & \eta_n(\ul r) &:= \Br([t_{r_1}(\gamma\big| _{\C P^n}) \ast \ldots \ast t_{r_k}(\gamma\big| _{\C P^n})]) \!\!\!\!\!\!& &\in [\C P^n, (\qsnmo)_r] \\
\text{and} \quad & & \eta_\infty(\ul r) &:= \Br([t_{r_1}(\gamma) \ast \ldots \ast t_{r_k}(\gamma)]) & &\in [\C P^\infty, (\qsnmo)_r].
\end{aligned}
\]
\end{defin}
\noindent
The notation in Definition \ref{def:eta2} is designed to match Theorem \ref{thm:eta_of_X}, which states that the complete intersection $X_n(\dd)$ admits a degree-$d$ normal map $(f_n(\dd), \bar f_n(\dd)) \colon X_n(\dd) \to \C P^n$ such that $\eta([f_n(\dd), \bar f_n(\dd)]) = \eta_n(\dd)$.

\subsection{The space $(\qsnmo)_r$ and connected sums} \label{ss:action} 
In order to apply Theorem \ref{thm:T} we will need to
make computations with the set of normal invariants $[P, (\qsnmo)_r]$.
For this we need information about the space $(\qsnmo)_r$, and 
we first adapt
the discussion of the related space $(QS^0/O)_r$ from Brumfiel and Madsen \cite[\S 4]{B-M}
to the oriented setting.
When $r = 1$, the space $(\qsnmo)_1 = G/O$ has been extensively studied.
In general, Brumfiel and Madsen \cite[Proposition 4.3]{B-M}
showed that there is a fibration sequence
\begin{equation} \label{eq:fib-qsnmo}
QS^0_r \xra{~i_r~} (\qsnmo)_r \xra{~\delta_r~} BSO,
\end{equation}
where the map $\delta_r \colon (\qsnmo)_r \to BSO$ classifies taking the formal difference of the source and target vector bundles of a fibrewise degree-$r$ map, $QS^0_r$ is the space of stable degree-$r$ self maps of the sphere, which classifies fibrewise degree-$r$ self-maps of trivialised vector bundles and
$i_r \colon QS^0_r \to (\qsnmo)_r$ classifies forgetting that the bundles are trivialised.
The space $QS^0_1$ is often denoted $SG$.
There is also a map $\cdot r \colon G/O \to (\qsnmo)_r$, which classifies taking fibrewise join with the trivial degree-$r$ map and which fits into the following map of fibration sequences
\begin{equation} \label{eq:fib-qsnmo2}
\xymatrix{
SG \ar[r]^-{i_1} \ar[d]_{\cdot r} & 
G/O \ar[r]^-{\delta} \ar[d]_{\cdot r} &
BSO \ar[d]^{\Id_{BSO}} \\
QS^0_r \ar[r]^-{i_r} &
(\qsnmo)_r \ar[r]^-{\delta_r} &
BSO,
}
\end{equation}
where $\cdot r \colon SG = QS^0_1 \to QS^0_r$ is the map obtained by composition
with a fixed map of degree $r$ and $\delta := \delta_1$ is the canonical map.

Since $QS^0 := \bigsqcup_{r \in \Z} QS^0_r$, the space of stable self maps of the sphere, is a grouplike $H$-space (with the ``loop sum" operation, which induces the addition on $\pi_0^s$), 
its connected components are all homotopy equivalent. So $\pi_i(QS^0_r) \cong \pi_i(SG) = \pi^s_i$ for all $r$ and $i$, and under this identification $\pi_i(\cdot r) \colon \pi_i(SG) \to \pi_i(QS^0_r)$ is multiplication by $r$. 
Therefore when we invert the primes dividing $r$, the map $\cdot r$ becomes a weak homotopy equivalence and hence a homotopy equivalence (see \cite[Proposition 4.6]{B-M}). Combining this with the commutative diagram of \eqref{eq:fib-qsnmo2}, we get

\begin{prop}[cf.\ {\cite[Proposition 4.6]{B-M}}] \label{prop:B-M}
If $r \neq 0$, the map $\cdot r \colon G/O \to (\qsnmo)_r$ induces a homotopy equivalence 
$(\cdot r)[1/r] \colon (G/O)[1/r] \simeq (\qsnmo)_r[1/r]$ such that
$\delta_r[1/r] \circ (\cdot r)[1/r] \simeq \delta[1/r]$. 
\qed
\end{prop}

Here $X[1/r]$ denotes the localisation of a space $X$ obtained by inverting the primes dividing $r$ (see \cite[Chapter 2]{Su}), and for a map $f \colon X \rightarrow Y$, $f[1/r] \colon X[1/r] \rightarrow Y[1/r]$ denotes the induced map. 

We now consider the effect of taking connected sum with a framed manifold in the source of a normal map. 
For this, we will tacitly assume that all manifolds have basepoints, the bundles considered over these basepoints have been trivialised (and for a fibrewise degree-$r$ map, the map over the basepoint is identified with some fixed degree-$r$ map between spheres) and that the connected sum operation is carried out at discs which have the basepoints on their boundaries so that the connected sum is itself based.

Suppose that $(f, \bar f) \colon M \to P$ is an $m$-dimensional degree-$r$ normal map and 
$[N, F] \in \Omega^{\fr}_m$. Let $f_N \colon N \rightarrow P$ be the constant map and $\bar{f}_N$ be the bundle map over $f_N$ corresponding to the framing $F$.  
We can assume that $(f, \bar f)$ is constant over a small $m$-disc $D^m \subset M$, and by taking
connected sum in the source and extending with the constant map, we obtain a degree-$r$ normal map
$(f \sharp f_N, \bar{f} \sharp \bar{f}_N) \colon M \sharp N \rightarrow P$.
Connected sum defines a natural operation
\[ 
\sharp \colon \NN_r(P) \times \Omega^\fr_m \to \NN_r(P),
\quad ([f, \bar f], [N, F]) \mapsto [f \sharp f_N, \bar f \sharp \bar f_N],
\]
and we will explain how to determine $\eta([f \sharp f_N, \bar f \sharp \bar f_N])$ in terms of $\eta([f, \bar f])$ and $[N, F]$.
To do this, we need to define an appropriate normal invariant of $[N, F]$. 
Define the homomorphism
\[ \eta^\fr_r \colon \Omega^\fr_m \to \pi_m(QS^0_r) \]
to be the composition $\Omega^{\fr}_m \xra{\PT} \pi^s_m \xra{\Ad} \pi_m(QS^0_0) \xra{\ls_{r*}} \pi_m(QS^0_r)$, where $\PT$ denotes the Pontryagin-Thom isomorphism, $\Ad$ is defined via the adjoint map, $\ls_r$ is given by taking the loop sum with a fixed degree-$r$ map and $\ls_{r*}$ is the induced map on $\pi_m$.

For any connected, oriented $m$-manifold $P$ and topological space $X$, connected sum of maps defines an action $\sharp \colon [P,X] \times \pi_m(X) \rightarrow [P,X]$ of $\pi_m(X)$ on $[P,X]$, where we can and do assume that maps are constant at the same value in a neighbourhood of the basepoints. This action is natural in $X$, i.e.\ a map $f \colon X \rightarrow Y$ determines a commutative diagram
\begin{equation} \label{eq:nat}
\xymatrix{
[P ,X] \times \pi_m(X) \ar[r]^-{\sharp} \ar[d]_-{f_* \times f_*} & [P, X] \ar[d]^-{f_*} \\
[P, Y] \times \pi_m(Y) \ar[r]^-{\sharp} & [P, Y].
}
\end{equation}

\begin{lemma} \label{lem:cs}
Let $(f, \bar f) \colon M \to P$ be an $m$-dimensional degree-$r$ normal map and $[N, F] \in \Omega^{\fr}_m$. The normal invariant of $(f \sharp f_N, \bar{f} \sharp \bar{f}_N) \colon M \sharp N \rightarrow P$ is given by
\[ 
\eta([f \sharp f_N, \bar{f} \sharp \bar{f}_N]) = \eta([f, \bar f]) \, \sharp\,  i_{r*}(\eta^\fr_r([N, F])) \in [P,(\qsnmo)_r],
\] 
where $i_r \colon QS^0_r \to (\qsnmo)_r$ is the inclusion of the fibre appearing in \eqref{eq:fib-qsnmo}.
\end{lemma}

\begin{proof}
We will prove that there is a commutative diagram
\[
\xymatrix{
\NN_r(P) \times \Omega^{\fr}_m \ar[d]_-{\sharp} & &
\FF_r(P) \times \pi^s_m \ar[ll]_-{T \times \PT^{-1}} \ar[r]^-{\Id \times \alpha} \ar[d] & 
\FF_r(P) \times \FF_r(S^m) \ar[r]^-{\Br \times \Br} \ar[d]^-{\sharp} & 
[P,(\qsnmo)_r] \times \pi_m((\qsnmo)_r) \ar[d]^-{\sharp} \\
\NN_r(P) & &
\FF_r(P) \ar[ll]_-{T} \ar@{=}[r] &
\FF_r(P) \ar[r]^-{\Br} & 
[P,(\qsnmo)_r],
}
\]
where $\alpha$ is the composition of $\Ad \colon \pi^s_m \cong \pi_m(QS^0_0)$, $\ls_{r*} \colon \pi_m(QS^0_0) \cong \pi_m(QS^0_r)$ and the canonical map $\pi_m(QS^0_r) \rightarrow \FF_r(S^m)$ which sends a homotopy class to the adjoint fibrewise degree-$r$ map between trivialised bundles over $S^m$. The second and third vertical maps will be defined in the course of the proof below. 

The commutativity of the diagram above suffices to prove the lemma, because $i_{r*}$ is the composition of the canonical map $\pi_m(QS^0_r) \rightarrow \FF_r(S^m)$ and $\Br$ (recall that $QS^0_r$ is the classifying space of fibrewise degree-$r$ maps between trivialised bundles and $i_r$ classifies forgetting that the bundles are trivialised). 

First we define the map $\FF_r(P) \times \pi^s_m \rightarrow \FF_r(P)$ and show that the first square commutes. Let $g \colon S\zeta \to S\theta$ represent an element of $\FF_r(P)$, where $\zeta$ and $\theta$ are bundles of rank $k{+}1 \gg m$ of the form $\zeta = \zeta' \oplus \ul \R$ and $\theta = \theta' \oplus \ul \R$. We define the sections $s_{\pm} \colon P \to S\theta$ by $s_{\pm}(x) = (\pm 1, x) \in S^0 \times P \subset S(\theta' \oplus \ul \R)$, and assume that $g$ is  transverse to $s_+(P)$, so that $T([g])$ is represented by 
$g^{-1}(s_+(P)) \rightarrow s_+(P) \approx P$ (covered by the appropriate bundle map); see Remark \ref{rem:T}.

Let $D^m \subset P$ be a small embedded disc, then $\zeta|_{D^m} = \R^{k+1} \times D^m$ and $\theta_{D^m} =  \R^{k+1} \times D^m$ are uniquely trivialised (up to homotopy). After a fibre homotopy of $g$ we may assume that $g \big| _{S\zeta|_{D^m}}$ is trivial, i.e.\ $g \big| _{S\zeta|_{D^m}} = h \times \Id_{D^m}$ for some degree-$r$ map $h \colon S^k \rightarrow S^k$, and that for some small disc $D^k \subset S^k$ the map $h \big| _{D^k}$ is constant with value $-1 \in S^0 \subset S^{k-1} \ast S^0 \approx S^k$. Then $g(D^k \times D^m) = s_-(D^m)$, so $(g \big| _{D^k \times D^m})^{-1}(s_+(P))$ is empty.  

The adjoint of $g \big| _{D^k \times D^m}$ is the constant map in $\Map((D^m, S^{m-1}), (\Map((D^k, S^{k-1}), (S^k, -1)), c_{-1}))$, where $c_{-1}$ denotes the constant map with value $-1$. Identifying $\Map((D^k, S^{k-1}), (S^k, -1)) = \Omega^k S^k$ we have the isomorphism
\[ 
\Phi \colon \pi_0 \bigl( \Map((D^m, S^{m-1}), (\Map((D^k, S^{k-1}), (S^k, -1)), c_{-1})) \bigr) \cong \pi_m(\Omega^k S^k) \cong \pi_{m+k}(S^k) = \pi^s_m.
\]
For $a \in \pi^s_m$, let $u_a \colon D^k \times D^m \to S^k \times D^m$ be the adjoint of a representative of $\Phi^{-1}(a)$. We define the fibrewise degree-$r$ map $g_a \colon S\zeta \to S\theta$ by 
\[
g_a(v) = 
\begin{cases}
u_a(v) & \text{if $v \in D^k \times D^m$} \\
g(v) & \text{if $v \in S\zeta \setminus \interior(D^k \times D^m)$.}
\end{cases}
\]
The map $g_a$ is well-defined and continuous, because $g$ and $u_a$ agree on $\partial(D^k \times D^m)$. We define the map $\FF_r(P) \times \pi^s_m \rightarrow \FF_r(P)$ by $([g],a) \mapsto [g_a]$. 

We may assume that $u_a$ is transverse to $1 \times D^m$, then $u_a^{-1}(1 \times D^m)$ is the image of $a$ under the Pontryagin construction. Moreover, $g_a$ is transverse to $s_+(P)$ and $g_a^{-1}(s_+(P)) = g^{-1}(s_+(P)) \sqcup u_a^{-1}(1 \times D^m)$. This is bordant to $g^{-1}(s_+(P)) \sharp u_a^{-1}(1 \times D^m)$, showing that $T([g_a])=T([g]) \sharp \PT^{-1}(a)$. Thus the first square commutes.

Next we define the map $\sharp \colon \FF_r(P) \times \FF_r(S^m) \rightarrow \FF_r(P)$ and consider the second square. Given fibrewise degree-$r$ maps over $P$ and $S^m$, we fix embeddings $D^m \rightarrow P$ and $D^m \rightarrow S^m$ at the basepoints and identify the restrictions of both maps with $h \times \Id_{D^m}$. Then we can take their connected sum over $P \sharp S^m \approx P$.

Now let $[g] \in \FF_r(P)$ and $a \in \pi^s_m$ as before. We let $g^0 := h \times \Id_{S^m} \colon S^k \times S^m \rightarrow S^k \times S^m$ be a trivial degree-$r$ map and construct $g^0_a \colon S^k \times S^m \rightarrow S^k \times S^m$ from $u_a$ and $g^0$, analogously to $g_a$. Then $[g_a] = [g] \sharp [g^0_a]$ (the connected sum $P \sharp S^m$ is formed using the embedding $D^m \rightarrow P$ that was previously used to construct $g_a$ and an embedding $D^m \rightarrow S^m$ whose image is the complement of the one used to construct $g^0_a$). We can extend $u_a$ with the projection $D^k \times (S^m \setminus D^m) \rightarrow \{ -1 \} \times (S^m \setminus D^m) \subset S^k \times (S^m \setminus D^m)$ to get $\tilde{u}_a \colon D^k \times S^m \rightarrow S^k \times S^m$. This is again the adjoint of $a$, and also of the corresponding element $a' \in \pi_m(QS^0_0)$. Since $g^0(y,x)= (h(y),x)$ for every $(y,x) \in S^k \times S^m$ and $h(y)=-1$ if $y \in D^k$, we have 
\[
g^0_a(v) = 
\begin{cases}
u_a(v) & \text{if $v \in D^k \times D^m$} \\
g_0(v) & \text{if $v \in S^k \times S^m \setminus \interior(D^k \times D^m)$}
\end{cases}
= 
\begin{cases}
\tilde{u}_a(v) & \text{if $v \in D^k \times S^m$} \\
g_0(v)=(h(y),x) & \text{if $v=(y,x) \in (S^k \setminus \interior D^k) \times S^m$.}
\end{cases}
\] 
Since $\ls_r \colon QS^0 \rightarrow QS^0_r$ is defined by taking loop sum with $h$, the map $g^0_a$ is the adjoint of $\ls_{r*}(a')$. Therefore $[g_a]= [g] \sharp \alpha(a)$, proving that the second square commutes.

Finally we consider the third square. For any space $X$, the action $\sharp \colon [P,X] \times \pi_m(X) \rightarrow [P,X]$ is equal to the composition 
\[
[P,X] \times \pi_m(X) \xra{\vee} [P \vee S^m,X] \xra{p^*} [P,X]
\] 
where $p \colon P \rightarrow P / S^{m-1} \approx P \vee S^m$ is the pinch map, collapsing the boundary of a small embedded $m$-disc $D^m \subset P$. Similarly, $\sharp \colon \FF_r(P) \times \FF_r(S^m) \rightarrow \FF_r(P)$ is the composition of $\vee \colon \FF_r(P) \times \FF_r(S^m) \rightarrow \FF_r(P \vee S^m)$ and 
$p^* \colon \FF_r(P \vee S^m) \rightarrow \FF_r(P)$. So the commutativity of the third square follows from the naturality of $\Br$.
\end{proof}

\subsection{Relative divisors}
In this subsection we give another description of the bijection $T \colon \FF_r(P) \to \NN_r(P)$ from Theorem \ref{thm:T} in terms of sections and divisors.
In order to relate fibre-preserving maps to sections we introduce the following notation.

\begin{defin} \label{def:bij}
Suppose that $\zeta$ and $\theta$ are vector bundles over some base space $Y$.

(a) For a fibre-preserving map $g \colon S\zeta \rightarrow S\theta$ we define a section $s_g \colon S\zeta \rightarrow \bigl( \pi_\zeta \big| _{S\zeta} \bigr)^*(S\theta) \subseteq S\zeta \times S\theta$ of the sphere bundle $\bigl( \pi_\zeta \big| _{S\zeta} \bigr)^*(S\theta)$, the pull-back of $S\theta$ via the projection $\pi_\zeta \big| _{S\zeta} \colon S\zeta \rightarrow Y$, by $s_g(x) = (x, g(x))$. The assignment $g \mapsto s_g$ is a bijection between fibre-preserving maps $S\zeta \rightarrow S\theta$ and sections of $\bigl( \pi_\zeta \big| _{S\zeta} \bigr)^*(S\theta)$.

(b) For a fibre-preserving map $f \colon D\zeta \rightarrow D\theta$ we define a section $s_f \colon D\zeta \rightarrow \bigl( \pi_\zeta \big| _{D\zeta} \bigr)^*(D\theta) \subseteq D\zeta \times D\theta$ by $s_f(x) = (x, f(x))$. The assignment $f \mapsto s_f$ is a bijection between fibre-preserving maps $D\zeta \rightarrow D\theta$ and sections of $\bigl( \pi_\zeta \big| _{D\zeta} \bigr)^*(D\theta)$. Moreover, $f$ is transverse to the zero section of $D\theta$ if and only if $s_f$ is transverse to the zero section of $\bigl( \pi_\zeta \big| _{D\zeta} \bigr)^*(D\theta)$. 

These two bijections are compatible in the sense that if $g$ is the restriction of some $f$, then $s_g$ is the restriction of $s_f$.
\end{defin}

Now suppose that $\tilde{\theta}$ is a rank-$k$ smooth vector bundle over a smooth manifold $V$ with boundary $\partial V$ and $s_{\partial} \colon \partial V \to S\tilde{\theta} \big| _{\partial V}$ is a section of $S\tilde{\theta} \big| _{\partial V}$ (hence a nowhere zero section of $E\tilde{\theta} \big| _{\partial V}$). 

\begin{defin}
If $s \colon V \to E\tilde{\theta}$ is a smooth section of $\tilde{\theta}$, which extends $s_{\partial}$, and which is transverse to the zero section, $s_0$, then we call
\[
Z(s) := s(V) \cap s_0(V) \subset s_0(V) \approx V
\]
a {\em divisor of $\tilde{\theta}$ relative to $s_{\partial}$}.
\end{defin}

\begin{rem} \label{rem:divisor}
We have $\nu_{Z(s)} \cong (\tilde{\theta} \oplus \nu_V) \big| _{Z(s)}$, because the normal bundle of the embedding $Z(s) \hookrightarrow V$ is given by $\nu(Z(s) \rightarrow V) \cong \tilde{\theta} \big| _{Z(s)}$ (in a tubular neighbourhood $U \approx D\nu(Z(s) \rightarrow V)$ of $Z(s)$ the section $s \big| _U$ corresponds to a fibre-preserving map $D\nu(Z(s) \rightarrow V) \rightarrow E\tilde{\theta} \big| _{Z(s)}$, and by transversality the restriction of its derivative is an isomorphism $\nu(Z(s) \rightarrow V) \cong \tilde{\theta} \big| _{Z(s)}$). 

Since the fibre of $E\tilde{\theta}$ is contractible, $s_{\partial}$ can always be extended to a (transverse) section $s$ and the extension is unique up to homotopy (rel $\partial V$). This also implies that the normal bordism class of the normal map
\[
\xymatrix{
\nu_{Z(s)} \ar[r] \ar[d] & \tilde{\theta} \oplus \nu_V \ar[d] \\
Z(s) \ar[r] & V
}
\]
is independent of the choice of $s$ (and it only depends on the homotopy class of $s_{\partial}$ as a nowhere zero section). 
\end{rem}

Suppose in addition that $V = D\zeta$ itself is the disc bundle of a rank-$k$ smooth vector bundle $\zeta$ over a closed smooth manifold $P$. Let $\theta = \tilde{\theta} \big| _{P}$ be the restriction of $\tilde{\theta}$, then $\tilde{\theta}$ can be identified with $\bigl( \pi_{\zeta} \big| _{D\zeta} \bigr)^*(\theta)$. 

Let $g \colon S\zeta \rightarrow S\theta$ be a fibrewise degree-$r$ map and $s_g$ the corresponding section (see Definition \ref{def:bij}). There exists a section $s \colon D\zeta \rightarrow \bigl( \pi_\zeta \big| _{D\zeta} \bigr)^*(D\theta)$ that extends $s_g$ and is transverse to the zero section. Let $p = \pi_{\zeta} \big| _{Z(s)} \colon Z(s) \rightarrow P$.

\begin{lem} \label{lem:div}
The map $p \colon Z(s) \rightarrow P$ has degree $r$ and it is covered by a bundle map $\bar{p} \colon \nu_{Z(s)} \rightarrow \theta \oplus \nu_P \ominus \zeta$ such that
\[ 
T([g]) = [p, \bar{p}] \in \NN_r(P) \text{\,.}
\]
\end{lem}

\begin{proof}
Using the bijection from Definition \ref{def:bij}\,(b) there is a fibre-preserving map $f \colon D\zeta \rightarrow D\theta$ such that $s = s_f$. This $f$ extends $g$ and it is transverse to the zero section, so it satisfies the conditions in the definition of $T$ (given before Theorem \ref{thm:T}). The manifold $M = f^{-1}(P)$ is then equal to $Z(s)$ and $f \big| _M = p$ (and it has degree $r$). We can choose $\bar{p} = \bar{f}_M$ and then $T([g]) = [p, \bar{p}]$.
\end{proof}

\subsection{The canonical degree-$d$ normal invariant of a complete intersection}
Consider a complete intersection $X_n(\dd)$. By cellular approximation the canonical embedding $i \colon X_n(\dd) \rightarrow \C P^{n+k}$ is homotopic to a map $f_n(\dd) \colon X_n(\dd) \rightarrow \C P^n$ and since $\C P^{n+k}$ 
has no $(2n{+}1)$-cells, $f_n(\dd)$ is well-defined up to homotopy. Since $i_*([X_n(\dd)])$ is $d$ times the preferred generator of $H_{2n}(\C P^n)$, $f_n(\dd)$ is a degree-$d$ map. 

The main result of this section is the computation of the normal invariant of a certain degree-$d$ normal map covering $f_n(\dd)$ in Theorem \ref{thm:eta_of_X} below. The importance of this calculation comes from the next lemma (which can be regarded as a variation of \cite[Proposition 10]{Kr}) and its application, Theorem \ref{thm:SC_via_cni}.

\begin{lem} \label{lem:SC_via_cni0}
Let $X_n(\dd)$ and $X_n(\dd')$ be complete intersections with
$\chi(X_n(\dd)) = \chi(X_n(\dd'))$ and the same total degree $d$.
Suppose that there are degree-$d$ normal maps
$(f, \bar{f}) \colon (X_n(\dd), \nu_{X_n(\dd)}) \rightarrow ({\C}P^n, \xi_n(\dd) \big| _{{\C}P^n})$ and $(f', \bar{f}') \colon (X_n(\dd'), \nu_{X_n(\dd')}) \rightarrow ({\C}P^n, \xi_n(\dd') \big| _{{\C}P^n})$ such that 
\[
[f, \bar{f}] = [f', \bar{f}'] \in \NN_d({\C}P^n) \text{\,.}
\] 
If $n \geq 3$, then $X_n(\dd)$ and $X_n(\dd')$ are diffeomorphic.
\end{lem}

\begin{proof}
Let $\xi = \xi_n(\dd)\big| _{\C P^n}$ and recall (see Notation \ref{not:bordism}) that $[g, \bar{g}]_\xi \in \Omega^{\fr}_{2n}({\C}P^n; \xi)_d$ denotes the element represented by a degree-$d$ normal map $(g, \bar{g})$ and the image of $[g, \bar g]_\xi$ in $\NN_d(\C P^n)$ is $[g, \bar g]$. By definition, the condition $[f, \bar{f}] = [f', \bar{f}']$ means that there is an isomorphism $\alpha \colon \xi_n(\dd') \big| _{{\C}P^n} \rightarrow \xi_n(\dd) \big| _{{\C}P^n} = \xi$ (which in particular implies that $SD_n(\dd) = SD_n(\dd')$)
such that
\[ 
[f, \bar f]_\xi = [f', \alpha \circ \bar f']_\xi \in \Omega^{\fr}_{2n}(\C P^n; \xi)_d \text{\,.}
\]
Now consider the composition
\[
\Omega^{\fr}_{2n}(\C P^n; \xi_n(\dd) \big| _{{\C}P^n})_d \rightarrow 
\Omega^{\fr}_{2n}(\C P^n; \xi_n(\dd) \big| _{{\C}P^n}) \to 
\Omega^{\fr}_{2n}(\C P^\infty; \xi_n(\dd)) \to 
\Omega^{O\an{n}}_{2n}(\C P^{\infty}; \xi_n(\dd)) \text{\,.}
\] 
We see that $X_n(\dd)$ and $X_n(\dd')$ admit bordant normal $(n{-}1)$-smoothings over $(B_n; \xi_n(\dd) \times \gamma_{BO\an{n{+}1}})$ and if $\dd \neq \{1\}, \{2\}$ or $\{2, 2\}$, then the lemma follows from Proposition \ref{prop:classification_general}. If $\dd = \{1\}, \{2\}$ or $\{2, 2\}$, then $SD_n(\dd) = SD_n(\dd')$ implies that $\dd' = \dd$. 
\end{proof}

\begin{thm} \label{thm:eta_of_X}
There is a bundle map $\bar f_n(\dd) \colon \nu_{X_n(\dd)} \rightarrow \xi_n(\dd) \big| _{\C P^n}$ over $f_n(\dd)$ such that  
\[
\eta([f_n(\dd), \bar f_n(\dd)]) = \eta_n(\dd) \in [\C P^n, (\qsnmo)_d] \text{\,.}
\]
(For the definitions of $\eta$ and $\eta_n(\dd)$ see Definitions \ref{def:eta1} and \ref{def:eta2}.)
\end{thm}

An immediate consequence of Theorem \ref{thm:eta_of_X}, the fact that $\eta$ is a bijection and Lemma \ref{lem:SC_via_cni0} is the following

\begin{thm} \label{thm:SC_via_cni}
Let $X_n(\dd)$ and $X_n(\dd')$ be complete intersections with the same total degree $d$ and the same Euler characteristic. If $n \geq 3$ and $\eta_n(\dd) = \eta_n(\dd') \in [\C P^n, (\qsnmo)_d]$, then $X_n(\dd)$ and $X_n(\dd')$ are diffeomorphic.
\hfill \qed
\end{thm}

\begin{proof}[Proof of Theorem \ref{thm:eta_of_X}]
Let $f^0 \colon D(k\gamma \big| _{\C P^n}) \rightarrow D(\gamma^{\dd} \big| _{\C P^n})$ denote the Whitney sum of the tensor power maps $D(\gamma \big| _{\C P^n}) \rightarrow D(\gamma^{d_i} \big| _{\C P^n})$ and let $g^0 \colon S(k\gamma \big| _{\C P^n}) \rightarrow S(\gamma^{\dd} \big| _{\C P^n})$ be its restriction to the sphere bundle. Hence, in the notation of Definition \ref{def:eta2}, $g^0 = t_{d_1}(\gamma \big|_{\C P^n}) \ast \ldots \ast t_{d_r}(\gamma \big|_{\C P^n})$, so $\Br([g^0]) = \eta_n(\dd)$ and we must prove the following: There is a map of stable vector bundles $\bar f_n(\dd) \colon \nu_{X_n(\dd)} \rightarrow \xi_n(\dd) \big| _{\C P^n}$ over $f_n(\dd)$ such that  
\[ 
T([g^0]) = [f_n(\dd), \bar f_n(\dd)] \in \NN_d(\C P^n) \text{\,.}
\]

First we describe a way of constructing a representative of the complete intersection $X_n(\dd)$ in an arbitrarily small neighbourhood of the subspace $\C P^n \subset \C P^{n+k}$. Let $[x_0, x_1, \ldots , x_{n+k}]$ be homogeneous coordinates on the ambient $\C P^{n+k}$. 
For $i = 1, 2, \ldots , k$, define $p_i^0 \in \C[x_0, x_1, \ldots , x_{n+k}]$ by $p_i^0(\ul x) = x_{n+i}^{d_i}$, where $\ul x = (x_0, x_1, \ldots , x_{n+k})$. Then 
\[
\bigl\{ [\ul x] \in \C P^{n+k} \mid p_1^0(\ul x) = p_2^0(\ul x) = \ldots = p_k^0(\ul x) = 0 \bigr\} 
= 
\bigl\{ [\ul x] \in \C P^{n+k} \mid x_{n+1} = x_{n+2} = \ldots = x_{n+k} = 0 \bigr\} = \C P^n \text{\,.}
\]
Note that if $d_i > 1$, then $p_i^0$ is singular at its zeros, so $\C P^n$ is not a representative of $X_n(\dd)$ unless $\dd = \{ 1, 1, \ldots , 1 \}$. However, by applying an arbitrarily small perturbation to the $p_i^0$ we can obtain new polynomials $p_i$ such that
\[
\bigl\{ [\ul x] \in \C P^{n+k} \mid p_1(\ul x) = p_2(\ul x) = \ldots = p_k(\ul x) = 0 \bigr\} = X_n(\dd)
\]
is a complete intersection and it is contained in the interior of a closed tubular neighbourhood $U$ of $\C P^n$ (we will fix a $U$ in Lemma \ref{lem:poly2} below). 

By Construction \ref{const:poly1} the polynomials $p_i^0$ and $p_i$ define sections of $\gamma^{d_i} \big| _{\C P^{n+k}}$. Therefore the tuples $(p_1^0, p_2^0, \ldots , p_k^0)$ and $(p_1, p_2, \ldots , p_k)$ define some sections $s^0$ and $s$ of $\gamma^{\dd} \big| _{\C P^{n+k}}$ (so the zero sets of $s^0$ and $s$ are $\C P^n$ and $X_n(\dd)$ respectively). Then we can assume that there is a homotopy $\C P^{n+k} \times I \rightarrow \gamma^{\dd} \big| _{\C P^{n+k}}$ of sections between $s^0$ and $s$ that is non-zero on $(\C P^{n+k} \setminus \interior U) \times I$. In particular, the restrictions of $s^0$ and $s$ are homotopic as non-zero sections over $\partial U$.

The normal bundle of $\C P^n$ in $\C P^{n+k}$ is $k\gamma \big| _{\C P^n}$, so $U$ can be identified with $D(k\gamma \big| _{\C P^n})$. Moreover, the projection $\pi_U \colon U \rightarrow \C P^n$ of $U$ is a deformation retraction, hence the bundle $\pi_U^* \bigl( \gamma^{\dd} \big| _{\C P^n} \bigr)$ is isomorphic to $\gamma^{\dd} \big| _U$. We will fix an identification and an isomorphism in Lemma \ref{lem:poly2}. With these identifications, the sections $s^0 \big| _U$ and $s \big| _U$ correspond to fibre-preserving maps $D(k\gamma \big| _{\C P^n}) \rightarrow D(\gamma^{\dd} \big| _{\C P^n})$ under the bijection of Definition \ref{def:bij}. Let $f \colon D(k\gamma \big| _{\C P^n}) \rightarrow D(\gamma^{\dd} \big| _{\C P^n})$ be the map such that $s_f = s \big| _U$. In Lemma \ref{lem:poly2} we prove that $s_{f^0} = s^0 \big| _U$. Let $g \colon S(k\gamma \big| _{\C P^n}) \rightarrow S(\gamma^{\dd} \big| _{\C P^n})$ be the restriction of $f$ (we can assume that it has values in the sphere bundle, because $f \big| _{S(k\gamma | _{\C P^n})}$ is nowhere zero, because $s \big| _{\partial U}$ is nowhere zero), then $s_g = s \big| _{\partial U}$. The restriction of $f^0$ is $g^0$, so $s_{g^0} = s^0 \big| _{\partial U}$. Since $s_{g^0} = s^0 \big| _{\partial U}$ and $s_g = s \big| _{\partial U}$ are homotopic as non-zero sections, $g^0$ and $g$ are fibre homotopic. The bijection $T$ is well-defined on fibre homotopy classes, so $T([g^0]) = T([g])$.

By construction $X_n(\dd) = Z(s)$ is a divisor of $\gamma^{\dd} \big| _{\C P^{n+k}}$ relative to $s \big| _{\partial U} = s_g$. Since $\pi_U$ is homotopic to the identity, we have $f_n(\dd) = \pi_U \big| _{X_n(\dd)}$ (up to homotopy). By Lemma \ref{lem:div} there is a bundle map $\bar f_n(\dd) \colon \nu_{X_n(\dd)} \rightarrow \gamma^{\dd} \oplus -(n{+}1)\gamma \ominus k\gamma \big| _{\C P^n} \cong \xi_n(\dd) \big| _{\C P^n}$ such that $T([g]) = [f_n(\dd), \bar f_n(\dd)] \in \NN_d(\C P^n)$. Therefore $T([g^0]) = [f_n(\dd), \bar f_n(\dd)]$. 
\end{proof}

\begin{remark}
There is a canonical bundle map $\nu_{X_n(\dd)} \rightarrow \xi_n(\dd) \big| _{{\C}P^{n+k}}$ over $i$, and hence over $f_n(\dd)$ (cf.\  Proposition \ref{prop:normal-b}), because there is a canonical isomorphism $\nu_{X_n(\dd)} \cong \nu(X_n(\dd) \rightarrow {\C}P^{n+k}) \oplus \nu_{{\C}P^{n+k}}$ and the normal bundle of a degree-$r$ hypersurface in ${\C}P^{n+k}$ is canonically isomorphic to the restriction of $\gamma^r$ (see Construction \ref{const:poly1} and Remark \ref{rem:divisor}). By following the definitions, we can see that the bundle map $\bar{f}_n(\dd)$ constructed in the proof of Theorem \ref{thm:eta_of_X} is equal to this canonical map (up to homotopy). 
\end{remark}

We used the following (well known) construction and the lemma below.

\begin{const} \label{const:poly1}
A homogeneous polynomial $q$ of degree $r$ in variables $x_0, x_1, \ldots , x_m$ determines a section of the bundle $\gamma^r \big| _{\C P^m}$ as follows.

If $r=1$, then the assignment $[x_0, x_1, \ldots , x_m] \mapsto [x_0, x_1, \ldots , x_m, q(x_0, x_1, \ldots , x_m)]$ is a well-defined map $\C P^m \rightarrow \C P^{m+1} \setminus [0, 0, \ldots , 0, 1]$. Since the map $\C P^{m+1} \setminus [0, 0, \ldots , 0, 1] \rightarrow \C P^m$, $[x_0, x_1, \ldots , x_{m+1}] \mapsto [x_0, x_1, \ldots , x_m]$ can be identified with the projection of the normal bundle of $\C P^m$ in $\C P^{m+1}$, which is isomorphic to $\gamma \big| _{\C P^m}$, we get that $q$ determines a section of $\gamma \big| _{\C P^m}$. So every linear monomial $x_i$ determines a section of $\gamma \big| _{\C P^m}$. If we have sections $s_1, s_2, \ldots , s_r$ of some vector bundle $\xi$, then their symmetric product $s_1s_2 \ldots s_r$ is a section of the symmetric power $\Sym^r(\xi)$ and if $\xi$ is a line bundle, then $\Sym^r(\xi) = \xi^r$. Therefore every degree-$r$ monomial and hence every degree-$r$ homogeneous polynomial determines a section of $\gamma^r \big| _{\C P^m}$.
\end{const}

\begin{lem} \label{lem:poly2}
We can identify $D(k\gamma \big| _{\C P^n})$ with a tubular neighbourhood $U$ of $\C P^n$ in $\C P^{n+k}$ and the bundle $\pi_U^* \bigl( \gamma^{\dd} \big| _{\C P^n} \bigr)$ with $\gamma^{\dd} \big| _U$ such that after these identifications the section $s_{f^0}$ corresponding to $f^0$ under the bijection of Definition \ref{def:bij}\,(b) is equal to $s^0 \big| _U$.
\end{lem}

\begin{proof}
First we will introduce ``coordinates" on the total space of $\gamma^r \big| _{\C P^m}$, then we will define $U$ and describe the necessary identifications, then we will show that $s_{f^0}$ (regarded as a section of $\gamma^{\dd} \big| _U$) is equal to $s^0 \big| _U$.

By Construction \ref{const:poly1} a pair $([\ul a], q)$ (where $[\ul a] = [a_0, a_1 , \ldots a_m] \in \C P^m$ and $q$ is a homogeneous polynomial of degree $r$ in variables $x_0, x_1, \ldots , x_m$) determines a point in $E(\gamma^r \big| _{\C P^m})$ (namely, the value of the section determined by $q$ over the point $[\ul a]$). Every point in $E(\gamma^r \big| _{\C P^m})$ can be described by such a pair and two pairs, $([\ul a], q)$ and $([\ul a], q')$, determine the same point if and only if $q(\ul a) = q'(\ul a)$. Similarly, if $q_i$ is a homogeneous polynomial of degree $d_i$, then a pair $([\ul a], (q_1, q_2, \ldots , q_k))$ determines a point in $E(\gamma^{\dd} \big| _{\C P^m})$. 

To simplify notation we will use the abbreviations $\ul a = (a_0, a_1, \ldots , a_n) \in \C^{n+1} \setminus \{ 0 \}$, $\ul b = (b_1, b_2, \ldots , b_k) \in \C^k$ and $\ul c = (c_0, c_1, \ldots , c_{n+k}) \in \C^{n+k+1} \setminus \{ 0 \}$. Also, $q_i$ and $r_i$ will always denote some homogeneous polynomials in variables $x_0, x_1, \ldots , x_n$ such that $q_i$ has degree $d_i$ and $r_i$ is linear.

The map $([\ul a], (r_1, r_2, \ldots , r_k)) \mapsto [\ul a, r_1(\ul a), r_2(\ul a), \ldots , r_k(\ul a)]$ is a homeomorphism between $E(k\gamma \big| _{\C P^n})$ and an open tubular neighbourhood of $\C P^n$ in $\C P^{n+k}$ (which is diffeomorphic to $\C P^{n+k} \setminus \C P^{k-1}$). We define $U$ to be the image of the disc bundle $D(k\gamma \big| _{\C P^n})$ under this map. Then this map identifies $D(k\gamma \big| _{\C P^n})$ with $U$.

Points of the subspace $E \bigl( \pi_U^* \bigl( \gamma^{\dd} \big| _{\C P^n} \bigr) \bigr) \subset U \times E \bigl( \gamma^{\dd} \big| _{\C P^n} \bigr)$ are of the form $([\ul a, \ul b], ([\ul a], (q_1, q_2, \ldots , q_k)))$. The map $([\ul a, \ul b], ([\ul a], (q_1, q_2, \ldots , q_k))) \mapsto ([\ul a, \ul b], (\bar{q}_1, \bar{q}_2, \ldots , \bar{q}_k)) \in E(\gamma^{\dd} \big| _{\C P^{n+k}})$ (where $\bar{q}_i$ is equal to $q_i$, but it is regarded as a polynomial in the variables $x_0, x_1, \ldots , x_{n+k}$) is an isomorphism between the bundles $\pi_U^* \bigl( \gamma^{\dd} \big| _{\C P^n} \bigr)$ and $\gamma^{\dd} \big| _U = (\gamma^{\dd} \big| _{\C P^{n+k}}) \big| _U$. 

By definition the section $s^0$ is the map $[\ul c] \mapsto ([\ul c], (p_1^0, p_2^0, \ldots , p_k^0))$. 

The map $f^0$ is given by the formula $([\ul a], (r_1, r_2, \ldots , r_k)) \mapsto ([\ul a], (r_1^{d_1}, r_2^{d_2}, \ldots , r_k^{d_k}))$.   After identifying $D(k\gamma \big| _{\C P^n})$ with $U$ the formula becomes $[\ul a, \ul b] \mapsto ([\ul a], (r_1^{d_1}, r_2^{d_2}, \ldots , r_k^{d_k}))$, where $r_i$ is chosen such that $r_i(\ul a) = b_i$. Therefore $s_{f^0}([\ul a, \ul b]) = ([\ul a, \ul b], ([\ul a], (r_1^{d_1}, r_2^{d_2}, \ldots , r_k^{d_k})))$ and this point is identified with $([\ul a, \ul b], (\bar{r}_1^{d_1}, \bar{r}_2^{d_2}, \ldots , \bar{r}_k^{d_k}))$. We have 
\[
\bar{r}_i^{d_i}(\ul a, \ul b) = r_i^{d_i}(\ul a) = b_i^{d_i} = p_i^0(\ul a, \ul b) \text{\,,}
\]
so $([\ul a, \ul b], (\bar{r}_1^{d_1}, \bar{r}_2^{d_2}, \ldots , \bar{r}_k^{d_k})) = ([\ul a, \ul b], (p_1^0, p_2^0, \ldots , p_k^0))$, therefore $s_{f^0} = s^0 \big| _U$.
\end{proof}

We conclude this section with a discussion of the bundle data $\bar f_n(\dd)$ in the canonical
normal invariant of $X_n(\dd)$.
Although the degree-$d$ normal map $(f_n(\dd), \bar f_n(\dd)) \colon X_n(\dd) \to \C P^n$ is canonically
constructed, so far we have not been able to characterise its homotopy class amongst all such
degree-$d$ normal maps. 
In particular, if there is a diffeomorphism $h \colon X_n(\dd) \to X_n(\dd')$,
then up to homotopy it induces a unique bundle map $\bar{h} \colon \nu_{X_n(\dd)} \to \nu_{X_n(\dd')}$
covering $h$ and in general we do not know whether $\bar{f}_n(\dd') \circ \bar{h}$ and $\bar f_n(\dd)$ are homotopic stable bundle maps.  
In this paper, we shall only need to address this question when $n = 4$ and $X_4(\dd)$ is non-spin.
In this case, the problem is solved via the following 

\begin{lemma} \label{lem:bundle_automorphism}
Let $X$ be a closed, connected non-spin $8$-manifold which is homotopy equivalent to a 
$CW$-complex with only even dimensional cells,
$\xi$ a stable vector bundle over $X$
and $\bar{g} \colon \xi \to \xi$ an orientation preserving stable bundle automorphism. 
Then $\bar{g}$ is fibre homotopic to the identity. 
\end{lemma}

\begin{proof}
By standard $K$-theoretic arguments (given for automorphisms of stable spherical fibrations in \cite[Lemma I.4.6]{Br1}), it is sufficient to prove that $[X, SO] = 0$. By \cite[Proposition 4C.1]{Ha}, we may assume that there is a homotopy equivalence
$X \simeq K \cup_f D^8$, where $K$ is a $6$-dimensional $CW$-complex with only even-dimensional cells 
and $f \colon S^7 \to K$ attaches a single $8$-cell.
Consider the Puppe sequence of the cofibration 
$K \to K \cup_f D^8 \to S^8$:
\[ [\Sigma K, SO] \xra{} \pi_8(SO) \to [K \cup_f D^8, SO] \xra{} [K, SO]. \]
Obstruction theory \cite[Corollary 4.73]{Ha} gives that $[K, SO] = 0$, since $\pi_{2i}(SO) = 0$ for $0 \leq 2i \leq 6$. So to prove that $[X, SO] \cong [K \cup_f D^8, SO]$ is trivial, it is enough to show that the map $[\Sigma K, SO] \to \pi_8(SO)$ is surjective.

The map $[\Sigma K, SO] \to \pi_8(SO)$ sends a homotopy class $[g] \in [\Sigma K, SO]$ to $[g \circ \Sigma f]$, where $\Sigma f \colon S^8 \to \Sigma K$ is the suspension of $f$ \cite[6.18 Ch.\ III]{Wh}. Since $K \cup_f D^8$ has no odd-dimensional cells, $H^1(X; \Z/2) = 0$, and so $X$ is orientable. Since $X$ is non-spin, $v_2(X) = w_2(X) \neq 0$ by \cite[Theorem 11.15]{M-S}, and so $Sq^2 \colon H^6(X; \Z/2) \to H^8(X; \Z/2)$ is non-zero. 
Let $K^{(4)} \subset K$ denote the $4$-skeleton of $K$, and let $c \colon K \to K/K^{(4)}$ be the collapse map, then $K/K^{(4)} = \vee_{i=1}^b S^6$ is a wedge of $6$-spheres. Since $H^6(S^7) \cong H^5(S^7) \cong H^8( \vee_{i=1}^b S^6) \cong H^7( \vee_{i=1}^b S^6) \cong 0$, we deduce that the functional Steenrod Square of $Sq^2$ applied to $c \circ f \colon S^7 \to \vee_{i=1}^b S^6$ is unambiguously defined and non-zero on $H^6(K/K^{(4)}; \Z/2)$; c.f.\ \cite[Ch.\ 16]{M-T}. Since $Sq^2$ is a stable operation, the functional Steenrod Square of $Sq^2$ applied to the suspension $\Sigma c \circ \Sigma f \colon S^8 \to \vee_{i=1}^b S^7$ is also non-zero on $H^7(\vee_{i=1}^b S^7; \Z/2)$, showing that $\Sigma c \circ \Sigma f$ is essential (see \cite[Ch.\ 16, Proposition 1]{M-T}). By Hilton's Theorem \cite[Theorem A]{Hi} and the computation of the $1$-stem \cite[Ch.\ XIV]{T}, $\pi_8\bigl( \vee_{i=1}^b S^7 \bigr) \cong \bigoplus_{i=1}^b \pi_8(S^7) \cong (\Z/2)^b$, and it follows that $\pr_j \circ (\Sigma c \circ \Sigma f) \colon S^8 \to S^7$ is essential for some $j \in \{1, \dots, b\}$, where $\pr_j \colon \vee_{i=1}^b S^7 \to S^7$ splits off the $j^{\text{th}}$ sphere in the wedge. Using \cite[Example 12.15]{A}, we see that precomposition with $\eta_7 \colon S^8 \to S^7$ induces a surjection $\pi_7(SO) \to \pi_8(SO)$, and so the composition 
\[ 
\xymatrix{
\pi_7(SO) \ar[r]^-{\pr_j^*} & [\Sigma \bigl(K/K^{(4)}\bigr), SO] \ar[r]^-{\Sigma c^*} & [\Sigma K, SO] \ar[r]^-{\Sigma f^*} & \pi_8(SO) 
}
\]
is onto.  
It follows that $[\Sigma K, SO] \to \pi_8(SO)$ 
is onto, completing the proof.
\end{proof}

\begin{corollary} \label{cor:bundle_map}
Let $(f_0, \bar{f}_0), (f_1, \bar{f}_1) \colon 
X_4(\dd)
\rightarrow ({\C}P^4, \xi_4(\dd) \big| _{{\C}P^4})$
be a pair of normal maps from a non-spin complete intersection $X_4(\dd)$
such that $f_0^*(x) = f_1^*(x)$. Then $(f_0, \bar{f}_0)$ and $(f_1, \bar{f}_1)$ are homotopic.
\end{corollary}

\begin{proof}
It follows from the assumption $f_0^*(x) = f_1^*(x)$ that $f_0$ and $f_1$ are homotopic as maps into 
${\C}P^{\infty} \simeq K(\Z,2)$. By cellular approximation they are also homotopic as maps into ${\C}P^4$, so we may assume that $f_0=f_1$. Then the bundle maps $\bar{f}_0, \bar{f}_1 \colon \nu_{X_4(\dd)} \to \xi_4(\dd) \big| _{{\C}P^4}$ differ by pre-composition with a bundle automorphism $\bar{g} \colon \nu_{X_4(\dd)} \to \nu_{X_4(\dd)}$. 
By Lemma \ref{lem:bundle_automorphism}, $\bar{g}$ is homotopic to the identity 
and so $\bar{f}_0$ and $\bar{f}_1$ are homotopic.
\end{proof}

\subsection{The Sullivan Conjecture in the case of odd total degree} 
Let $X_n(\dd)$ be a complete intersection. By Theorem \ref{thm:eta_of_X} there is a degree-$d$ normal map $(f_n(\dd), \bar f_n(\dd)) \colon X_n(\dd) \to \C P^n$ such that
\[ 
\eta(f_n(\dd), \bar f_n(\dd)) = \eta_n(\dd) \in [\C P^n, (\qsnmo)_d]. 
\]
Our goal is to show that if $d$ is odd and $SD_4(\dd) = SD_4(\dd')$, then $\eta_4(\dd) = \eta_4(\dd')$. 
Then Theorem \ref{thm:SC_via_cni} allows us to deduce the Sullivan Conjecture when $n = 4$ and $d$ is odd. To compare the normal invariants $\eta_4(\dd)$ and $\eta_4(\dd')$, we will apply results of Feshbach on the Segal Conjecture, using the fact that $\eta_n(\dd)$ is the restriction of $\eta_\infty(\dd) \colon \C P^{\infty} \to (\qsnmo)_d$.

Due to a combination of Theorem \ref{thm:FK} of Fang and Klaus and Proposition \ref{prop:B-M} of Brumfiel and Madsen, localising at the prime $2$ will prove to be an effective strategy when the total degree $d$ is odd.
For a simple space $Z$ and a prime $p$ we shall write $Z_{(p)}$ (and even $(Z)_{(p)}$ where necessary) 
for the $p$-localisation of $Z$.  Similarly, we write $A_{(p)}$ for the $p$-localisation of an abelian group $A$.
If $\varphi \in [Y, Z]$ is a homotopy class of maps 
from some other space $Y$ to $Z$,
we write $\varphi_{(p)} \in [Y, Z_{(p)}]$ for the homotopy class of 
the composition $Y \xra{\varphi} Z \to Z_{(p)}$, where $Z \to Z_{(p)}$ is the natural map.
Similarly, for $Z[1/p]$, the space obtained from $Z$ by inverting $p$,
we write $\varphi_{[1/p]} \in [Y, Z[1/p]]$ for the homotopy class of 
the composition $Y \xra{\varphi} Z \to Z[1/p]$, where $Z \to Z[1/p]$ is the natural map.

\begin{lem} \label{lem:local}
Let $n$ and $d$ be positive integers.
Suppose that $\varphi, \psi \in [{\C}P^n, (\qsnmo)_d]$ are homotopy classes such that 
$\varphi_{(p)} = \psi_{(p)}$ and $\varphi_{[1/p]} = \psi_{[1/p]}$ for some prime $p$.  Then $\varphi = \psi$.
\end{lem}

\begin{proof}
We partition the set of all primes into the sets
$l := \{p\}$ and $l' := \{ q \mid q \neq p\}$.
By \cite[(4) p.\ 41]{Su}, for any simple space $Z$ the natural maps $Z \to Z[1/p]$ and $Z \to Z_{(p)}$ fit into a fibre square
\[ 
\xymatrix{
Z \ar[r] \ar[d] &
Z[1/p] \ar[d] \\
Z_{(p)} \ar[r] &
Z_\Q,
}
\]
where $Z_\Q$ denotes the rationalisation of $Z$.
Hence there is a homotopy fibration
sequence $Z \to Z_{(p)} \times Z[1/p] \to Z_\Q$,
and for $Z = (\qsnmo)_d$
we have a homotopy fibration sequence
\[
(\qsnmo)_d \xra{~\ell_p~} ((\qsnmo)_d)_{(p)} \times ((\qsnmo)_d)[1/p]  \to ((\qsnmo)_d)_{\Q}.
\]
The Puppe sequence for homotopy classes of maps from $\C P^n$ into this fibration contains the 
exact sequence
\[
[{\C}P^n, \Omega ((\qsnmo)_d)_{\Q}] \to [{\C}P^n, (\qsnmo)_d] \xra{~\ell_{p*}~} 
[{\C}P^n, ((\qsnmo)_d)_{(p)}] \times [{\C}P^n, ((\qsnmo)_d)[1/p]],
\]
where $\ell_{p*}(\varphi) = (\varphi_{(p)}, \varphi_{[1/p]})$ and
$[{\C}P^n, \Omega ((\qsnmo)_d)_{\Q}]$ acts transitively on the fibres of $\ell_{p*}$.
We will show that $[{\C}P^n, \Omega ((\qsnmo)_d)_{\Q}] = 0$, which implies that $\ell_{p*}$ is injective
and proves the lemma.

Since $QS^0_d$ is connected with finite homotopy groups (see Section \ref{ss:action}), its rationalisation is contractible. So, by the rationalisation of the fibration sequence \eqref{eq:fib-qsnmo}, 
$((\qsnmo)_d)_{\Q} \simeq (BSO)_{\Q}$. 
It is well-known that there is an equivalence $(BSO)_{\Q} \simeq \prod_{i=1}^\infty K(\Q, 4i)$ 
(see e.g.\ \cite[(12) p.\ 42-3]{Su}), 
and so we have a chain of isomorphisms 
$[{\C}P^n, \Omega ((\qsnmo)_d)_{\Q}] \cong [\Sigma {\C}P^n, ((\qsnmo)_d)_{\Q}] \cong [\Sigma {\C}P^n, (BSO)_{\Q}] \cong \bigoplus_{i=1}^\infty H^{4i}(\Sigma {\C}P^n;\Q) \cong 0$.
\end{proof}

\begin{lem} \label{lem:2local}
Let $X_4(\dd)$ and $X_4(\dd')$ be 
non-spin complete intersections such that $SD_4(\dd) = SD_4(\dd')$. 
If $\eta_4(\dd)_{(2)} = \eta_4(\dd')_{(2)} \in [{\C}P^4, ((\qsnmo)_d)_{(2)}]$, then $\eta_4(\dd) = \eta_4(\dd')$.
\end{lem}

\begin{proof}
By Theorem \ref{thm:FK}, there is a homotopy $8$-sphere $\Sigma$ and a diffeomorphism $h \colon X_4(\dd) \approx X_4(\dd') \sharp \Sigma$. We may assume that $h$ preserves the cohomology class $x$ (see the proof of Proposition \ref{prop:SC-conv}) and hence the maps $f_4(\dd)$ and $f_4(\dd') \circ h$ are homotopic (as in the proof of Corollary \ref{cor:bundle_map}). 
Let $\bar{h} \colon \nu_{X_4(\dd)} \to \nu_{X_4(\dd') \sharp \Sigma}$ be the stable bundle map covering $h$, which is uniquely determined up to homotopy by the derivative of $h$. 
By Theorem \ref{thm:eta_of_X}, there are bundle maps $\bar f_4(\dd)$ and $\bar f_4(\dd')$ such that $\eta_4(\dd) = \eta([f_4(\dd), \bar f_4(\dd)])$ and $\eta_4(\dd') = \eta([f_4(\dd'), \bar f_4(\dd')])$. As in Section \ref{ss:action}, let $f_{\Sigma} \colon \Sigma \rightarrow {\C}P^4$ be the constant map, then the choice of an arbitrary framing of $\Sigma$ determines a bundle map $\bar{f}_{\Sigma}$ over $f_{\Sigma}$ and we get a normal map $(f_4(\dd') \sharp f_{\Sigma}, \bar{f}_4(\dd') \sharp \bar{f}_{\Sigma}) \colon X_4(\dd') \sharp \Sigma \rightarrow {\C}P^4$. 
As explained in Section \ref{ss:surgery}, since $SD_4(\dd) = SD_4(\dd')$, there is a stable bundle isomorphism $\alpha \colon \xi_4(\dd') \big| _{{\C}P^4} \to \xi_4(\dd) \big| _{{\C}P^4}$. We have the following diagram of stable bundle maps, which commutes by Corollary \ref{cor:bundle_map}:
\[
\xymatrix{
\nu_{X_4(\dd)} \ar[d]_{\bar{h}} \ar[rr]^-{\bar f_4(\dd)} & &
\xi_4(\dd) \big| _{{\C}P^4} \\
\nu_{X_4(\dd') \sharp \Sigma} \ar[rr]^-{\bar{f}_4(\dd') \sharp \bar f_\Sigma} & & 
\xi_4(\dd') \big| _{{\C}P^4} \ar[u]_\alpha 
}
\]
It follows that the degree-$d$ normal maps 
\[
(f_4(\dd),\bar f_4(\dd)) \colon X_4(\dd) \to \C P^4 
\quad \text{and} \quad 
(f_4(\dd') \sharp f_{\Sigma}, \bar{f}_4(\dd') \sharp \bar{f}_{\Sigma}) \colon X_4(\dd') \sharp \Sigma \to \C P^4
\]
represent the same element in $\NN_d(\C P^4)$, so $\eta([f_4(\dd), \bar f_4(\dd)]) = \eta([f_4(\dd') \sharp f_\Sigma, \bar{f}_4(\dd') \sharp \bar f_\Sigma])$. 
Therefore by Lemma \ref{lem:cs} we have 
\[ \eta_4(\dd) = \eta_4(\dd') \sharp [\psi] \]
for some $[\psi] \in (i_d)_*(\pi_8(QS^0_d))$.
By the naturality of $\sharp$ (see \eqref{eq:nat}) we have
$\eta_4(\dd)_{[1/2]} = \eta_4(\dd')_{[1/2]} \sharp [\psi_{[1/2]}]$.
Since $\pi_8(QS^0_d) \cong \pi^s_8$ and $\pi^s_8\cong \Z/2 \oplus \Z/2$ by \cite[Theorem 7.1]{T}, $[\psi] \in \pi_8((\qsnmo)_d)$ is $2$-torsion. 
This implies that $[\psi_{[1/2]}]=0$, and hence $\eta_4(\dd)_{[1/2]} = \eta_4(\dd')_{[1/2]}$.
We assumed that $\eta_4(\dd)_{(2)} = \eta_4(\dd')_{(2)}$ and so by 
Lemma \ref{lem:local}, $\eta_4(\dd) = \eta_4(\dd')$.
\end{proof}

From now on we assume that $d$ is odd. Then $(Z[1/d])_{(2)} \simeq Z_{(2)}$ for any simple space $Z$, so from the Brumfiel-Madsen equivalence $(G/O)[1/d] \simeq (\qsnmo)_d[1/d]$ of Proposition \ref{prop:B-M}, we deduce the existence of a homotopy equivalence
\[ \chi \colon ((\qsnmo)_d)_{(2)} \simeq (G/O)_{(2)} \]
such that $\delta_{(2)} \circ \chi = (\delta_d)_{(2)}$,
where $\delta_{(2)}$ and $(\delta_d)_{(2)}$ are the $2$-localisations the canonical maps $\delta$ 
and $\delta_d$ from \eqref{eq:fib-qsnmo2}. 
Moreover, by Sullivan's $2$-primary splitting theorem for $G/O$ \cite[Theorem 5.18]{M-M},
there is a homotopy equivalence
\begin{equation} \label{eq:SS}
 \phi \colon (G/O)_{(2)} \to (BSO)_{(2)} \times \coker J_{(2)},
\end{equation}
where the space $\coker J_{(2)}$ 
is defined in \cite[Definition 5.16]{M-M}
and the map $\phi$ is constructed in the proof of \cite[Theorem 5.18]{M-M}.
From the splitting of $(G/O)_{(2)}$ in \eqref{eq:SS} we obtain a projection map
\begin{equation*}
\pi \colon (G/O)_{(2)} \to \coker J_{(2)}.
\end{equation*}

The following result is contained in Feshbach's proof of \cite[Theorem 6]{Fe1},
where the arguments rely on work of Feshbach and Ravenel on the Segal Conjecture \cite{Fe2, R}.

\begin{theorem}[cf.\ {\cite[Proof of Theorem 6]{Fe1}}] \label{thm:Feshbach}
For any prime $p$, $[\C P^\infty, \coker J_{(p)}] \cong 0$.
\end{theorem}

\begin{proof}
The proof of \cite[Theorem 6]{Fe1} states that the stable cohomotopy group $\pi^0_s(\C P^\infty)$ is trivial, where $\pi^0_s(\C P^\infty) = [\C P^\infty, QS^0_0]$.
The natural map $QS^0_0 \to \prod_p (QS^0_0)_{(p)}$ from $QS^0_0$ to the product of its $p$-localisations, taken over all primes $p$, is a weak equivalence, because $QS^0_0$ is connected with finite homotopy groups $\pi_i(QS^0_0) \cong \pi_i^s$ (hence $\pi_i(QS^0_0) \cong \prod_p \pi_i(QS^0_0)_{(p)}$).
Now by Sullivan's splitting of $QS^0_1 \simeq QS^0_0$ \cite[Theorem 5.18]{M-M}, 
$(QS^0_0)_{(p)} \simeq \mathrm{im} J_{(p)} \times \coker J_{(p)}$ for a certain $p$-local
space $\mathrm{im} J_{(p)}$.
Therefore 
\[ 0 \cong [\C P^\infty, QS^0_0] \cong 
\prod_p \bigl( [\C P^\infty, \mathrm{im}J_{(p)}] \times [\C P^\infty, \coker J_{(p)}] \bigr) \]
and the theorem follows.
\end{proof}

As a consequence of Theorem \ref{thm:Feshbach} we have

\begin{corollary} \label{cor:mod_p_normal_inv}
If $d$ is odd, then $\pi_{*}(\chi_*(\eta_n(\dd)_{(2)})) = 
0 \in [\C P^n, \coker J_{(2)}]$ for all $n$.
\end{corollary}

\begin{proof}
Let $i \colon \C P^n \to \C P^\infty$ be the inclusion and
consider the following commutative diagram:
$$
\xymatrix{
[\C P^\infty, (\qsnmo)_d] \ar[d]_-{i^*} \ar[r] &
[\C P^\infty, ((\qsnmo)_d)_{(2)}] \ar[d]_-{i^*} \ar[r]^-{\chi_*} &
[\C P^\infty, (G/O)_{(2)}] \ar[d]^-{i^*} \ar[r]^-{\pi_{*}} &
[\C P^\infty, \coker J_{(2)}] \ar[d]^-{i^*} \\
[\C P^n, (\qsnmo)_d] \ar[r] &
[\C P^n, ((\qsnmo)_d)_{(2)}] \ar[r]^-{\chi_*} &
[\C P^n, (G/O)_{(2)}] \ar[r]^-{\pi_{*}} &
[\C P^n, \coker J_{(2)}] \\
}
$$
Now $\eta_n(\dd) = i^*(\eta_\infty(\dd))$ by Definition \ref{def:eta2} and $[\C P^\infty, \coker J_{(2)}] \cong 0$ by Theorem \ref{thm:Feshbach}, so the corollary follows from the commutativity of the diagram.
\end{proof}

\begin{theorem} \label{thm:bordism_over_CP4}
Let $X_4(\dd)$ and $X_4(\dd')$ be complete intersections with $SD_4(\dd) = SD_4(\dd')$ and odd total degree.
Then $\eta_4(\dd) = \eta_4(\dd') \in [\C P^4, (\qsnmo)_d]$.
\end{theorem} 

The Sullivan Conjecture for $n = 4$ and odd total degree follows directly
from Theorems \ref{thm:SC_via_cni} and \ref{thm:bordism_over_CP4}.

\begin{theorem} \label{thm:SC_d=odd}
Let $X_4(\dd)$ and $X_4(\dd')$ be complete intersections with $SD_4(\dd) = SD_4(\dd')$ and
odd total degree.
Then $X_4(\dd)$ is diffeomorphic to $X_4(\dd')$. \hfill \qed
\end{theorem}

\begin{proof}[Proof of Theorem \ref{thm:bordism_over_CP4}]
By Lemma \ref{lem:2local}, it is enough to prove that
$\eta_4(\dd)_{(2)} = \eta_4(\dd')_{(2)}$.
Since the map $\chi \colon ((\qsnmo)_d)_{(2)} \to (G/O)_{(2)}$ is a homotopy equivalence, 
it suffices to show that
$\chi_*(\eta_4(\dd)_{(2)}) = \chi_*(\eta_4(\dd')_{(2)}) \in  [\C P^4, (G/O)_{(2)}]$.  To simplify the notation
we set
\[\hat{\eta}(\dd) := \chi_*(\eta_4(\dd)_{(2)})
\quad \text{and} \quad
\hat{\eta}(\dd') := \chi_*(\eta_4(\dd')_{(2)}).\]

Let $\mu \colon (G/O)_{(2)} \to BSO_{(2)}$ and $\alpha_{(2)} \colon (BSO)_{(2)} \to (G/O)_{(2)}$ be the projection and inclusion defined by the Sullivan splitting of $(G/O)_{(2)}$ in \eqref{eq:SS} respectively, so that 
$\phi = \mu \times \pi \colon (G/O)_{(2)} \to (BSO)_{(2)} \times \coker J_{(2)}$,
and consider the bijection
\[ \mu_* \times \pi_* \colon 
[\C P^4, (G/O)_{(2)}] \equiv [\C P^4, (BSO)_{(2)}] \times [\C P^4, \coker J_{(2)}]. \]
Corollary \ref{cor:mod_p_normal_inv} states that $\pi_*(\hat{\eta}(\dd)) = \pi_*(\hat{\eta}(\dd')) = 0$, hence
\begin{compactenum}[a)]
\item it remains to show that $\mu_*(\hat{\eta}(\dd)) = \mu_*(\hat{\eta}(\dd'))$; and 
\item we have 
\[
\hat{\eta}(\dd) = (\alpha_{(2)} \circ \mu)_*(\hat{\eta}(\dd))
\quad \text{and} \quad 
\hat{\eta}(\dd') = (\alpha_{(2)} \circ \mu)_*(\hat{\eta}(\dd')) .
\]
\end{compactenum}
It follows from Lemma \ref{lem:d-equal} that $\delta_{(2)*}(\hat{\eta}(\dd)) = \delta_{(2)*}(\hat{\eta}(\dd'))$, hence
\[
(\delta_{(2)} \circ \alpha_{(2)} \circ \mu)_*(\hat{\eta}(\dd)) = (\delta_{(2)} \circ \alpha_{(2)} \circ \mu)_*(\hat{\eta}(\dd')).
\]
By Lemma \ref{lem:da-inj} $(\delta_{(2)} \circ \alpha_{(2)})_* \colon [\C P^4, (BSO)_{(2)}] \to [\C P^4, (BSO)_{(2)}]$ is injective, so $\mu_*(\hat{\eta}(\dd)) = \mu_*(\hat{\eta}(\dd'))$, which completes the proof.
\end{proof}

\begin{lem} \label{lem:d-equal}
Suppose that $n \not\equiv 1$ mod $4$. If $X_n(\dd)$ and $X_n(\dd')$ are complete intersections such that $SD_n(\dd) = SD_n(\dd')$, then $(\delta_{(2)} \circ \chi)_*(\eta_n(\dd)_{(2)}) = (\delta_{(2)} \circ \chi)_*(\eta_n(\dd')_{(2)})$.
\end{lem}

\begin{proof}
By the assumption $SD_n(\dd) = SD_n(\dd')$, the complete intersections $X_n(\dd)$ and $X_n(\dd')$, and hence their normal bundles $\nu_{X_n(\dd)}$ and $\nu_{X_n(\dd')}$, have the same Pontryagin classes (regarded as integers). This implies that $p_j(\xi_n(\dd)) = p_j(\xi_n(\dd'))$ for $2j \leq n$ (see Section \ref{ss:SD}). By \cite[Theorem 3.9]{Sa}, if $n \not\equiv 1$ mod $4$, then $[\C P^n, BSO] \cong \Z^{\lfloor n/2 \rfloor}$, detected by the total Pontryagin class. Therefore $\xi_n(\dd) \big| _{{\C}P^n} \cong \xi_n(\dd') \big| _{{\C}P^n}$, and hence $(n{+}1)\gamma \oplus \xi_n(\dd) \big| _{{\C}P^n} \cong (n{+}1)\gamma \oplus \xi_n(\dd') \big| _{{\C}P^n}$.

Since $\delta_d$ classifies taking the formal difference of the source and target vector bundles of a fibrewise degree-$d$ map (see Section \ref{ss:action}), $(\delta_d)_*(\eta_n(\dd)) \in [\C P^n, BSO]$ is the classifying map of $-k\gamma \oplus \gamma^{d_1} \oplus \ldots \oplus \gamma^{d_k} \big| _{{\C}P^n} \cong (n{+}1)\gamma \oplus \xi_n(\dd) \big| _{{\C}P^n}$. Therefore we have $(\delta_d)_*(\eta_n(\dd)) = (\delta_d)_*(\eta_n(\dd'))$. 

By localising at $2$ we get that $(\delta_d)_{(2)*}(\eta_n(\dd)_{(2)}) = (\delta_d)_{(2)*}(\eta_n(\dd')_{(2)})$. We saw that $\chi$ satisfies $\delta_{(2)} \circ \chi = (\delta_d)_{(2)}$ (see Proposition \ref{prop:B-M}), so this means that $(\delta_{(2)} \circ \chi)_*(\eta_n(\dd)_{(2)}) = (\delta_{(2)} \circ \chi)_*(\eta_n(\dd')_{(2)})$.
\end{proof}

\begin{lem} \label{lem:da-inj}
Suppose that $n \not\equiv 1$ mod $4$. Then the map $(\delta_{(2)} \circ \alpha_{(2)})_* \colon [\C P^n, (BSO)_{(2)}] \to [\C P^n, (BSO)_{(2)}]$ is injective.
\end{lem}

\begin{proof}
The proof of \cite[Theorem 5.18]{M-M} shows that there is a commutative diagram,
\[ \xymatrix{ 
& &
(G/O)_{(2)} \ar[d]^{\delta_{(2)}} \\
(BSO)_{(2)} \ar[urr]^{\alpha_{(2)}} \ar[rr]^{\psi^3 - \Id} 
& &
(BSO)_{(2)},
}\]
where $\psi^3$ is the map induced by the third power Adams operation; see \cite[5.13 \& Theorem 5.18]{M-M}. The map $\psi^3 - \Id$ is a rational homotopy equivalence (to see this, we note that the homotopy fibre of $\psi^3 - \Id$ is connected and has finite homotopy groups by the second Sullivan splitting in \cite[Theorem 5.18]{M-M}), hence $\delta_{(2)} \circ \alpha_{(2)}$ is a rational homotopy equivalence.

Now let $j_\Q \colon (BSO)_{(2)} \to ((BSO)_{(2)})_\Q = BSO_\Q$ be the natural map to the rationalisation of $BSO$.  There is a commutative square
\[
\xymatrix{
[\C P^n, (BSO)_{(2)}] \ar[r]^-{j_{\Q*}} \ar[d]^{(\delta_{(2)} \circ \alpha_{(2)})_*} &
[\C P^n, BSO_\Q] \ar[d]^{(\delta_\Q \circ \alpha_\Q)_*} \\
[\C P^n, (BSO)_{(2)}] \ar[r]^-{j_{\Q*}} &
[\C P^n, BSO_\Q],
}
\]
where $\delta_\Q$ and $\alpha_\Q$ are the rationalisations of $\delta_{(2)}$ and $\alpha_{(2)}$ respectively. In particular, $\delta_\Q \circ \alpha_\Q$ is a homotopy equivalence, hence $(\delta_\Q \circ \alpha_\Q)_*$ is a bijection.

Since there are isomorphisms $[\C P^n, BSO_\Q] \cong [\C P^n, (BSO)_{(2)}] \otimes \Q \cong (\Z_{(2)})^{\lfloor n/2 \rfloor} \otimes \Q \cong \Q^{\lfloor n/2 \rfloor}$, it follows that $j_{\Q*}$ is injective. Therefore $(\delta_\Q \circ \alpha_\Q)_* \circ j_{\Q*} = j_{\Q*} \circ (\delta_{(2)} \circ \alpha_{(2)})_*$ is injective, which implies that $(\delta_{(2)} \circ \alpha_{(2)})_*$ is injective.
\end{proof}

\begin{remark} \label{rem:SC_prime_to_d}
The arguments of this section can be generalised to prove the Sullivan Conjecture
``prime to the total degree''.  We plan to take this up in future work. 
\end{remark}


\section{Appendix: Extensions and Toda brackets} \label{s:appendix}
Recall that $\SS^0$ denotes the sphere spectrum,
and that the $i^{\text{th}}$ stable stem, $\pi_i(\SS^0)$ is denoted $\pi^s_i$.
The $k$-fold suspension of $\SS^0$ is denoted by
$\SS^k$, and that if $f \colon \SS^k \to \SS^0$ is a map, then $C_f$ denotes the cofibre of $f$.
The aim of this Appendix is to prove Lemma \ref{lem:extension_and_Toda},
which concerns the role of Toda brackets in computing extensions
for homotopy groups of $C_f$.
Lemma \ref{lem:extension_and_Toda} is presumably well-known, but we did
not find a proof for it in the literature so far.

The stable homotopy groups of $C_f$
lie in the following fragment of the long exact Puppe sequence:
\[ \ldots \to  \pi^s_{j-k} \xra{f_*} \pi^s_{j} \xra{i_*} \pi_{j}(C_f) 
\xra{c_*} \pi^s_{j-k-1} \to \ldots \]
Here $f_*, i_*$ and $c_*$ are respectively the homomorphisms induced by composition with $f$, the inclusion 
$i \colon \SS^0 \subset C_f$, the collapse map $c \colon C_f \to \SS^{k+1}$.
We shall be interested in describing the extension
\begin{equation} \label{eq:extension}
0 \to \im(i_*) \to \pi_{j}(C_f) \to \im(c_*) \to 0.
\end{equation}
To do this we take 
an element $g \in \pi^s_{j-k-1}$ of order $a$
for some positive integer $a$, which lifts to $\bar g \in \pi_{j}(C_f)$. 
Then $a \bar g \in \im(i_*) \cong \coker(f_*)$.
The element $a \bar g \in \pi^s_j$ will of course depend on the choice of $\bar g$
in general.

To describe $a \bar g$ we consider the sequence of maps
$$ \SS^{j-1} \xra{~f~} \SS^{j-k-1} \xra{~g~} \SS^0 \xra{~a~} \SS^0.$$
Since $g \circ f$ and $a \circ g$ are both null-homotopic, the Toda bracket
$$ \an{a, g, f} \subseteq \pi^s_j $$
is defined.  Representatives for the elements of $\an{a, g, f}$ are defined as 
unions  
$$ (a \circ H_1) \cup (C(f) \circ H_2)  \colon C(\SS^{j-1}) \cup C(\SS^{j-1}) \to \SS^0,$$
where $H_1$ is a null-homotopy of
$g \circ f$, $H_2$ is a null-homotopy of $a \circ g$ and $C(-)$ denotes
the cone of a spectrum or a map.
The indeterminacy of $\an{a, g, f}$ arises from the choice of null-homotopies
$H_1$ and $H_2$ and is given by
$$ I(\an{a, g, f}) = f_*(\pi^s_{j-k}) + a\pi^s_j \subseteq \pi^s_j.$$
We now relate the restriction of the extension \eqref{eq:extension} to the 
cyclic subgroup $\an g \subset \pi^s_{j-k-1}$ generated by $g$
to the Toda bracket $\an{a, g, f}$.
 
\begin{lemma} \label{lem:extension_and_Toda}
Suppose that $g \in \pi^s_{j-k-1}$ has order $a$ and 
that $\bar g \colon \SS^j\to C_f$ is a map such that $c \circ \bar g = g$.
Then 
$$ a \bar g \in i_*( \an{a, g, f}) \subset \pi_j(C_f).$$
In particular, the extension
$$ 0 \to \im(i_*) \to (c_*)^{-1}(\an{g}) \to \an{g} \to 0 $$
is trivial if and only if $0 \in \an{a, g, f}$.
\end{lemma}

\begin{proof}
Given $H_1 \colon C(\SS^{j-1}) \to \SS^0$, a null-homotopy of $g \circ f \colon \SS^{j-1} \to \SS^0$,
we define a choice of $\bar g \in \pi_j(C_f)$ by
$$ \bar g = H_1 \cup C(g) \colon C(\SS^{j-1})_1 \cup C(\SS^{j-1})_2 \to C_f,$$
where the subscripts 
label two copies of $C(\SS^{j-1})$.
There is an $a$-fold fold map
$a_{C_f} \colon (C_f, \SS^0) \to (C_f, \SS^0)$, which extends $a \colon \SS^0 \to \SS^0$
and we have $a \bar g = a_{C_f} \circ \bar g$.
On the first copy of $C(\SS^{j-1})$ we have $(a_{C_f} \circ \bar g)\big| _{C(\SS^{j-1})_1} = a \circ H_1$.
On the second copy of $C(\SS^{j-1})$, the map $(a_{C_f} \circ \bar g)\big| _{C(\SS^{j-1})_2}$ defines the zero element of $\pi^s_j(C_f, \SS^0) \cong \pi^s_{j-1}$.  
It follows that $(a_{C_f} \circ \bar g)\big| _{C(\SS^{j-1})_2}$ is homotopic rel.~$\SS^{j-1}$ to $H_2 \circ C(f)$,
where $H_2 \colon C(\SS^{j-k-1}) \to \SS^0$ is a null-homotopy of $a g$.
It follows that $a \bar g = a_{C_f} \circ \bar g$ is homotopic to 
$i \circ ((a \circ H_1) \cup (H_2 \circ C(f))$ and so $a \bar g \in i_*(\an{a, g, f})$ as required.

Finally, the extension $0 \to \im(i_*) \to (c_*)^{-1}(\an{g}) \to \an{g} \to 0$ is trivial
if and only if there is $\bar g \in \pi_j(C_f)$ such that $a \bar g = 0$.  
Given such a $\bar g$, then $0 \in i_*(\an{a, g, f})$ by the previous paragraph
and so $\an{a, g, f}$ contains an element of $\ker(i_*) = f_*(\pi^s_{j-k})$.
Hence $\an{a, g, f} \cap I(\an{a, g, f}) \neq 0$ and so $0 \in \an{a, g, f}$.
Conversely, $0 \in \an{a, g, f}$ if and only if $\an{a, g, f} = f_*(\pi^s_{j-k}) + a \pi^s_j$
and then $a \bar g \in i_*(\an{a, g, f}) = ai_*(\pi^s_j)$.
Hence we can modify our choice of $\bar g$ to achieve $a \bar g = 0$.
\end{proof}


\begin{thebibliography}{999}
\bibitem[A]{A} J. F. Adams, {\em On the groups $J(X)$. IV}, Topology {\bf 5} (1966), 21--71.

\bibitem[Ba]{Ba} D.~Baraglia, {\em The $\alpha$-invariant of complete intersections}. Available at \href{https://arxiv.org/abs/2002.06750}{https://arxiv.org/abs/2002.06750}

\bibitem[Br1]{Br1} W.~Browder, {\em Surgery on simply-connected manifolds}, Ergebnisse der Mathematik und ihrer Grenzgebiete, Band 65. Springer-Verlag, New York-Heidelberg 1972. 

\bibitem[Br2]{Br2} W.~Browder, \emph{Complete intersections and the {K}ervaire invariant}, Algebraic topology, {A}arhus 1978 ({P}roc. {S}ympos., {U}niv. {A}arhus, {A}arhus, 1978), Lecture Notes in Math., vol. 763, Springer, Berlin, 1979, pp.~88--108.

\bibitem[Br3]{Br3} E.~H. Brown, Jr., \emph{Cohomology theories}, Ann. of Math. (2) \textbf{75}  (1962), 467--484. 

\bibitem[B-M]{B-M} G.~Brumfiel and I.~Madsen, {\em Evaluation of the transfer and the universal surgery classes}, Invent.\ Math.\ {\bf 32} (1976), 133--169. 

\bibitem[D]{D} A.~Dimca, \emph{Singularities and topology of hypersurfaces}, Universitext, Springer-Verlag, New York, 1992.

\bibitem[F-K]{F-K} F.\ Fang and S.\ Klaus, {\em Topological Classification of $4$-Dimensional Complete Intersections}, Manuscripta Math.\ {\bf 90} (1996), 139--147. 

\bibitem[F-W]{F-W} F.~Fang and J.~Wang, \emph{Homeomorphism classification of complex projective complete intersections of dimensions 5, 6 and 7}, Math. Z. \textbf{266} (2010), 719--746.

\bibitem[Fe1]{Fe1} M. Feshbach, {\em Essential maps exist from $BU$ to $\mathrm{Coker}J$}, Proc.\ Amer.\ Math.\ Soc.\ {\bf 97} (1986), 539--545.

\bibitem[Fe2]{Fe2} M. Feshbach, {\em The Segal conjecture for compact Lie groups}, Topology {\bf 26} (1987), 1--20.

\bibitem[Fr]{Fr} M. H. Freedman, {\em The topology of four-dimensional manifolds}, J.~Diff.~Geom.~{\bf 17} (1982), 357--453.

\bibitem[G]{G} V. Giambalvo, 
{\em On $\langle 8 \rangle$-cobordism}, Illinois J.\ Math.\ {\bf 15} (1971), 533--541.

\bibitem[Ha]{Ha} A. Hatcher, 
{\em Algebraic Topology}, 
Cambridge University Press, 2002.

\bibitem[Hi]{Hi} P. J. Hilton,
{\em On the homotopy groups of the union of spheres},
J.\ Lond.\ Math.\ Soc.\ {\bf 30} (1955), 154--172.

\bibitem[H-M]{H-M} I.~Hambleton and I.~Madsen, {\em Local surgery obstructions and space forms}, Math.\ Z.\ {\bf 193} (1986), 191--214.

\bibitem[J]{J} P. E. Jupp, {\em Classification of certain 6-manifolds}, Proc.\ Cambridge Philos.\ Soc.\ {\bf 73} (1973), 293--300. 

\bibitem[K-M]{K-M} M.~A. Kervaire and J.~W. Milnor, \emph{Groups of homotopy spheres. {I}}, Ann.\ of Math.\ \textbf{77} (1963), 504--537.  

\bibitem[Ka]{Ka} R. Kasilingam, {\em Classification of smooth structures on a homotopy complex projective space}, Proc.\ Indian Acad.\ Sci.\ Math.\ Sci.\ {\bf 126} (2016), 277--281.

\bibitem[Kr]{Kr} M.~Kreck, {\em Surgery and Duality}, Ann.\ of Math.\ {\bf 149} (1999), 707--754.

\bibitem[L-W]{L-W} A.~S. Libgober and J.~W. Wood, \emph{Differentiable structures on complete intersections. {I}}, Topology \textbf{21} (1982), 469--482.

\bibitem[M-M]{M-M} I. Madsen and J. Milgram, {\em The classifying spaces for surgery and cobordism of manifolds}, Ann.\ of Math.\ Stud., No.\ 92, Princeton University Press, Princeton, N.\ J.; University of Tokyo Press, Tokyo, 1979.

\bibitem[M-S]{M-S} J.~W. Milnor and J.~D. Stasheff, \emph{Characteristic classes}, Ann.\ of Math.\ Stud., No.\ 76, Princeton University Press, Princeton, N.\ J.; University of Tokyo Press, Tokyo, 1974.

\bibitem[M]{M} J. Milnor, {\em On the Whitehead homomorphism $J$}, Bull.\ Amer.\ Math.\ Soc.\ {\bf 64} (1958), 79--82.

\bibitem[M-T]{M-T} R. E. Mosher and M.C. Tangora, 
{\em Cohomology operations and applications in homotopy theory}, Harper \& Row, New York-London, 1968

\bibitem[N]{N} Cs.\ Nagy, {\em The classification of $8$-dimensional $E$-manifolds}, PhD thesis, University of Melbourne 2021. Available at \href{http://hdl.handle.net/11343/279361}{http://hdl.handle.net/11343/279361}

\bibitem[R]{R} D.~Ravenel, {\em The Segal conjecture for cyclic groups and its consequences}, Amer.\ J.\ Math.\ {\bf 106} (1984), 415--446.

\bibitem[Sa]{Sa} B.~J. Sanderson, \emph{Immersions and embeddings of projective spaces}, Proc. London Math. Soc.\ \textbf{14} (1964), 137--153.

\bibitem[Sm]{Sm} S.~Smale, \emph{On the structure of manifolds}, Amer. J. Math. \textbf{84} (1962), 387--399.

\bibitem[Su]{Su} D. P. Sullivan, {\em Geometric topology: localization, periodicity and Galois symmetry (The 1970 MIT notes).} Edited and with a preface by Andrew Ranicki. $K$-Monographs in Mathematics, 8. Springer, Dordrecht, 2005.

\bibitem[T]{T} H. Toda, {\em Composition methods in homotopy groups of spheres}, Princeton University Press, Princeton, N.\ J.\ 1962.

\bibitem[Wa1]{Wa1} C.~T.~C.~Wall, {\em Classification problems in differential topology. V. On certain $6$-manifolds}, Invent.\ Math.\ {\bf 1} (1966), 355-374.

\bibitem[Wa2]{Wa2} C.~T.~C.~Wall, {\em Surgery on compact manifolds}, Second edition. Edited and with a foreword by A. A. Ranicki. Mathematical Surveys and Monographs, {\bf 69}. American Mathematical Society, Providence, RI, 1999.

\bibitem[Wh]{Wh} G.~W.~Whitehead, 
{\em Elements of homotopy theory}, 
Graduate Texts in Mathematics, Springer-Verlag 1978.

\end{thebibliography}
\end{document}